\documentclass{amsart}
\usepackage{amssymb,amsthm,accents}
\usepackage{nicefrac}
\usepackage{enumerate}

\usepackage[utf8]{inputenc}
\usepackage[T1]{fontenc}
\usepackage{subfigure}
\usepackage{ifpdf}

\ifpdf
   \usepackage[pdftex]{hyperref}
   \usepackage[pdftex]{graphicx}
   \newcommand{\myext}{pdf}
\else
   \usepackage[dvips]{hyperref}
   \usepackage[dvips]{graphicx}
   \newcommand{\myext}{eps}
\fi

\newenvironment{dscription}
{\begin{list}{$\bullet$}{%
    \setlength{\leftmargin}{0cm}%
    \setlength{\labelwidth}{\leftmargin}%
    
}{}}{\end{list}}

\newcommand{\NULL}{\mathcal N}
\newcommand{\MEAGER}{\mathcal M}
\DeclareMathOperator{\val}{val}
\DeclareMathOperator{\cov}{cov}
\DeclareMathOperator{\add}{add}
\DeclareMathOperator{\cof}{cof}
\DeclareMathOperator{\non}{non}
\newcommand{\prunedtree}{pruned-$\nicefrac12$ }
\newcommand{\plaintext}[1]{\textup{\textrm{#1}}}
\newcommand{\tttext}[1]{\textup{\textrm{\ttfamily #1}}}
\newcommand{\yy}{\n y} 
\newcommand{\ww}{{\tt w}}
\newcommand{\trklgth}[2]{\plaintext{trklg}^{#1}(#2)}
\newcommand{\al}{\aleph}
\newcommand{\epsint}{\varepsilon^{\cap}}
\newcommand{\slu}{{\bf u}}
\newcommand{\slv}{{\bf v}}
\newcommand{\slvdec}{{\bf v}^\plaintext{dec}}
\newcommand{\ca}{\mathfrak a}
\newcommand{\cb}{\mathfrak b}
\newcommand{\cc}{\mathfrak c}
\newcommand{\cs}{\mathfrak s}
\newcommand{\cd}{\mathfrak d}
\newcommand{\cK}{\textsf{K}}
\DeclareMathOperator{\glue}{glue}
\DeclareMathOperator{\half}{half}
\DeclareMathOperator{\POSS}{POSS}
\DeclareMathOperator{\poss}{poss}
\newcommand{\mdn}{m^\plaintext{dn}}
\newcommand{\mup}{m^\plaintext{up}}
\DeclareMathOperator{\nor}{nor}
\newcommand{\splitnor}{\nor_\plaintext{split}}
\newcommand{\lsnor}{\nor_\plaintext{limsup}}
\newcommand{\linor}{\nor_\plaintext{liminf}}
\newcommand{\sacksnor}{\nor_\plaintext{Sacks}}
\newcommand{\widthnorm}{\nor_\plaintext{width}}
\newcommand{\gol}{\plaintext{lognor}}
\newcommand{\norint}{\nor^{\cap}}
\newcommand{\norhalf}{\nor^{\div}}
\DeclareMathOperator{\sublevels}{sblvls}
\newcommand{\maxwidth}[1]{#1}
\DeclareMathOperator{\emptymaxwidth}{\ell}
\DeclareMathOperator{\maxposs}{maxposs}
\DeclareMathOperator{\supp}{supp}
\newcommand{\suppord}{\mathord{\supp}}

\newcommand{\QQQ}{{\mathbb Q}}

\numberwithin{equation}{subsection}

\newtheorem{thm}[equation]{Theorem}
\newtheorem{lem}[equation]{Lemma}
\newtheorem{cor}[equation]{Corollary}
\newtheorem{fact}[equation]{Fact}

\theoremstyle{definition}
\newtheorem*{construction}{Basic Construction}
\newtheorem*{defin}{Definition}
\newtheorem{definition}[equation]{Definition}
\newtheorem{defnandlem}[equation]{Definition and Lemma}
\newtheorem{remark}[equation]{Remark}
\newtheorem{example}[equation]{Example}
\newtheorem{notn}[equation]{Notation}
\newtheorem{assumption}[equation]{Assumption}

\newtheorem{nb}[equation]{Note}   

\newcommand{\n}[1]{\underaccent{\tilde}{#1}}

\newcommand{\typels}{\tttext{ls}}
\newcommand{\typeli}{\tttext{li}}
\newcommand{\typesk}{\tttext{sk}}
\newcommand{\typecn}{\tttext{cn}}
\newcommand{\typenn}{\tttext{nn}}
\newcommand{\typenm}{\tttext{nm}}
\newcommand{\typenonsk}{{\plaintext{non-}\tttext{sk}}}
\newcommand{\Xils}{\Xi_\typels}
\newcommand{\Xinm}{\Xi_\typenm}
\newcommand{\Xinn}{\Xi_\typenn}
\newcommand{\Xicn}{\Xi_\typecn}
\newcommand{\Xili}{\Xi_\typeli}
\newcommand{\Xisk}{\Xi_\typesk}
\newcommand{\Xinonsk}{\Xi_{\typenonsk}}

\newcommand{\proofclaim}[2]{
\begin{equation}\label{#1}
  \parbox{0.8\textwidth}{#2}
\end{equation}}
\newcommand{\proofclaimnl}[1]{
\begin{equation*}
  \parbox{0.8\textwidth}{#1}
\end{equation*}}

\begin{document}
\subjclass[2010]{03E17;03E35;03E40}
\date{\today}

\title[Five cardinal characteristics]{Creature forcing and\\ five cardinal characteristics in Cicho\'{n}'s diagram}
\author[A.~Fischer]{Arthur Fischer}
\address{Kurt G\"odel Research Center for Mathematical Logic\\
Universit\"at Wien\\
W\"ahringer Stra\ss e 25\\
1090 Wien, Austria}
\email{arthur.james.fischer@univie.ac.at}
\author[M.~Goldstern]{Martin Goldstern}
\address{Institut f\"ur Diskrete Mathematik und Geometrie\\
Technische Universit\"at Wien\\
Wiedner Hauptstra{\ss}e 8--10/104\\
1040 Wien, Austria}
\email{martin.goldstern@tuwien.ac.at}
\urladdr{http://www.tuwien.ac.at/goldstern/}
\author[J.~Kellner]{Jakob Kellner}
\address{Institut f\"ur Diskrete Mathematik und Geometrie\\
Technische Universit\"at Wien\\
Wiedner Hauptstra{\ss}e 8--10/104\\
1040 Wien, Austria}
\email{kellner@fsmat.at}
\urladdr{http://www.logic.univie.ac.at/$\sim$kellner/}
\author[S.~Shelah]{Saharon Shelah}
\address{Einstein Institute of Mathematics\\
Edmond J. Safra Campus, Givat Ram\\
The Hebrew University of Jerusalem\\
Jerusalem, 91904, Israel\\
and
Department of Mathematics\\
Rutgers University\\
New Brunswick, NJ 08854, USA}
\email{shelah@math.huji.ac.il}
\urladdr{http://shelah.logic.at/}
\thanks{
We gratefully acknowledge the following partial support: US National Science
Foundation Grant No. 0600940 (all authors); US-Israel Binational Science
Foundation grant 2006108 (fourth author); FWF Austrian Science Fund: 
P23875-N13 and I1272-N25 (first and third author); P24725-N25 (second author).
This is publication 1044 of the fourth author.}

\dedicatory{Dedicated to the memory of James E. Baumgartner (1943--2011)}
\begin{abstract}
   We use a (countable support) creature construction
   to show that consistently
   \[
   \mathfrak d=\aleph_1=
   \cov(\NULL)
   <
   \non(\MEAGER)
   <
   \non(\NULL)
   <
   \cof(\NULL)
   <
   2^{\aleph_0}.
   \]
   The same method shows the consistency of
   \[
   \mathfrak d=\aleph_1=
   \cov(\NULL)
   <
   \non(\NULL)
   <
   \non(\MEAGER)
   <
   \cof(\NULL)
   <
   2^{\aleph_0}.
   \]
\end{abstract}

\maketitle

\section{Introduction}

\subsection{The result and its history}
Let $\NULL$ denote the ideal of Lebesgue null sets, and $\MEAGER$ the ideal of meager sets.
We prove (see Theorem~\ref{thm:main}) that  consistently, several cardinal characteristics of Cicho\'n's Diagram (see Figure~\ref{fig:cichon}) are (simultaneously) different: 
\[
   \aleph_1=\cov(\NULL)=\mathfrak d<\non(\MEAGER)<\non(\NULL)<\cof(\NULL)<2^{\aleph_0}.
\]
\begin{figure}[h]
\centering
\scalebox{0.8}{\includegraphics{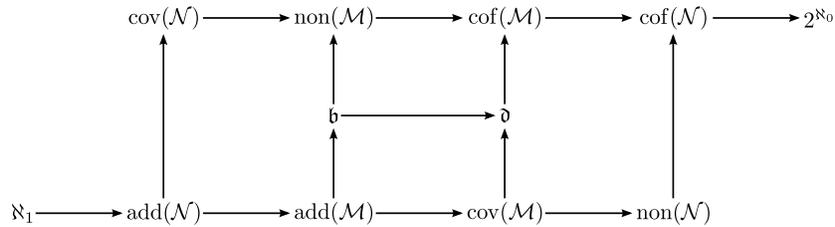}}
\caption{Cicho\'{n}'s diagram. An arrow between $\mathfrak x$ and $\mathfrak y$ indicates that 
$\mathfrak x\le \mathfrak y$. Moreover, 
$\max(\mathfrak d,\non(\MEAGER))=\cof(\MEAGER)$ and $\min(\mathfrak b,\cov(\MEAGER))=\add(\MEAGER)$.}
\label{fig:cichon}
\end{figure}
Since our model will satisfy $\mathfrak d=\aleph_1$,  will also have 
$ \non(\MEAGER) = \cof(\MEAGER) $.  The desired values of the cardinals 
$ \non(\MEAGER), \non(\NULL),  \cof(\NULL) , 2^{\aleph_0}$ can be chosen quite 
arbitrarily, as long as they are ordered as indicated and each satisfies
$\kappa^{\aleph_0}=\kappa$.

A (by now) classical series of
theorems~\cite{MR719666,MR1233917,MR781072,MR1071305,MR1022984,Krawczyk83,MR613787,MR735576,MR697963,MR1613600,MR1623206}
proves these (in)equalities in ZFC and shows that they are the only ones
provable. More precisely, all assignments of the values $\aleph_1$ and $\aleph_2$
to the characteristics in Cicho\'n's Diagram are consistent, provided they do
not contradict the above (in)equalities.  (A complete proof can be found
in~\cite[chapter 7]{MR1350295}.)

This does not answer the question whether three (or more) characteristics can
be made simultaneously different. The general expectation is that this should
always be possible, but may require quite complicated forcing methods.
We cannot use the two best understood methods, countable support iterations of
proper
forcings (as it forces $2^{\aleph_0}\le \aleph_2$) and, at least for the ``right hand side'' of the diagram, we cannot use finite support
iterations of ccc forcings in the straightforward way (as it adds
lots of Cohen reals, and thus increases $\cov(\MEAGER)$ to $2^{\aleph_0}$).

There are ways to overcome this obstacle. One way would be to first
increase the continuum in a ``long'' finite support iteration, resulting in
$\cov(\MEAGER)=2^{\aleph_0}$, and then ``collapsing'' $\cov(\MEAGER)$ in
another ``short'' finite support iteration.  In a much more sophisticated
version of this idea, Mej{\'{\i}}a~\cite{MR3047455} recently constructed
several models with many simultaneously different cardinal characteristics in
Cicho\'n's Diagram (building on work of Brendle~\cite{MR1129144},
Blass-Shelah~\cite{MR1005010} and Brendle-Fischer~\cite{MR2791343}).

We take a different approach, completely avoiding finite support, and use
something in between a countable and finite support
product (or, a form of iteration with very ``restricted memory'').

This construction avoids Cohen reals, it is in fact $\omega^\omega$-bounding,
resulting in $\mathfrak d=\aleph_1$. This way we get an independence result 
``orthogonal'' to the ccc/finite-support results of Mej{\'{\i}}a.

The fact that our construction is $\omega^\omega$-bounding is not incidental,
but rather a necessary consequence of the two features which, in our
construction, are needed to guarantee properness: a ``compact'' or ``finite
splitting'' version of pure decision, and fusion (which together give a strong
version of Baumgartner's Axiom A and in particular properness and $\omega^\omega$-bounding).

We think that our construction can be used for various other independence
results with $\mathfrak d=\aleph_1$, but the construction would require
considerable remodeling if we want to use it for similar results with
$\mathfrak d>\aleph_1$,  even more so for  $\mathfrak b>\aleph_1$.

\subsection{A very informal overview of the construction}

The obvious attempt to prove the theorem would be to 
find a forcing  for each cardinal characteristic $\mathfrak x$ that increases 
$\mathfrak x$ but leaves the other characteristics unchanged.
More specifically, find the following forcing notions.
\begin{itemize}
\item $\QQQ_\typenm$, adding a new meager set which will contain all old reals. \\
Adding many such sets will tend to make $\non(\MEAGER)$ large.
\item $\QQQ_\typenn$, adding a new measure zero set which will contain all old reals. \\
Adding many such sets will tend to make $\non(\NULL)$ large.
\item $\QQQ_\typecn$, adding a new measure zero set which is not contained in any old measure zero set. \\
Adding many such sets will tend to make $\cof(\NULL)$ large.
\item $\QQQ_\typesk$, adding a kind of Sacks real,
in the sense that the generic real does not change any other cardinal characteristic; in particular, every new real is bounded by an old real, is contained in an old measure zero set, etc.
\\
Adding many such reals will tend to make the continuum large. 
\end{itemize}

For each $t\in\{\typenm,\typenn,\typecn,\typesk\}$, our $\QQQ_t$ will be a
finitely splitting tree forcing; $\QQQ_\typenm$ will be ``lim-inf'' (think of a
tree forcing where we require large splitting at every node, not just
infinitely many along every branch; i.e., more like Laver or Cohen
than Miller or Sacks; however
note that in contrast to Laver all our forcings are finitely splitting); the
other ones will be ``lim-sup'' (think of forcings like Sacks or Silver).

We then fix for each $t$
a cardinal $\kappa_t$, and take some kind of product (or, iteration) of
$\kappa_t$ many copies of $\QQQ_t$, and hope for the best.
Here we arrive at the obvious problem: which product or iteration will work?
As mentioned above, neither a finite support iteration\footnote{To avoid giving a
wrong impression, our specific forcings $\QQQ_t$ will not be ccc, so a finite
support iteration would not work anyway.}
nor a countable support iteration will work, and it is
not clear why a product will not collapse the continuum. So we will introduce
a modification of the product construction. 

\medskip

The paper is divided into two parts. In part 1 we describe the ``general''
forcing construction (let us call it the ``framework''), in part 2, the
``application'', we use the framework to construct a specific forcing that
proves the main theorem.

{\bf Part 1:}
In Sections~\ref{sec:defQ1}--\ref{sec:proper} we present the ``framework''.
Starting with building blocks (so-called ``subatoms''), we define the forcing
$\QQQ$.  This is an instance of creature forcing. (The standard reference for
creature forcing is Ros{\l}anowski and Shelah~\cite{MR1613600}, but our
presentation will be self-contained.  Our framework is a continuation
of~\cite{MR2864397,MR2499421}, where the central requirement to get properness
was ``decisiveness''.  In this paper, decisiveness does not appear explicitly,
but is implicit in the way that the subatoms are combined to form so-called atoms.)

We fix a set $\Xi$ of indices.
(For the application, we will partition $\Xi$
into sets $\Xi_t$ of size $\kappa_t$ for $t\in \{\typenm,\typenn,\typecn,\typesk\}$ as above.) The forcing  $\QQQ$ will ``live'' on the product
$\Xi\times\omega$, i.e., a condition $p\in \QQQ$ will contain for certain
$(\xi,n)$ a ``creature'' $p(\xi,n)$, a finite object that gives some
information about the generic filter. 

More specifically, there is a countable
subset $\supp(p)\subseteq \Xi$, and for each $\xi\in\supp(p)$ the condition
up to some level $n_0(\xi)$
consists of a so-called trunk (where a
finite initial segment of the generic real $\yy_\xi$ is already completely determined), 
and
for all $n>n_0(\xi)$ there is a creature $p(\xi,n)$,
an element of a fixed finite set $\cK_{\xi,n}$,
which gives several (finitely many)
possibilities for the corresponding segment of the generic real $\yy_\xi$.
We assign a ``norm'' to the creature, a real number that measures the 
``number of possibilities'' (or, the amount of freedom that the creature
leaves for the generic). More possibilities means larger norm.

Moreover, for each $m$ there are only finitely many $\xi$ with $n_0(\xi)\le m$
(i.e., at each level $m$ there live only finitely many creatures of $p$).  We can
then set the norm of~$p$ at~$m$ to be the minimum of the norms of $p(\xi,n)$
over all $\xi$ ``active'' at level $m$.

A requirement for a $p$ to be a valid condition in 
$\QQQ$ is that the norms at level $m$ diverge to infinity for $m\to\infty$
(i.e., the lim-inf of the norms is infinite).

So far, $\QQQ$ seems to be a lim-inf forcing, but recall that we
want to use lim-inf as well as lim-sup.

So let us redefine $\QQQ$.
We will ``cheat'' by allowing ``gluing''.
We declare a subset of $\Xi$ to be the set $\Xils$ of ``lim-sup
indices'' (in the application this will be $\Xinn\cup\Xicn$).
Forget the ``norm of $p$ at level $m$'' and the lim-inf condition above.
Instead, we partition the set of levels $\omega$
into finite intervals
$\omega=I_0\cup I_1\cup\dots$ 
(this partition depends on the condition and can be coarsened
when we go to a stronger condition). For such an interval $I$,
we declare all creatures whose  levels belong to $I$ 
to constitute a ``compound creature'' with a ``compound norm'', intuitively computed as follows:
\begin{itemize}
\item for each $\xi\in \Xils$ we set $\nor(p,I,\xi)$
to be the maximum of the norms of $p(\xi,m)$ with $m\in I$;
\item for other $\xi$ we take the minimum rather than the maximum; and
\item we set $\nor(p,I)$ to be the minimum of $\nor(p,I,\xi)$
for all (finitely many) $\xi$ active at some level in $I$.
\end{itemize}
The new lim-inf condition is that $\nor(p,I_k)$
diverges to infinity with $k\to\infty$.

While this may give some basic idea about the construction, things really are
more complicated. We will require the well-known ``halving'' property of
creature forcing (to prove Axiom~A). Moreover, the Sacks part, i.e.,
$\QQQ_\typesk$ on the indices $\Xi_\typesk\subset \Xi$, does not fit
well into the framework as presented above and requires special treatment.
This will not be very complicated mathematically but will unfortunately make
our notation much more awkward and unpleasant. 

A central requirement on our building blocks (subatoms) will be another
well-known property of creature-forcing: ``bigness''. This is a kind of Ramsey property
connected to the statement that creatures at a level $m$ are ``much bigger''
than everything that ``happened below $m$''.

Using these requirements, we will show the following.
\begin{itemize}
  \item (Assuming CH in $V$) $\QQQ$ is $\al_2$-cc. (Accomplished via a standard $\Delta$-system argument.)
  \item We say that $p$ ``essentially decides'' a name $\n\tau$ of an ordinal
    if there is a level $m$ such that whenever we increase the trunk of $p$ up to $m$
    (for this there are only finitely many possibilities), we know the value of $\n \tau$.
    In other words, knowing the generic up to $m$ (on some finite set of indices),
    we also know the value of $\n\tau$.
  \item Pure decision and fusion.
    Given a name $\n\tau$ of an ordinal and a condition $p$,
    we can strengthen $p$ to a condition $q$ essentially deciding $\n\tau$.
    Moreover, we can do this in such a way that $p$ and 
    $q$ agree below a given level $h$
    and the norms above this level do not drop below a given bound.
    (This is called ``pure decision''.)

    This in turn implies ``fusion'' in that we can iterate this strengthening for infinitely many names $\n\tau_\ell$, resulting in a common extension $q_\infty$
    which essentially decides each $\n\tau_\ell$.

    (While fusion is an obvious property of the framework,
     pure decision is the central result of part~1, and will use the
     requirements on bigness and halving).
  \item
    The usual standard argument then gives continuous reading (every real is a continuous image of (countably many) generic reals),
    a strong version of Axiom A, and thus
    $\omega^\omega$-bounding  and properness.
    (Recall that we have ``finite splitting'', i.e., essentially deciding implies that there are only finitely
    many potential values.)
  \item We also get a Lipschitz variant of continuous reading, ``rapid reading'', which implies that the forcing adds no random reals
(and which will be essential for many of the proofs in part~2).
\end{itemize}

{\bf Part 2:}
In Sections~\ref{sec:specificQ}--\ref{ss:cn} we define the specific forcings
$\QQQ_t$ (or rather, the building blocks, i.e., the subatoms, for these
forcings) for $t\in \{\typenm,\typenn,\typecn\}$ (the Sacks case is already
dealt with in part~1).

We prove that these subatoms satisfy the bigness requirements of the 
framework, and we prove the various parts of the main theorem.

\goodbreak
\subsection*{Annotated Contents}
\noindent{\bf Part 1:} We present a \emph{forcing framework}.
\begin{description}
\item[Section~\ref{sec:defQ1}, p.\ \pageref{sec:defQ1}]
Starting with building blocks (the so-called subatomic families,
which are black boxes that will be described later)
we describe how to build a forcing $\QQQ$.
\item[Section~\ref{sec:Qprop}, p.~\pageref{sec:Qprop}]
We give some simple properties of $\QQQ$, including the $\aleph_2$-cc.
\item[Section~\ref{sec:complete.construction}, p.~\pageref{sec:complete.construction}]
We impose additional requirements on the subatomic families,
and give an inductive construction that shows how we can choose 
suitable subatomic families so that the requirements are satisfied.
\item[Section~\ref{sec:proper}, p.~\pageref{sec:proper}]
Using the additional requirements, we show that $\QQQ$ satisfies Axiom~A, is $\omega^\omega$-bounding
and has continuous and rapid reading.  This implies $\mathfrak d=\cov(\NULL)=\aleph_1$ in the generic extension.
\end{description}
\noindent{\bf Part 2:} We give the \emph{application}.
\begin{description}
\item[Section~\ref{sec:specificQ}, p.~\pageref{sec:specificQ}]
We present the specific forcing.
There are four ``types'' $t$, $\typenm$, $\typenn$, $\typecn$, and $\typesk$,
corresponding to 
$\non(\MEAGER)$, $\non(\NULL)$, $\cof(\NULL)$ and the continuum, respectively.
The $\typenm$-part will be lim-inf, $\typenn$ and $\typecn$ lim-sup
(and $\typesk$ lim-sup as well, but treated differently).
The actual definitions of the $t$-subatoms (other than Sacks) will be given in 
Sections~\ref{ss:nm}, \ref{ss:nn}, \ref{ss:cn}.
For each type $t$ the forcing will contain  a ``$t$-part'' of size $\kappa_t$.
\\
We formulate the main theorem: $\QQQ$ will force each invariant to be the respective
$\kappa_t$.
\\
We show that the Sacks part satisfies a Sacks property, which implies
$\cof(\NULL)\le \kappa_\typecn$ in the generic extension.
\\ 
Using the fact that 
only the $\typenm$-indices are ``lim-inf'', we show that $\non(\MEAGER)\le\kappa_{\typenm}$.
\item[Section~\ref{ss:nm}, p.~\pageref{ss:nm}] We define the 
	$\typenm$-subatoms and prove $\non(\MEAGER)\ge\kappa_{\typenm}$.
\item[Section~\ref{ss:nn}, p.~\pageref{ss:nn}] We define the 
	$\typenn$-subatoms and prove $\non(\NULL)\ge\kappa_{\typenn}$.
\item[Section~\ref{sec:counting}, p.~\pageref{sec:counting}] We mention
 some simple facts about counting, and use them to define the counting norm, $\gol$, for the $\typecn$ subatoms.
\item[Section~\ref{ss:cn}, p.~\pageref{ss:cn}] We define the 
	$\typecn$-subatoms and prove $\cof(\NULL)\ge\kappa_{\typecn}$.
   And finally, we show $\non(\NULL)\le \kappa_\typenn$.
\end{description}
\subsection{Acknowledgements} 
We are grateful to Diego Mej\'\i a for pointing out several
embarrassing oversights. 
We also thank the anonymous referee for pointing out additional errors, and making numerous helpful  suggestions for improving the text.

\goodbreak

\section{The definition of the forcing \texorpdfstring{$\QQQ$}{Q}}\label{sec:defQ1}

\subsection{Subatomic creatures}

\begin{definition}\label{def:subatoms}
Let $\POSS$ be a finite set.
A \emph{subatomic family} living on $\POSS$ consists of a finite set
$\cK$ (whose elements are called  \emph{subatomic creatures}, or \emph{subatoms}, for short), 
a quasiorder $\le$ on $\cK$ and functions $\poss$ and  $\nor$
with domain $\cK$,
satisfying the following for all $x \in \cK$:
\begin{itemize}
\item 
$\poss(x)$ is a nonempty subset of $\POSS$;
\item 
$\nor( x )$ is a nonnegative real number; and
\item
$y \le x$
implies 
$\poss(y ) \subseteq \poss( x )$.
\end{itemize}
To simplify notation, we further assume:
\begin{itemize}
\item if $|\poss(y)|=1$, then $\nor(y)<1$; and
\item for each   $x\in\cK$ and
$a\in\poss(x)$ there is a $y\le x$ with $\poss(y)=\{a\}$.
(Such a subatom will be called a \emph{singleton}.)
\end{itemize}
\end{definition}

\begin{notn}
Abusing notation, we will just write $\cK$ for the subatomic family
$(\cK,\mathord\le,\nor,\poss)$. If $y\le x$ we will also say that $y$ is ``stronger than
$x$'' or is ``a successor of $x$''.
\end{notn}

\begin{remark}\label{remark:simplesubatoms}
  Our subatomic families will also have the following properties (which might make
  the picture clearer, but will not be used in any proof).
  \begin{itemize}
    \item Each subatom $x$ is determined\footnote{The analogous statement will not be true for ``compound creatures'' (cf.~Definition~\ref{def:cc}) because of the halving parameters.}
     by $\poss(x)$ (i.e., the function
      $\poss:\cK\to 2^{\POSS}$ is injective).
      In particular $\nor(x)$ is determined
     by $\poss(x)$.
    \item $\poss(y)\subseteq \poss(x)$ implies $\nor(y)\le\nor(x)$.
    \item $y\leq x$ iff $\poss(y)\subseteq \poss(x)$.
  \end{itemize}
\end{remark}

In the usual way we often identify a natural number
$n$ with the set $\{ 0 , \ldots , n-1 \}$, and write $m \in n$ for $m < n$;
for example in the following definition.
\begin{definition}\label{def:big}
  Fix a natural number $B > 0$. We
  say that a subatom 
  $x\in\cK$ has \emph{$B$-bigness} if for each coloring
  $c:\poss(x)\to B$ there is a $y\le x$ such that $c\restriction\poss(y)$ is constant
  and $\nor(y)\ge\nor(x)-1$.\footnote{As only the number of ``colors'' is of importance, 
    we may consider the codomain of the coloring function to be any set of cardinality $B$.}
  We say that the subatomic family $\cK$ has \emph{$B$-bigness} if each $x\in \cK$ has $B$-bigness.
\end{definition}

Given a subatom $x$ in a fixed subatomic family $\cK$, we have the following facts.
\begin{itemize}
\item
  If $\nor(x) \leq 1$, then $x$ has $B$-bigness for all $B > 0$. (Any coloring $c : \poss(x) \to B$ 
  will be constant on $\poss(y)$ for any singleton $y \leq x$.)
\item
  If $\nor(x) \geq 2$, then $x$ cannot have $| \poss(x) |$-bigness. (The identity function 
  $c : \poss(x) \to \poss(x)$ is only constant on singleton sets, and any singleton subatom 
  has norm $< \nor(x) -1$.)
\item
  If $x$ has $B$-bigness, then $x$ has $B^\prime$-bigness for all $1 \leq B^\prime \leq B$.
\end{itemize}

\begin{example}
The basic example of a subatomic family with $B$-bigness is the following
``counting norm''.
For a fixed finite set $\POSS$, a subatom $x$ is a nonempty subset of $\POSS$, with 
$\poss(x):=x$, 
$y\leq x$ defined as $y\subseteq x$, and 
\[
  \nor(x):=\log_B{|x|}.
\]
We get a stronger variant of bigness if we divide the norm by $B$:
\[
  \nor'(x):=\frac{\log_B(|x|)}{B}.
\]
Then for each $F:\poss(x)\to B$ there is a $y\le x$ such that
$F\restriction \poss(y)$ is constant and $\nor'(y)\geq \nor'(x)-\nicefrac{1}{B}$.
\end{example}

\begin{remark}
The above example (in the version $\nor'$) is actually used for the 
$\non(\MEAGER)$-subatoms (cf.~\ref{def:Knm}).
The $\cof(\NULL)$-subatoms (cf.~Section~\ref{ss:cndef})
still use a counting norm, i.e.,
$\nor(x)$ only depends on the cardinality 
of $\poss(x)$, but the relation between $|\poss(x)|$ and $\nor(x)$
is more complicated.
The $\non(\NULL)$-subatoms (cf.~Section~\ref{ss:nndef})
will use a different kind of norm which does not just depend on the cardinality of $\poss(x)$,
but also on its structure.
\end{remark}

Given a subatomic family with $2$-bigness, it is straightforward to construct 
another subatomic family with arbitrary bigness by only altering the norm.
\begin{lem}\label{lem:bbigfrom2big}
If $\cK$ is a subatomic family with $2$-bigness, then given any $b \geq 1$ replacing the norm of $\cK$ with $\nor^\prime$ defined by $\nor^\prime (x) := \nicefrac{\nor(x)}{b}$ results in a subatomic family with $2^b$-bigness.
\end{lem}

\begin{proof}
Given $x \in \cK$, and a coloring $c : \poss(x) \to \mathcal{P} (b)$, use the 
$2$-bigness of the original subatomic family to inductively pick 
$x = x_0 \geq x_1 \geq \cdots \geq x_b = y$ so that for each $i < b$ we have 
$\nor(x_{i+1}) \geq \nor(x_i) - 1$ and $c_i \restriction \poss(x_{i+1})$ is constant, 
where $c_i : \poss(x_i) \to 2$ is defined by $c_i (a) = 1$ iff $i \in c(a)$.
Then $c \restriction \poss(y)$ is constant, and 
$\nor^\prime (y) = \nicefrac{\nor(y)}{b} \geq \nicefrac{(\nor(x)-b)}{b} = \nor^\prime(x) - 1$.
\end{proof}

\begin{remark}
Of course, any subatomic family $\cK$ can be made to have arbitrary bigness 
by simply ensuring that all subatoms have norm $\leq 1$.  The benefit of the method presented 
in Lemma~\ref{lem:bbigfrom2big} is that the norm of each subatom decreases proportionally to 
the logarithm of the desired bigness. As our construction depends on the existence of 
subatomic families with ``big'' bigness and also having subatoms with ``large'' norm, the above 
Lemma gives an indication of how this can be achieved.
\end{remark}

\subsection{Atomic creatures} \label{sec:atomic}

We now describe how to combine subatomic families to create so-called atoms.
Fix a natural number $J > 0$, and 
fix a parameter $\emptymaxwidth\in \omega $.
We will first define the ``measure'' of subsets of $J$ with respect to this parameter.
\begin{definition}\label{def:mudef}
  For $A\subseteq J$, we set 
\[
\mu^{\emptymaxwidth}(A):=\frac{\log_3(|A|)}{\emptymaxwidth+1}
\] (or $0$, if $A=\emptyset$).\footnote{So, technically $\mu^{\emptymaxwidth} (A)$ is defined to be $\nicefrac{\log_3 ( \max \{ |A| , 1 \} )}{\emptymaxwidth +1}$.}
\end{definition}

We will later use the following easy observation about the ``measure''.

\begin{lem} \label{lem:splitJ}
  Suppose $k \leq \emptymaxwidth$, and $A_0 , \ldots , A_k$ are subsets of $J$.
  Then there are pairwise disjoint sets $B_0 , \ldots , B_k$ such that $B_i \subseteq A_i$, 
  and $\mu^{\emptymaxwidth} ( B_i ) \geq \mu^{\emptymaxwidth} ( A_i ) - 1$ for all $i \leq k$.
\end{lem}

\begin{proof}
Note that if for some $i \leq k$ we have that $\mu^{\emptymaxwidth} (A_i) \leq 1$, then simply picking $B_i := \emptyset$ will introduce no obstructions.
We may then assume that $\mu^{\emptymaxwidth} (A_i) > 1$ (meaning that $|A_i| \geq 3^{\emptymaxwidth+1}$) for each $i \leq k$.
We now inductively construct $(k+1)$-tuples $( A^j_0 , \ldots , A^j_k )$ ($j \leq n := \nicefrac{k(k+1)}{2}$ where $A^0_i = A_i$ for each $i \leq k$, and at stage $j < n$ we handle a distinct pair $(i^0,i^1 )$ with $i^0 < i^1 \leq k$ so that 
\begin{itemize}
\item
$A^{j+1}_{i^0} \subseteq A^j_{i^0}$, $| A^{j+1}_{i^0} | \geq \nicefrac{|A^j_{i^0}|}{3}$;
\item
$A^{j+1}_{i^0} \subseteq A^j_{i^1}$, $| A^{j+1}_{i^1} | \geq \nicefrac{|A^j_{i^1}|}{3}$; and 
\item
$A^{j+1}_{i^0} \cap A^{j+1}_{i^1} = \emptyset$.
\end{itemize}
(and $A^{j+1}_i = A^j_i$ for all other $i \leq k$).  As $|A_{i^0}| \geq 3^{\ell+1}$ it follows by the induction that $|A^j_{i^0}| \geq 3$, and similarly $| A^j_{i^1} | \geq 3$, and so it is possible to partition the intersection $A^j_{i^0} \cap A^j_{i^1}$ into $Y \cup Z$ so that $| A^j_{i^0} \setminus Y | \geq \nicefrac{|A^j_{i^0}|}{3}$ and $| A^j_{i^1} \setminus Z | \geq \nicefrac{|A^j_{i^1}|}{3}$. We may then take $A^{j+1}_{i^0} := A^j_{i^0} \setminus Y$ and $A^{j+1}_{i^1} := A^j_{i^1} \setminus Z$.

After these steps, set $B_i := A^n_i$ for each $i \leq k$.
It is clear that the $B_i$ are pairwise disjoint (since if $i^0 < i^1 \leq k$ at some stage $j$ we would have handled this pair, meaning that $A^{j+1}_{i^0} \cap A^{j+1}_{i^1} = \emptyset$, but $B_{i^0} \subseteq A^{j+1}_{i^0}$ and $B_{i^1} \subseteq A^{j+1}_{i^1}$).
As each $A_i$ was modified at most $k$ times in the inductive construction it follows that $|B_i| \geq \nicefrac{|A_i|}{3^k}$, and so $\mu^{\emptymaxwidth}(B_i) = \nicefrac{\log_3(|B_i|)}{\ell+1} \geq \nicefrac{\log_3 ( \nicefrac{|A_i|}{3^k} )}{\ell+1} \geq \nicefrac{\log_3 (|A_i|) - \ell}{\ell+1} \geq \mu^\ell(A_i)-1$.
\end{proof}

Suppose now that for each $j\in J$ we have a subatomic family $\cK_j$ living on a finite set $\POSS_j$.
We can now define the atoms built from the subatoms.
\begin{definition}\label{def:atomsfromsubatoms}
\begin{itemize}
  \item An \emph{atomic creature}, or \emph{atom}, $\ca$
    consists of a sequence $(x_j)_{j\in J}$
    where $x_j$ is a $\cK_j$-subatom for all $j \in J$.
  \item The \emph{norm} of an atom $\ca=(x_j)_{j\in J}$, $\nor(\ca)$, is the 
    maximal $r$
    for which there is a set $A\subseteq J$ with $\mu^{\emptymaxwidth}(A)\ge r$ and
    $\nor(x_j) \ge r$ for all $j\in A$.
    We say that such an $A$ ``witnesses the norm'' of $\ca$.
\end{itemize}
\end{definition}

So the norm of an atom is large if there is a
``large'' subset $A$ of $J$ such that all subatoms in $A$ are ``large''.

The following easy fact will be useful later.
\begin{fact}
  Suppose $A \subseteq J$ witnesses the norm of an atom $\ca = ( x_j )_{j \in
  J}$, and let $\cb = ( y_j )_{j \in J}$ be any atom which agrees with $\ca$ on
  all indices in $A$.  Then $\nor ( \cb ) \geq \nor ( \ca )$.
In particular, if $\nor ( y_j ) \leq \nor ( x_j )$ for all $j \notin A$, then $\nor ( \cb ) = \nor ( \ca )$.
\end{fact}

\subsection{Sacks columns}

Given a (finite) tree $T$, its \emph{splitting-size}, $\splitnor(T)$, is defined as the 
maximal $\ell\in\omega$ such that 
there is a subset $S\subseteq T$ (with the induced order) which is order
isomorphic to the complete binary tree $2^ {\le \ell}$ (of height $\ell$ with
$2^\ell$ many leaves).  Equivalently, $2^{\le \ell}$ order-embeds into $T$.

Given a finite subset $I$ of $\omega$ and $F\subseteq 2^I$, we can identify 
$F$ with the tree of its restrictions $T_F=F\cup \{\eta\restriction n:\, \eta\in F,\ n\in I\}$
(a tree of partial functions from $I$ to~$2$, ordered by inclusion).
We write $\splitnor(F)$ for $\splitnor(T_F)$.

The following establishes a basic combinatorial fact about
this norm.

\begin{defnandlem}\label{def:martinsgloriousdef}
There exists a function $f$ with the following properties.
\begin{itemize}
\item
For each $j$, $n$, $c$, whenever
$(2^{f(j,n,c)})^j$ is colored with $c$ colors
there are  subsets $A_1,\ldots, A_j$ of $2^{f(j,n,c)}$ 
such that the set $A_1\times \cdots\times  A_j$ is homogeneous,
and $\splitnor(A_i)\ge n$ for all $i$.%
\footnote{As in the case of the bigness of subatoms, only the number of ``colors'' of our coloring functions is of importance. Moreover, by the definition of the splitting norm it follows that $T_1 , \ldots , T_j$ are trees each of splitting size at least $f(j,n,c)$ and $\pi : T_1 \times \cdots \times T_j \to c$ is a coloring, then there are $A_i \subseteq T_i$ ($i \leq j$) such that $\splitnor(A_i) \geq n$ for each $i$ and $\pi \restriction A_1 \times \cdots \times A_j$ is constant.}%
\textsuperscript{,}%
\footnote{If $j = 1$ this condition becomes whenever $2^{f(1,n,c)}$ is colored with $c$ colors
there is a homogeneous subset $A$ of $2^{f(1,n,c)}$ such that $\splitnor(A)\ge n$.} 
\item
$f$ is monotone in each argument.
\end{itemize}
\end{defnandlem}

\begin{proof}
We define $f(j,n,c)$ recursively on $j$ by $f(1,n,c) = n \cdot c$, and $f(j+1,n,c) = f ( 1 , n , c^{2^{j \cdot f(j,n,c)}} ) = n \cdot c^{2^{j \cdot f(j,n,c)}}$.
Note that $f(j,n,1) = n$, and clearly any coloring $\pi : ( 2^n )^j \to 1$ is constant.
We may then assume that $c > 1$ for the remainder of the proof.

We first show by induction on $c$ that $f(1,n,c)$ is as required.
Suppose that $f(1,n,c)$ works for some $c \geq 1$, and let $\pi : 2^{n \cdot (c+1)} \to c+1$ be a coloring.
For $\eta \in 2^n$, let $[\eta] := \{ \nu \in 2^{n+c \cdot n} : \eta \subseteq \nu \}$.
Note that $\splitnor([\eta]) = 2^{c \cdot n}$ for each $\eta \in 2^n$.
If there is an $\eta \in 2^n$ such that $\pi \restriction [ \eta ]$ omits one of $0 , \ldots , c$, then $\pi \restriction [\eta]$ is a coloring with at most $c$ colors, and so there must be an $A \subseteq [\eta] \subseteq 2^{n+c\cdot n}$ such that $\splitnor(A) \geq n$ and $\pi \restriction A$ is constant.

Otherwise, for each $\eta \in 2^n$ there is an $\nu_\eta \in [\eta]$ such that $\pi ( \nu_\eta ) = 0$.
It follows that $A := \{ \nu_\eta : \eta \in 2^n \}$ has splitting size $n$, and $\pi \restriction A$ is constantly $0$.

Assume that $f(j,n,c)$ satisfies the desired property for some $j \geq 1$.
Set $p := f(j,n,c)$ and $q:= c^{2^{j \cdot p}}$, so that $f(j+1,n,c) = n \cdot q = f ( 1 , n , q )$.
Suppose $\pi : ( 2^{n \cdot q} )^{j+1} \to c$ is a coloring.
Define $T := \{ \eta \in 2^{n \cdot q} : \eta \restriction [ p , n \cdot q )\text{ is constantly }0 \}$.
Since $c \geq 2$ it follows that $p < n \cdot q$, and so $\splitnor(T) = p$.
For $\eta \in 2^{n \cdot q}$ define $\pi_\eta : T^j \to c$ by $\pi_\eta ( \eta_1 , \ldots , \eta_j ) = \pi ( \eta_1 , \ldots , \eta_j , \eta )$.
Note that the mapping $\eta \mapsto \pi_\eta$ is a coloring of $2^{n \cdot q}$ by at most $c^{(2^p)^j} = q$ many colors.
By the above it follows that there is an $A_{j+1} \subseteq 2^{n \cdot q}$ and a $\pi^* : T^j \to c$ such that $\splitnor (A_{j+1}) \geq n$ and $\pi_\eta = \pi^*$ for each $\eta \in A_{j+1}$.

Then as $\pi^*$ is a coloring of $T^j$ by at most $c$ colors, and as $\splitnor (T) = p = f(j,n,c)$ by hypothesis for each $i \leq j$ there are $A_i \subseteq T \subseteq 2^{n \cdot q}$ with $\splitnor (A_i) \geq n$ (for $i \leq j$) such that $A_1 \times \cdots \times A_j$ is homogeneous for $\pi^*$.
It then follows that $A_1 \times \cdots \times A_j \times A_{j+1}$ is homogeneous for $\pi$.
\end{proof}

\begin{definition}
Suppose that $I$ is a nonempty (finite) interval in $\omega$.
By a \emph{Sacks column on $I$} we mean a nonempty $\cs \subseteq 2^I$.
We say that another Sacks column $\cs'$ on $I$ is \emph{stronger} than $\cs$, 
and write $\cs' \leq \cs$, if $\cs' \subseteq \cs$.
\end{definition}

We can naturally take products of columns that are stacked above each other.
\begin{definition}\label{def:sacksproduct}
Let $\cs_1$ be a Sacks column on an interval $I_1$ and let $\cs_2$ be a Sacks column 
on an interval $I_2$.
If $\min ( I_2 ) = \max ( I_1 ) + 1$, then
the \emph{product} $\cs'=\cs_1\otimes\cs_2$ is the
Sacks column on $I_1 \cup I_2$
defined by $f\in\cs'$ iff
$f\restriction I_1\in\cs_1$ and
$f\restriction I_2\in\cs_2$.
\\
Iterating this, we can take products of finitely many properly stacked\footnote{Sacks 
columns $\cs_1, \ldots , \cs_n$ on intervals $I_1 , \ldots , I_n$, respectively, are 
called \emph{properly stacked} if $\min ( I_{i+1} ) = \max ( I_i ) + 1$ for each $i < n$.}
Sacks columns.
\end{definition}

We now define the norm of a Sacks column $\cs$ on an interval $I$.
Actually, we define a family of norms, using two parameters $B$ and $m$.
Later, we will virtually always use 
values of $B$ and $m$ determined by $\min ( I )$; more details will come
in Subsection~\ref{sec:compound} and Section~\ref{sec:complete.construction}.

\begin{definition}\label{def:sacknsor}
$\sacksnor^{B,m}(\cs) \ge n $ iff $n=0$ or $\splitnor(\cs)\ge F^B_{m}(n)$
where $F^B_{m}:\omega\to\omega$ is defined as follows: $F^B_{m}(0)=1$ and
$F^B_{m}(n+1)=f(\maxwidth{m},F^B_m(n),B)$, where we use
the function $f$ of Definition~\ref{def:martinsgloriousdef}.
\end{definition}

In other words, 
\begin{equation}\label{eq:sacksnor}
   \sacksnor^{B,m}(\cs) = \max(\{n\in \omega: F^B_{m}(n)\le \splitnor(\cs)\} \cup \{0\}).
\end{equation}

The exact definition of this norm will not be important in the rest of the
paper; we will only require the following properties.

\begin{lem}\label{lem:martinsgloriouscorollary}\ 
\begin{enumerate}
\item If $\cs, \cs'$ have the same splitting-size, 
   then $\sacksnor^{B,m}(\cs')=\sacksnor^{B,m}(\cs)$.
\item\label{item:two} If $\cs' \leq \cs$, $B'\ge B$ and $m'\ge m$, then $\sacksnor^{{B'},m'}(\cs')\le\sacksnor^{B,m}(\cs)$.
\item\label{item:three} $\sacksnor^{B,m}(\cs_1\otimes\cdots\otimes\cs_n)\ge \sacksnor^{B,m}(\cs_i)$
for all $1\le i\le n$.
\item 
If $I$
is large (with respect to $B$ and $m$),
then $\sacksnor^{B,m}(2^I)$ will be large.  More precisely, 
given $a\in \omega$, if $|I|>F^B_m(a)$, then $\sacksnor^{B,m}(2^I)\ge a$.
\item\label{item:bla} 
We will later use the following simple (but awkward) consequence.
Fix properly stacked intervals $I,I'$ 
and a 
Sacks column $\cs$ on $I \cup I'$. Then there is 
an $\tilde{\cs} \leq \cs$ such that 
\[
  \sacksnor^{B,m} ( \tilde{\cs} ) \geq \min 
    \left(
    \sacksnor^{B,m} ( \cs ) ,
    \sacksnor^{B,m} ( 2^I ) 
    \right),
\] 
and 
$| \tilde{\cs} | \leq | 2^I |$.
\item\label{item:big} (Bigness)
   For $i<\maxwidth{m}$,
   fix Sacks columns $\cs_i $ such that
   $ \sacksnor^{B,m}(\cs_i)\ge n+1$. 
 \\
   Then for any  ``coloring'' function $\pi:\prod_{i<\maxwidth{m}}\cs_i \to B$
    there are Sacks columns $\cs'_i\le\cs_i$ with 
   $\sacksnor^{B,m}(\cs'_i)\ge n$ such that
   $\pi$ is constant on $\prod_{i<\maxwidth{m}}\cs'_i $.
\end{enumerate}
\end{lem}

\begin{proof}
For (\ref{item:bla}), just prune all unnecessary branches. In  more detail, 
note that $\splitnor(2^{I})=|I|$, and that 
$\sacksnor^{B,m}$ is determined by the splitting-size $\splitnor$.
So
we have to find $\tilde \cs\subseteq \cs$
with splitting size $r:=\min(\splitnor(\cs),|I|)$.
Obviously we can find the binary tree $2^{\le r}$ inside $\cs$ (as a suborder).
Extend each of its maximal elements (uniquely), and take the downwards closure.
This gives $\tilde \cs$.

(\ref{item:big}) follows immediately from Lemma~\ref{def:martinsgloriousdef}.
We have 
$\splitnor(\cs_i) \ge 
F^B_{m}(n+1)=f(\maxwidth{{m}},F^B_{m}(n),B)$; so by the 
characteristic property of the function $f$, for any
coloring 
   function $\pi:\prod_{i<\maxwidth{m}}\cs_i \to B$
    there are Sacks columns $\cs'_i\le\cs_i$ with 
   $\splitnor(\cs'_i)\ge F^B_{m}(n)$  such that
   $\pi$ is constant on $\prod_{i<\maxwidth{m}}\cs'_i $.
So $\sacksnor^{B,m}(\cs'_i)\ge n$.
\end{proof}

\subsection{Setting the stage}

We fix for the rest of this paper a nonempty (index) set~$\Xi$.  We furthermore
assume that $\Xi$ is partitioned into subsets $\Xi_{\typels} , \Xi_{\typeli} ,
\Xi_{\typesk}$ ($\Xi_{\typeli}$ is nonempty, but $\Xi_{\typels}$ and
$\Xi_{\typesk}$ could be empty).
For each $\xi \in \Xi$, we say that $\xi$ is of type \emph{lim-sup}, \emph{lim-inf} or \emph{Sacks} if $\xi$ is an element of $\Xi_{\typels}$, $\Xi_{\typeli}$, or $\Xi_{\typesk}$, respectively.
We set $\Xi_{\typenonsk} := \Xi_{\typels} \cup \Xi_{\typeli} = \Xi \setminus \Xi_{\typesk}$.

Our forcing will ``live'' on $\Xi\times \omega$. For $(\xi,\ell)\in
\Xi\times\omega$ we call $\xi$ the \emph{index} and $\ell$ the \emph{level}.

The ``frame'' of the forcing will be as follows.
\begin{definition}\label{def:frame}
\begin{enumerate}
  \item (For the ``Sacks part''.)
    We fix a sequence $( I_{\typesk,\ell} )_{\ell \in \omega}$ of properly stacked intervals in $\omega$.\footnote{I.e., 
    $I_{\typesk,\ell} = [ \min ( I_{\typesk,\ell} ) , \min ( I_{\typesk,\ell+1} ) )$ for all $\ell \in \omega$.}
    For simplicity we further assume that $\min ( I_{\typesk,0} ) = 0$.
   Given natural numbers $\ell < m$ we set 
 $I_{\typesk,[\ell,m)}:=\bigcup_{\ell \leq h < m} I_{\typesk,h}= 
 [\min (I_{\typesk,\ell}), \min(I_{\typesk,m}))$. 
  A Sacks column on $I_{\typesk,[\ell,m)}$ is also called a ``Sacks column 
   between $\ell$ and~$m$''. 
  \item 
    We fix for each level $\ell\in\omega$ some $J_\ell\in\omega \setminus \{ 0 \}$.
  A \emph{sublevel} is a pair $(\ell,j)$ for $\ell\in\omega$
  and $j\in J_\ell\cup\{-1\}$.  (The sublevel $(\ell,-1)$ 
  will be associated with the Sacks part at level $\ell$.)
  We will usually denote sublevels by $\slu$ or $\slv$.
  \item
  We say $\slv$ is below $\slu$, or $\slv<\slu$, if 
  $\slv$ lexicographically precedes $\slu$.
  Note that this order has order type $\omega$.
  \item
  A sublevel $(\ell,-1)$ is called a Sacks sublevel; 
  all other sublevels are called subatomic.
  Instead of $(\ell,-1)$ we will sometimes just write ``the sublevel $\ell$'',
  and we sometimes just write ``$\slv$ is below $\ell$''
  instead of $\slv<(\ell,-1)$.
  \item (For the ``non-Sacks part''.)
    For each subatomic sublevel $\slu$ and index $\xi\in \Xinonsk$
    we fix a subatomic family
    $\cK_{\xi,\slu}$ living on a finite set $\POSS_{\xi,\slu}$.
  \item  For each level $\ell\in\omega$ and index $\xi\in \Xinonsk$, 
    each 
sequence $(x_j)_{j\in J_\ell}$ with $x_j\in
\cK_{\xi,\slu}$  constitutes 
     (as in~\ref{def:atomsfromsubatoms})
     an atom $\ca$, where 
 we use $\maxwidth{\ell}$ as the parameter in
$\mu^{\maxwidth{\ell}}$ for the definition of the norm of the atom.
\end{enumerate}
\end{definition}

\begin{figure}[h]
\centering
\scalebox{0.4}{\includegraphics{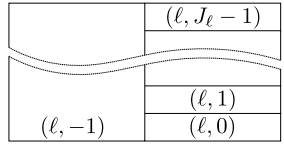}}
\caption{Diagram of the sublevels at level $\ell$, with the Sacks sublevel $(\ell,-1)$ occurring ``before'' the subatomic sublevels $(\ell,0), (\ell,1),  \ldots , (\ell,J_\ell-1)$.}
\label{fig:sublevels}
\end{figure}

To be able to use this frame to construct a reasonable (in particular, proper)
forcing, we will have to add several additional requirements
of the following form. The Sacks intervals
$I_{\typesk,\ell}$ (that ``appear'' at sublevel $\ell$) are ``large''
with respect to everything that was constructed in sublevels $\slv$ below $\ell$; and
the subatoms at a subatomic sublevel $\slu$ have ``large'' bigness with respect
to everything that was constructed at sublevels $\slv<\slu$.  The complete
construction with all requirements will be given in Section~\ref{sec:complete.construction}.

\subsection{Compound creatures} \label{sec:compound}

We can now define compound creatures, which are made up from subatomic
creatures and Sacks columns.  

\begin{definition}\label{def:cc}
A \emph{compound creature} $\cc$ consists of 
\begin{enumerate}
\item natural numbers $\mdn < \mup$;
\item a nonempty, finite\footnote{We could assume without loss of generality
that the size of $\supp$ is at most $\maxwidth{\mdn}$. This will be shown
in Lemma~\ref{lem:variousprunedcrap}.} subset $\supp$ of $\Xi$
\label{item:cc.nonempty}
\item for each $\xi\in \suppord\cap \Xi_\typesk$
a Sacks column $\cc(\xi)$ between $\mdn$ and $\mup$;
\item \label{item:cc.subatoms}
 for each $\xi\in \suppord\cap \Xi_\typenonsk$ and
 each subatomic  sublevel $\slu=(\ell,j)$ with $\mdn\le \ell<\mup$
 a subatom $\cc(\xi,\slu) \in \cK_{\xi,\slu}$; and
\item
for each $\mdn\le \ell<\mup$ a real number $d(\ell) \geq 0$, called the ``halving parameter'' of $\cc$ at level~$\ell$).\footnote{One could (without loss of generality, in some sense) restrict the halving parameter to a finite subset of the reals; then for fixed $\suppord,\mdn,\mup$ there are only finitely many compound creatures.}
\end{enumerate}
We additionally require ``modesty'':\footnote{Again, without this requirement, the resulting forcing poset would be equivalent.}
\begin{enumerate}
\setcounter{enumi}{5}
\item\label{item:modesty} 
for each subatomic sublevel $\slu$ with $\mdn < \slu < \mup$
there is at most one $\xi\in \suppord \cap \Xi_\typenonsk$
such that the subatom $\cc(\xi,\slu)$ is not a singleton.
\end{enumerate}
Note that by (\ref{item:cc.subatoms}) for each level $\ell$ with $\mdn \leq \ell < \mup$ and each $\xi \in \suppord \cap \Xi_\typenonsk$ there is a naturally defined atom $\cc ( \xi , \ell ) := ( \cc ( \xi , ( \ell , j ) ) )_{j \in J_\ell}$.
\end{definition}

We also write $\mdn(\cc) , \mup(\cc) , \supp(\cc)$, $d(\cc,h)$.

We will use the following assumptions (later there will be more;
a complete list will be given in Section~\ref{sec:complete.construction}).
\begin{assumption}\label{asm:bigfattreesandcreatures}
  Let $\ell\in\omega$.
 \begin{itemize}
    \item We fix natural numbers $B(\ell)$ and $\maxposs({<}\ell)$,
     such that for each $k\le \ell$ we have $B(k)\le B(\ell)$  and 
     $\maxposs({<}k)\le \maxposs({<}\ell)$.
(These parameters will be defined in Section~\ref{sec:complete.construction}.)
    \item
   We assume that $I_{\typesk,\ell}$ is large enough so that 
   there are Sacks trees of large norm. (More concretely,
     $\sacksnor^{B(\ell),\ell}(2^{I_{\typesk,\ell}}) \ge \ell.$)
    \item We assume that 
    $J_\ell$ is large enough such that $\mu^\ell(J_\ell)$ is big. 
    (More concretely,
   $\mu^\ell(J_\ell)\ge 2^{\ell\cdot \maxposs({<}\ell)}$).
    \item
   We assume that for every $\xi\in\Xinonsk$ and $j\in J_\ell$ 
   there is (at least) one subatom $x\in\cK_{\xi,(\ell,j)}$
   with $\nor(x)\ge 2^{\ell\cdot \maxposs({<}\ell)}$.
 \end{itemize}
\end{assumption}

Using these assumptions, we can now define the norm of a compound
creature.
\begin{definition}\label{def:norm}
The \emph{norm} of a compound creature $\cc$, $\nor (\cc)$, is defined to be the minimum of the following values.
\begin{enumerate}
\item
The ``width norm'':
\[ \widthnorm ( \supp (\cc) ) := \frac{\maxwidth{\mdn(\cc)}}{|\supp(\cc)|}. \]
\item
For each $\xi \in \suppord \cap \Xi_\typesk$ the ``Sacks norm'' at index $\xi$:
\[ \sacksnor ( \cc (\xi) ):= \sacksnor^{B(\mdn),\mdn} ( \cc (\xi) )\]
(with $\mdn:=\mdn(\cc)$) as defined in~\eqref{eq:sacksnor}.
\item
For each $\xi \in \suppord \cap \Xi_\typenonsk$ the ``lim-sup norm'' at index $\xi$:
\[ \lsnor(\cc,\xi):=\max(\nor(\cc(\xi,h)):\, \mdn\le h<\mup). \]
\item
For each $\mdn(\cc) \leq h < \mup(\cc)$ the ``lim-inf norm'' at level $h$:
\[ \linor^{\maxposs({{<}\mdn})}(\cc,h):=\frac{\log_2(N-d(\cc,h))}{\maxposs({<}\mdn)}, \]
where $N:=\min \{\nor(\cc(\xi,h)):\, \xi\in \suppord \cap\Xi_\typeli\}$.\footnote{As usual, if the logarithm results in a
negative number, or if we apply the logarithm to a negative number, we
instead define the
resulting norm to be $0$. So really we mean 
  $\linor^{\maxposs({{<}\mdn})}(\cc,h):=\frac{\log_2(\max(1,N-d(\cc,h))}{\maxposs({<}\mdn)}$.
}\textsuperscript{,}\footnote{The reason for the logarithm, and the use of the halving parameters, will become clear only in Section~\ref{ss:halving}.}
\end{enumerate}
\end{definition}

(So for both $\lsnor$ and $\linor$ we use the norms of atoms $\cc(\xi,h)$;
recall that the level $h$ of this atom is used in
Definition~\ref{def:atomsfromsubatoms} of $\nor(\cc(\xi,h))$, more
specifically, $\mu^h$ is used to measure the size of subsets of $J_h$.)

\begin{remark}
  As $\supp(\cc)$ is nonempty, the width norm (and thus $\nor(\cc)$ as well)
  is at most $\mdn(\cc)$.
\end{remark}

The assumptions imply the following.
\begin{lem}\label{lem:jhwethewt}
Fix $2<\mdn<\mup$ and $\suppord\subseteq
\Xi$ with $|\suppord|<\maxwidth{\mdn}$ and $\suppord\cap\Xisk$, $\suppord\cap \Xili$,
$\suppord\cap \Xils$ all nonempty.
Then there is a compound creature $\cc$ with
$\mdn(\cc)=\mdn$, $\mup(\cc)=\mup$, $\supp(\cc)=\suppord$ such that 
$\nor(\cc)= \widthnorm(\suppord)$.
\end{lem}

\begin{proof}
We can first use for all subatoms and  Sacks columns the ``large'' ones 
guaranteed by the assumptions. However, this will in general
not satisfy modesty.
So we just apply Lemma~\ref{lem:splitJ} at each $\mdn\le\ell<\mup$,
resulting (for each $\ell$)
in disjoint sets $A^\ell_\xi\subseteq J_\ell$ for $\xi\in\suppord\cap\Xinonsk$.
We keep the large subatoms at the sublevels in $A^\ell_\xi$, and choose arbitrary
singleton subatoms at other sublevels. Now we have a compound creature, whose
norm is the minimum of the following:
\begin{itemize}
\item the width norm;
   \item  
the (unchanged) Sacks norms, which  are $\ge\mdn>\widthnorm(\suppord)$; 
\item the lim-sup norms, noting that all atoms at level $\ell$
have norm $\ge 2^{\ell\cdot \maxposs({<}\ell)}-1\ge 2^{\mdn\cdot \maxposs({<}\mdn)}-1>\widthnorm(\suppord)$,
so all lim-sup norms drop by at most $1$; and
  \item   the lim-inf norms, which  drop by
         an even smaller amount, due to the logarithm.
\qedhere
\end{itemize}
\end{proof}

\begin{fact}
  Let $\cc$ be a compound creature and $u\subseteq \supp(\cc)$ 
  such that $u\cap\Xisk$, $u\cap \Xili$, $u\cap \Xils$ are all nonempty.
  Then the naturally defined $\cc\restriction u$ is again a compound
  creature with norm at least $\nor(\cc)$.
\end{fact}

\begin{definition}\label{def:purelystrongercompound}
  A compound creature $\cd$ is ``purely stronger'' than $\cc$,
  if $\cc$ and $\cd$ have the same $\mdn$, $\mup$, the same halving
  parameters, the same $\supp$; and if for each $\xi\in\suppord\cap\Xisk$
  the Sacks column $\cd(\xi)$ is stronger than $\cc(\xi)$ and
  for each subatomic sublevel $\slu$ that appears in $\cc$ and $\xi\in\suppord\cap\Xinonsk$
  the subatom $\cd(\xi,\slu)$ is stronger than $\cc(\xi,\slu)$.
  (In other words, the only difference between $\cc$ and  $\cd$  occurs 
  at the Sacks columns and the subatoms, where they become stronger.)

  For $r \geq 0$ we say that $\cd$ is ``$r$-purely stronger'' than $\cc$, if additionally
  $\nor(\cd)\ge\nor(\cc)-r$.
\end{definition}

To show that our forcing has the $\al_2$-cc, we will use the following property.
\begin{lem}\label{lem:unioncreature}
  Fix two compound creatures $\cc_1$ and $\cc_2$ with same $\mdn$ and $\mup$
  and the same halving parameters, with
  disjoint supports, and such that $\nor(\cc_1),\nor(\cc_2)>x$.
  Then there exists a
  compound creature $\cd$ 
  with same $\mdn$ and $\mup$ and support $\supp(\cc_1)\cup\supp(\cc_2)$
  such that $\nor(\cd)\ge \frac x{2}-1$ and 
  $\cd\restriction\supp(\cc_i)$ is purely stronger than $\cc_i$
  for $i=1,2$.

  More generally, the same is true if $\cc_1$ and $\cc_2$ are not
  necessarily disjoint, but identical on the
  intersection $u:=\supp(\cc_1)\cap\supp(\cc_2)$, i.e., $\cc_1\restriction u =
  \cc_2\restriction u$.
\end{lem}

\begin{proof}
Let $\cd'$ be the ``union'' of $\cc_1$ and $\cc_2$, which is defined in the obvious
way.\footnote{In particular $\supp ( \cd' ) = \supp ( \cc_1 ) \cup \supp ( \cc_2 )$, and the Sacks columns and subatoms of $\cd'$ at index $\xi \in \supp ( \cd' )$ are exactly those from either $\cc_1$ or $\cc_2$, depending on whether $\xi \in \supp ( \cc_1 )$ or $\xi \in \supp ( \cc_2 )$.}

As $\cd'$ may not satisfy the modesty requirement (\ref{item:modesty}) or Definition~\ref{def:cc}, 
we apply the procedure from the first part of the proof of Lemma~\ref{lem:jhwethewt} to ensure that the resulting object $\cd$ does.
Then $\cd$ is a compound creature with norm $\ge\frac x2-1$.  (The factor $\frac12$ comes from doubling the size of the support, which decreases the width norm.)
\end{proof}

\subsection{The elements (conditions) of the forcing poset \texorpdfstring{$\QQQ$}{Q}}

\begin{definition}\label{def:Q}
$\emptyset$ is the weakest condition. Any other condition
$p$ consists of $\ww^p$,
 $(p(h))_{h\in\ww^p}$ and $t^p$ such that the following are satisfied.
\begin{itemize}
  \item $\ww^p\subseteq \omega$ is infinite. 
  \item For each $h\in \ww^p$, $p(h)$ is a compound creature 
     whose $\mdn$ is $h$, and whose $\mup$ is the $\ww^P$-successor of $h$.
  \item For $h<h'$ in $\ww^p$, $\supp(p(h))\subseteq \supp(p(h'))$,
  \item $\lim_{h\in \ww^p}(\nor(p(h)))=\infty$.
  \item We set $\supp(p):=\bigcup_{h\in \ww^p}\supp(p(h))$ (a nonempty countable subset of $\Xi$).
  \item For $\xi\in\supp(p)$, we define $\trklgth{p}{\xi}$ (the ``trunk length'' at $\xi$)  to be the minimal $h$ such that
    $\xi\in\supp(p(h))$. 
  \item The ``trunk'' $t^p$ is a function which assigns 
    \begin{itemize}
      \item to each $\xi\in\supp(p)\cap \Xisk$ and $\ell<\trklgth{p}{\xi}$
        an element of $2^{I_{\typesk,\ell}}$; 
      \item to each $\xi\in\supp(p)\cap \Xinonsk$ and 
        subatomic sublevel $\slu$ below $\trklgth{p}{\xi}$
        an element of $\POSS_{\xi,\slu}$.
    \end{itemize}
\end{itemize}
\end{definition}

Note that Assumption~\ref{asm:bigfattreesandcreatures} guarantees that $\mathbb
Q$ is nonempty (cf. Lemma~\ref{lem:jhwethewt}).

\begin{notn}
Given $p\in \QQQ$, $h\in\ww^p$ and $\ell$ which is $\ge h$ and less than the $\ww^p$-successor
of $h$, and a sublevel $\slu=(\ell,j) $ we use the following notations.
\begin{itemize}
\item $\supp(p,\slu)=\supp(p,\ell):=\supp(p(h))$.
\item $d(p,\ell):=d(p(h),\ell)$ (the halving parameter of $p$ at level $\ell$).
\item
For $\xi\in \Xi_\typenonsk\cap \supp(p,\slu)$ and $j\neq-1$ we set
$p(\xi,\slu):=p(h)(\xi,\slu)$, the subatom located at index $\xi$ and sublevel $\slu$.
\item  For $\xi\in \Xi_\typesk \cap \supp(p(h))$ we 
set $p(\xi,h):=p(h)(\xi)$, the Sacks column at index $\xi$ starting at
level $h$ (note that we require $h\in\ww^p$).
\end{itemize}
\end{notn}

\subsection{The set of possibilities}\label{sec:poss}

We will now define the ``possibilities'' of a condition $p$, which give
information about the possible value of the generic objects $\yy_\xi$ and
which we will use to define the order of the forcing.
The possibilities of a condition $p$ come from three sources, informally described below.
\begin{itemize}
  \item The trunk $t^p$, where there is a unique possibility.
  \item The subatoms $p ( \xi , \slu )$ (each with a set of possibilities,  $\poss( p ( \xi , \slu ) )$).
  \item The Sacks columns $p ( \xi , h )$ (which we interpret as a set of possible branches)
     which ``live'' between $h\in\ww^p$ and 
     the $\ww^p$-successor $h^+$ of $h$.  The possibilities of the 
     whole Sacks column have to be counted as belonging to the 
     sublevel $(h, -1)$, i.e., we have to list them \emph{before} 
     the subatomic sublevel $(h,0)$, even though their domain 
     reaches up to just below $h^+$.

This property of the Sacks columns will make our notation quite awkward.  As a
consequence, the following section has the worst ratio of mathematical contents
to notational awkwardness. Things will improve later on. We promise.
\end{itemize}

We first (in~\ref{def:vera}) describe a way to define the set of possibilities
separately for each~$\xi\in\supp(p)$; all possibilities then are the product
over the $\xi$-possibilities. 

Then (in~\ref{def:verb}) we will describe a variant in which possibilities
at a sublevel~$\slu$ are defined, and all possibilities are a product over the
$\slu$-possibilities.

Both versions result in the same set of possibilities (up to an awkward
but canonical bijectionl; see Fact~\ref{lem:blabla}).  The first version is more
useful in formulating things such as ``a stronger condition has as smaller set
of possibilities'', whereas the second is the notion that will actually be used
later in proofs.

\begin{definition}\label{def:vera}
	Fix a condition $p$ and an index $\xi\in\supp(p)$.
	\begin{itemize}
		\item   
                  If $\xi \in \Xinonsk$, then for each subatomic sublevel $\slu = ( \ell , j )$ we define the set $\poss(p,\xi,{=}\slu)$ to be either the singleton $\{ t^p ( \xi , \slu ) \}$ (if $\ell < \trklgth{p}{\xi}$), or $\poss (x)$ for the subatom $x := p(\xi,\slu)$ (if $\ell \geq \trklgth{p}{\xi}$).
                         (In either case 
                        $\poss(p,\xi,{=}\slu) \subseteq
			\POSS_{\xi,\slu}$.)
			\\
			We set $\poss(p,\xi,{<}\slu) := \prod \{ \poss(p,\xi,{=}\slv) : \slv<\slu \text{ is a subatomic sublevel} \}$.
		\item If $\xi\in\Xisk$ and $\slu=(m,j)$ is a sublevel, we set $\ell$ to be either 
                      $m$ (if $j = -1$ and $m \in \ww^p$), or the least number $>m$ in $\{0,\dots,\trklgth{p}{\xi}-1\}\cup\ww^p$, (otherwise).
			We then define $\poss(p,\xi,{<}\slu)$ to be the set of all
			functions $\eta\in 2^{[0,\min(I_{\typesk,\ell}))}$
			compatible\footnote{In more detail, for each $h < \ell$ an element of $\{ 0 , \ldots , \trklgth{p}{\xi} - 1 \} \cup \ww^p$ 
                        if $h < \trklgth{p}{\xi}$, then $\eta \restriction I_{\typesk,h} = t^p ( \xi , h )$, 
                        and otherwise $\eta \restriction I_{\typesk,[h,h^\prime)} \in p ( \xi , h )$, where $h^\prime$ is the $\ww^p$-successor of $h$.}
                        with the trunk and the Sacks columns of $p$ 
			at~$\xi$.  
        \item
        We set $\poss(p,{<}\slu)$ to be $\prod_{\xi\in\supp(p)}\poss(p,\xi,{<}\slu)$.
\item
Recall that we identify $\ell$ with the sublevel $(\ell,-1)$, so we can 
write 
 $\poss(p,{<}\ell) $ instead of $ \poss(p,{<}(\ell,-1))$. 
\end{itemize}
\end{definition}

Note that each possibility below $\slu$ restricted to the non-Sacks part can
be seen as a ``rectangle'' with width $\supp(p)\cap\Xinonsk$ and height $\slu$;
whereas the restriction to the Sacks part is a rectangle with height in
$\ww^p$ (which is generally above $\slu$).
So together this gives an ``L-shaped'' domain.
Only in case $\slu=(\ell,-1)$ for $\ell\in\ww^p$ do we get a more pleasant overall
rectangular shape.

In the following alternative definition we ignore
a part of $p$ which is ``trivial'' because we have no freedom/choice left.
More specifically, we ignore the
trunk  and singleton subatoms (but not singleton Sacks columns). 
Also, we do not first
concentrate on some fixed index $\xi$, but directly define $\poss'(p,{=}\slu)$
for certain sublevels $\slu$.
\begin{definition}\label{def:verb}
  We define the set $\sublevels(p)$ of ``active'' sublevels of $p$ 
   by case distinction, and then for each $\slu\in\sublevels(p)$
  we define the object $\poss'(p,{=}\slu)$.
  \begin{itemize}
    \item If $\slu=(\ell,-1)$ is a Sacks sublevel, then
      $\slu\in\sublevels(p)$  iff $\ell\in\ww^p$. In this case we set 
      $S:=\supp(p,\ell)\cap \Xisk\neq \emptyset$, define $p(\slu)$
      to be the sequence $(p(\xi,\ell))_{\xi\in S}$
      of these Sacks columns,
      and set $\poss'(p,{=}\slu) $ to be the product of this sequence. 
    \item If $\slu=(\ell,j)$ is a subatomic sublevel, then
      $\slu\in\sublevels(p)$  iff $\ell\ge \min(\ww^p)$
      and if there is a non-singleton subatom at sublevel $\slu$, say at
      index $\xi$.
      In this case according to the modesty condition (\ref{item:modesty}) of Definition~\ref{def:cc}
      this is the only non-singleton subatom at $\slu$.
      We call $\xi$ the ``active index'' at $\slu$, set
      $p(\slu):=p(\xi,\slu)$
      (the ``active subatom'')
      and define $\poss'(p,{=}\slu):=\poss(p(\slu))$.
  \end{itemize}
  So $\sublevels(p)$ is  a subset (and thus suborder) of the set of all
  sublevels, also of order type $\omega$. 
  We set $\poss'(p,{<}\slu):=\prod \{ \poss'(p,{=}\slv) : \slv<\slu,\ \slv\in\sublevels(p) \}$.
\end{definition}

The definition of the following bijection $\iota$ is easy to see/understand,
but very awkward to formulate precisely, and hence left as an exercise.
\begin{fact}\label{lem:blabla}
There is a natural/canonical
correspondence   $\iota:\poss(p,{<}\slu)\to \poss'(p,{<}\slu)$.
Given an $\eta\in \poss(p,{<}\slu)$, we first omit from $\eta$
all the ``trivial'' information contained in the trunk and in the singleton
subatoms; and then ``relabel'' the resulting sequence (instead of a sequence
indexed by elements of $\xi$
we wish to have one indexed by elements of $\sublevels(p)$).
\end{fact}

Later in this paper we will not distinguish between 
$\poss$ and $\poss'$; actually, we will mostly use $\poss'$, and often use the 
following trivial observation.
\begin{fact}\label{lem:trivialdd}
For $\slv<\slu$ in $\sublevels(p)$, 
\[
  \poss'(p,{<}\slu)=\poss'(p,{<}\slv)\times \poss'(p,{=}\slv)
  \times \poss'(p,{>}\slv),
\]
where we set $\poss'(p,{>}\slv):=\prod_{\slv'\in\sublevels(p),
\slv<\slv'<\slu}\poss(p,{=}\slv')$.
\\
$\poss'(p,{=}\slv)$ is a product of Sacks columns if $\slv$ is Sacks,
otherwise it is $\poss(x)$ for the active subatom at $\slv$.
\end{fact}

\subsection{The order of the forcing}\label{sec:order}

\begin{definition}\label{def:order}
  A condition $q$ is stronger than $p$, written $q\leq p$, iff the following conditions hold.
  \begin{enumerate}
    \item $\ww^q\subseteq \ww^p$.
    \item
      $\supp(p)\cap \supp(q(h))=\supp(p(h))$ for each $h\in \ww^q$.\footnote{This condition in particular implies that $\supp(q(h))\supseteq \supp(p(h))$ for each $h \in \ww^q$, that $\trklgth{q}{\xi} = \min\{ \ell\in \ww^q: \ell \ge \trklgth p\xi \}$ for each $\xi \in \supp (p)$, and that $t^q$ is defined on a  larger domain than $t^p$.}
    \item\label{item:wjk}
      The trunk $t^q$ of $q$ extends the trunk $t^p$ of $p$
      and is ``compatible'' with $p$ in the sense that for each $\xi \in \supp (p)$ 
      the singleton $\poss(q,\xi,{<}\trklgth{q}{\xi}) $ is a subset 
      of $ \poss(p,\xi,{<}\trklgth{q}{\xi}) $.\footnote{%
      Equivalently, for any $\eta\in\poss(q,{<}\min(\ww^q))$,
      the restriction of $\eta$ to $\supp(p)$ is in $\poss(p,{<}\min(\ww^q))$.}
      (I.e., the subatoms and Sacks columns of $p$ that disappeared have become part of the trunk of $q$
      which is compatible with the respective possibilities of $p$.)
    \item If $\xi\in\supp(p)\cap\Xi_\typenonsk$ and $\slu$ is a subatomic sublevel 
      above $\trklgth{q}{\xi}$,
      then the subatom $q(\xi,\slu)$ is stronger than $p(\xi,\slu)$.
    \item If $\xi\in\supp(p)\cap\Xi_\typesk$ and $h\in \ww^q$ such that
      $h\ge  \trklgth{q}{\xi}$, then the Sacks column $q(\xi,h)$ is stronger than
      (i.e., a subset of) the product of the Sacks columns $p(\xi,\ell)$
      for $\ell\in\ww^p$, $h\le \ell<h^+$, where $h^+$ is the $\ww^q$-successor 
      of $h$.
    \item The halving parameters do not decrease; i.e.,
    $d(q,\ell)\ge d(p,\ell)$ for all $\ell\in\omega$ with $\ell\ge\min(\ww^q)$.
  \end{enumerate}
\end{definition}

\section{Some simple properties of \texorpdfstring{$\QQQ$}{Q}}\label{sec:Qprop}

\subsection{Increasing the trunk}

We now introduce an obvious way to strengthen a condition: increase the trunk.

\begin{definition}\label{def:wedge1}
  Given $\ell\in\ww^p$ and $\eta\in \poss(p,{<}\ell)$, we define $p \wedge \eta$
  to be the condition $q$ resulting from replacing 
  the compound creatures below $\ell$ with the trunk~$\eta$.
  More formally, 
  $\ww^q:=\ww^p  \setminus\ell$, $q(k) := p(k)$ for all $k\in\ww^q$,
  $t^q(\xi,\slu) := \eta ( \xi , \slu )$ for each $\xi \in \supp(p) \cap \Xinonsk$ and each subatomic $\slu < \ell$,  
  and $t^q(\xi,h) = \eta ( \xi ) \restriction I_{\typesk,h}$ for each $\xi \in \supp(p) \cap \Xisk$ and each $h < \ell$.
\end{definition}

The definition of the order yields the following simple consequences.
\begin{fact}\label{lem:martinwantsthis}
Fix $p \in \QQQ$ and $\ell\in\ww^p$.
\begin{itemize}
  \item For $\eta\in \poss(p,{<}\ell)$, $p\wedge\eta\le p$.
  \item $\{p\wedge \eta:\, \eta\in \poss(p,{<}\ell)\}$ is predense below $p$.
  \item In particular, assuming that
    $p$ and $q$ are conditions that above some $\ell_1$
    have the same $\ww$ and the same compound creatures,\footnote{More formally, $\ell_1\in\ww^p$, $\ww^p\setminus\ell_1=\ww^q\setminus\ell_1$, and 
$p(h)=q(h)$ for all $h\in\ww^p\setminus\ell_1$. Note that this implies $\supp(p)=\supp(q)$.}
   and that $\poss(q,{<}\ell_1)\subseteq \poss(p,{<}\ell_1)$, 
   then $q\leq ^* p$.\footnote{Here, $q\leq^* p$ means that $q$ forces that $p$ belongs to the generic filter. Equivalently, every $r\leq q$
    is compatible with $p$.}
\end{itemize}
\end{fact}

We can define a variant of $\wedge$, which works for any sublevel 
(not only those Sacks sublevels  $\slu=(\ell,-1)$ with $\ell\in \ww^p$).

\begin{definition}\label{def:wedge2}
  Given $\eta\in \poss(p,{<}\slu)$, we define $p\curlywedge \eta$
  as the condition $q$ obtained by replacing the according parts of $p$ with the singleton subatoms (or
  singleton Sacks columns) given by $\eta$.
  More formally, the only possible differences between $p$ and $q$ are that for each subatomic sublevel 
  $\slv < \slu$ and each $\xi \in \supp(p,\slv) \cap \Xinonsk$ the subatomic creature $q(\xi,\slv)$ is 
  the singleston subatom $\{ \eta ( \xi,\slv ) \}$, and for each $\ell \in \ww^p$ strictly below $\slu$ 
  and each $\xi \in \supp(p,\ell) \cap \Xisk$ the Sacks column $p(\xi,\ell)$ is the singleton 
  $\{ \eta ( \xi ) \restriction I_{\typesk,[\ell,\ell^+)} \}$, where $\ell^+$ is the $\ww^p$-successor of $\ell$.
\end{definition}

We can now define the generic sequence added by the forcing. (Note that the
generic filter will generally not be determined by this sequence, due to additional
information given by $\ww$ and the halving parameters.)

\begin{definition}
For $\xi\in \Xinonsk$, let  $\yy_\xi$ be (the name for) 
\[
  \{(\slu,a):\,\slu\text{ a subatomic sublevel and }(\exists p\in G)\ t^p(\xi,\slu)=a \}.
\]
For  $\xi\in \Xisk$, we set $\yy_\xi$ to be 
\[
  \bigcup\{t^p(\xi,\ell):\, p\in G,\, \ell<\trklgth{p}{\xi} \}.
\]
\end{definition}

\begin{fact} Let $\slu$ be a sublevel.
\begin{itemize}
\item
  For $\eta\in\poss(p,{<}\slu)$, $p\curlywedge \eta\le p$.
\item 
  If $\ell\in\ww^p$, $\slu=(\ell,-1)$ and $\eta\in\poss(p,{<}\ell)$, then 
  $p\curlywedge \eta\le ^* p\wedge\eta$ and $p\wedge\eta\le p\curlywedge \eta$.
\item 
  $\{p\curlywedge \eta:\, \eta\in \poss(p,{<}\slu)\}$ is predense below $p$.
\item 
  $p\curlywedge \eta$ and $p\curlywedge \eta'$ are incompatible if 
  $\eta' , \eta \in \poss(p,{<}\slu)$ are distinct.
\item
  $p\curlywedge \eta$ forces that $\bar\yy$ extends $\eta$, i.e., 
  that $\yy_\xi$ extends $\eta(\xi)$ for all $\xi\in\supp(p)$.
\\
  In particular, $p$ forces that $\bar\yy$ extends $t^p$.
\item
  $\eta\in\poss(p,{<}\slu)$ iff\/\footnote{For the direction 
  ``right to left'', which we will not need in this paper, 
  we of course  have to assume that $\eta$ has the right ``format'', i.e.,
  $\eta=\prod_{\xi\in\supp(p)}\eta(\xi)$ and each 
  $\eta(\xi)$ has the appropriate length/domain.} $p$ does not force that $\eta$ is
  incompatible with the generic reals $\bar{\yy}$.
\item 
  For $\eta\in\poss(p,{<}\slu)$, $p \Vdash \text{``}\bar\yy \text{ extends } \eta \Leftrightarrow p\curlywedge \eta\in G\text{.''}$
\item 
  $\QQQ$ forces that $\bar\yy$ is ``defined everywhere''; i.e., $\yy_\xi \in 2^{\omega}$ 
  for all $\xi \in \Xisk$, and $\yy_\eta ( \slu ) \in \POSS_{\xi,\slu}$ is defined for 
  all $\xi \in \Xinonsk$ and every subatomic sublevel $\slu$.
\end{itemize}
\end{fact}

\begin{proof}[Proof of the last item.]
Given a condition $p$ and $\xi\in\Xi$, we have to show that we can find a $q\leq p$ 
with $\xi\in\supp(q)$. 
This is shown just as Lemma~\ref{lem:unioncreature}, using at $\xi$
the large Sacks columns/subatoms guaranteed by~\ref{asm:bigfattreesandcreatures}.
Then ``increasing the trunk'' shows that $\yy_\xi(n)$ is defined for all $n$.
\end{proof}

Note that we can use the equivalent $\poss'$ (defined in~\ref{def:verb})
instead of $\poss$.  Formally, we could use the bijection $\iota$
of~\ref{lem:blabla} and set $p\wedge\eta':= p\wedge \iota^{-1}(\eta')$ for
$\eta'\in \poss'(p,{<}l)$ (and $p\curlywedge\eta':= p\curlywedge
\iota^{-1}(\eta')$ for $\eta'\in \poss'(p,{<}\slu)$). But what we really mean is that
for some $\eta'\in\poss'$ we can define $p\wedge\eta'$ ($p\curlywedge
\eta'$) in the obvious and natural way; and this results in the same object as
when using $p\wedge\eta$ ($p\curlywedge\eta$) for the $\eta\in\poss$ that
corresponds to $\eta'$ (i.e., for $\eta=\iota^{-1}(\eta')$).

\subsection{The set of possibilities of stronger conditions}
If $q\leq p$, then $\poss(q,{<}\slu)$ is ``morally'' a subset of
$\poss(p,{<}\slu)$ for any $\slu$. 

If we just consider a sublevel $(\ell,-1)$ for $\ell\in\ww^q$ then this
is literally true:
\begin{quote}
  Assume that $q\leq p$,
  $\xi\in\supp(p)$ and $\ell\in\ww^q$. Then
  $\poss(q,\xi,{<}\ell)\subseteq \poss(p,\xi,{<}\ell)$.
\end{quote}

In the general case it is more cumbersome to make this explicit
for the Sacks part.    
However, we will only need the following.

\begin{lem}\label{lem:superlemma}
  Given $q\leq p$ and $\eta\in\poss(q,{<}\slu)$ there is a unique
  $\eta'\in\poss(p,{<}\slu)$ such that $q\curlywedge \eta\le
  p\curlywedge \eta'$.
\end{lem}

\begin{proof}
  Uniqueness follows from the fact that $p\curlywedge \eta'$ and
  $p\curlywedge \eta''$ are incompatible for distinct $\eta',\eta''$
  in $\poss(p,{<}\slu)$.

  We define $\eta'(\xi)$ separately for each $\xi\in\supp(p)$.
  For $\xi\in\Xinonsk$ we just use $\eta'(\xi):=\eta(\xi)$.
  So assume $\xi\in \Xisk$. Let $k$ be the smallest element of $\ww^p$ above
  $\slu$.

  \begin{itemize}
    \item If $\slu$ is below $\trklgth{p}{\xi}$ (and therefore
      also below  $\trklgth{q}{\xi}$), then
      again we set $\eta'(\xi):=\eta(\xi)$.
    \item If $\slu$ is above $\trklgth{p}{\xi}$ but below $\trklgth{q}{\xi}$,
      then we extend $\eta(\xi)$ up to $k$ with the values given by the trunk 
      $t^q$.
      This gives $\eta'(\xi)$.
    \item If $\slu$ is above  $\trklgth{q}{\xi}\ge \trklgth{p}{\xi}$,
      then $\eta'(\xi)$ is the restriction of $\eta(\xi)$ to~$k$.\qedhere
  \end{itemize}
\end{proof}

\begin{remark}
Note that $q\leq p$ does not imply $\sublevels(q)\subseteq \sublevels(p)$, as
a previously ``inactive'' sublevel of $p$ can become active in $q$ (with active index outside of $\supp(p)$, of
course).
Also, $\slu$ can be an active subatomic sublevel in both $p$ and $q$,
but with different active indices.
The ``old'' active subatom at $\xi$ can shrink to a singleton in $q$,
while $q$ gains a new index with an active subatom (outside of $\supp(p)$).
Because of this, it is even more cumbersome to formulate an exact version of
``stronger conditions have fewer possibilities'' for $\poss'$ than it is for $\poss$.
\end{remark}

\subsection{\texorpdfstring{$\aleph_2$}{Aleph-2} chain condition}

\begin{lem}
  Assuming CH, $\QQQ$  is $\aleph_2$-cc.
\end{lem}

\begin{proof}
  Assume that $A=\{p_i:\, i\in\aleph_2\}$ is a set of conditions.
  By thinning out $A$ (only using CH and the $\Delta$-system lemma for 
  families of countable sets), 
  we may assume that there is a countable set  $\Delta\subseteq \Xi$
  such that for $p\neq q$ in~$A$ the following hold:
  \begin{itemize}
    \item $\ww^p=\ww^q$;
    \item $d(p,\ell)=d(q,\ell)$ for all $\ell\ge \min( \ww^p)$;
    \item $\Delta=\supp(p)\cap\supp(q)$, and, moreover, 
      $\supp(p,\ell)\cap\Delta=\supp(q,\ell)\cap\Delta$
      for all $\ell\in\ww^p$; and
    \item $p$ and $q$ are identical on $\Delta$, i.e., 
      for each $\ell\in\ww^p$ the compound creatures $p(\ell)$
      and $q(\ell)$ are identical on the intersection,
      as in Lemma~\ref{lem:unioncreature};
      and the trunks agree on $\Delta$, i.e., $t^p(\xi,\ell)$
      is the same as $t^q(\xi,\ell)$ for each $\xi\in \Delta\cap\Xisk$ and $\ell<h(\xi)$, and analogously for the subatomic sublevels.
  \end{itemize}

  As in Lemma~\ref{lem:unioncreature}
  we can (for each $p,q\in A$ and  $\ell\in\ww^p$)
  find a compound creature $\cd(\ell)$
  ``stronger than'' both $p(\ell)$ and $q(\ell)$.
  These creatures (together with the union of the trunks)
  form a condition stronger than both $p$ and $q$.  Hence $A$ is not an
  antichain.
\end{proof}

\subsection{Pruned conditions}

Let $p$ be a condition. All compound creatures $p(\ell)$ above some $\ell_0$
will have norm at least $1$. Note that by the definition of $\widthnorm$ this
implies that $|\supp(p,\ell)|\le\ell$.

The norm of a compound creature $\cc$ is at most $\mdn$
(where we set $\mdn:=\mdn(\cc)$).
We assumed that $\sacksnor^{B(\mdn),\mdn}(2^{I_{\typesk,\mdn}})$
is at least $\mdn$.
Let $\cs$ be any Sacks column in $\cc$.
By Lemma~\ref{lem:martinsgloriouscorollary}(\ref{item:bla})  (using 
$I:=I_{\typesk,\mdn}$ and $I':= I_{\typesk,[\mdn+1,\mup)}$),
 there
is an $\tilde \cs\subseteq \cs$ with
$|\tilde\cs|\le 2^{I_{\typesk,\mdn}}$ and 
$\sacksnor^{B(\mdn),\mdn}(\tilde\cs)\ge \min(\mdn,\sacksnor^{B(\mdn),\mdn}(\cs))$.
So when we replace $\cs$ by $\tilde\cs$ in $\cc$, the norm of the compound
creature does not change. Furthermore, if we replace all Sacks columns in 
$\cc$ with appropriate stengthenings, the resulting compound creature $\cd$
will be $0$-purely stronger than $\cc$.\footnote{See Definition~\ref{def:purelystrongercompound}.}

This leads us to the following definitions.

\begin{definition}
  We call a Sacks column $\cs$ between $\ell$ and $n$ \emph{Sacks-pruned}
  if $|\cs|\le 2^{|I_{\typesk,\ell}|}$.
  A compound creature is \emph{Sacks-pruned} if all its Sacks columns are.
  A condition $q$ is \emph{Sacks-pruned} if $q(h)$ is Sacks-pruned for all $h \in \ww^q$.
  A condition $p$ is \emph{pruned} if it is Sacks-pruned and all compound creatures $p(h)$ have norm bigger than~$1$.
\end{definition}

\begin{definition}\label{def:purelystrongercondition}
A condition $q$ is \emph{purely stronger} (\emph{$r$-purely stronger}) 
than $p$, if $\ww^q=\ww^p$, $t^q = t^p$, and $q(\ell)$ is purely stronger ($r$-purely stronger)
than $p(\ell)$ for all $\ell\in\ww^q$.
(Note that this implies $q\le p$.)
\end{definition}

For every condition $p$ there is a Sacks-pruned condition $q$ which is $0$-purely stronger than $p$.
Given $p\in \QQQ$ Sacks-pruned, $\ell\in\ww^p$ sufficiently large, and $\eta\in\poss(p,{<}\ell)$,
the condition $q=p\wedge \eta<p$ is pruned.

In particular, we have the following. 
\begin{fact}\label{lem:variousprunedcrap}\ 
\begin{itemize}
\item If $p$ is pruned, then 
$|\supp(p(h))|<\maxwidth{h}$ for all $h\in\ww^p$.
\item 
The set of pruned conditions in $\QQQ$ is dense.
\end{itemize}
\end{fact}

\subsection{Gluing}

So far we have increased trunks to strengthen conditions, as well as
taking disjoint unions and pure strengthenings.
This subsection introduces two more methods of strengthening conditions.

\begin{definition}
A compound creature $\cd$ is the result of \emph{increasing the halving parameters} in
$\cc$, if $\cd$ and $\cc$ are identical except that for each $\mdn \leq \ell < \mup$ we may have $d(\cd,\ell) > d(\cc,\ell)$.

Analogously, we define a condition $q$ to be the result of \emph{increasing the
halving parameters} in $p$.
(Again, this implies $q\le p$.)
\end{definition}

\begin{definition}\label{def:properlystacked}
We call a finite sequence of compound creatures $\cc_1,\dots,\cc_n$ \emph{properly
stacked} if $\mup(\cc_i)=\mdn(\cc_{i+1})$ and $\supp(\cc_i)\subseteq
\supp(\cc_{i+1})$ for each $i < n$.
Given such a sequence, we can glue it together to get the new creature $\cd=\glue(\cc_1,\dots,\cc_n)$ with the following description.
\begin{itemize}
\item $\mdn(\cd)=\mdn(\cc_1)$ and $\mup(\cd)=\mup(\cc_n)$
  (i.e., vertically the creature lives on the union of the levels of the old creatures).
\item $\supp(\cd)=\supp(\cc_1)$ (i.e., the rectangle-shape of the new creature 
  is the result of taking the union of the old rectangles and cutting off the 
  stuff that sticks out horizontally beyond the base). 
\item For $\xi\in\supp(\cd)\cap\Xi_\typenonsk$ and subatomic sublevels 
   $\slu$ between  $\mdn(\cd) $ and $\mup(\cd)$,
  the subatom $\cd(\xi,\slu)$ is $\cc_i(\xi,\slu)$ for the appropriate $i$.
\item For $\xi\in\supp(\cd)\cap\Xi_\typesk$, the Sacks column
  $\cd(\xi)$ is defined as the product 
  $\cc_1(\xi)\otimes \cdots\otimes \cc_n(\xi)$.
\end{itemize}
\end{definition}
By the definition of the norm (see~\ref{def:norm}), 
the monotonicity of $B$ and $\maxposs$ (Assumption~\ref{asm:bigfattreesandcreatures}) and Lemma~\ref{lem:martinsgloriouscorollary}(\ref{item:two}),(\ref{item:three}) we get 
\[ \nor(\glue(\cc_1,\dots,\cc_n))\ge \min(\nor(\cc_1),\dots,\nor(\cc_n)).\]

This gives another way to strengthen a condition $p$: shrinking the set $\ww$.
\begin{definition}
Given a
condition $p$ and an infinite subset $U$ of $\ww^p$ such that
$\min(U)=\min(\ww^p)$, we 
say that $q$ results from \emph{gluing $p$ along $U$} if
\begin{itemize}
  \item $\ww^q=U$; 
  \item for $h\in \ww^q$,
    if $h=h_1<h_2<\cdots<h_n$ enumerates the
    elements of $\ww^p$ that are $\ge h$ and less than 
    the $\ww^q$-successor of $h$, then the compound creature $q(h)$ is
    $\glue(p(h_1),\dots,p(h_n))$; and
  \item the new parts of the trunk are compatible with $p$.
\end{itemize}
\end{definition}
Note that $q$ is not uniquely determined by $p$ and $U$, as in general there are many choices to increase
the trunk (in the last item).
Of course, any such resulting $q$ is stronger than $p$.

\begin{figure}[h]
\centering
\subfigure[]{%
\label{subfig:cond}
\scalebox{0.4}{\includegraphics{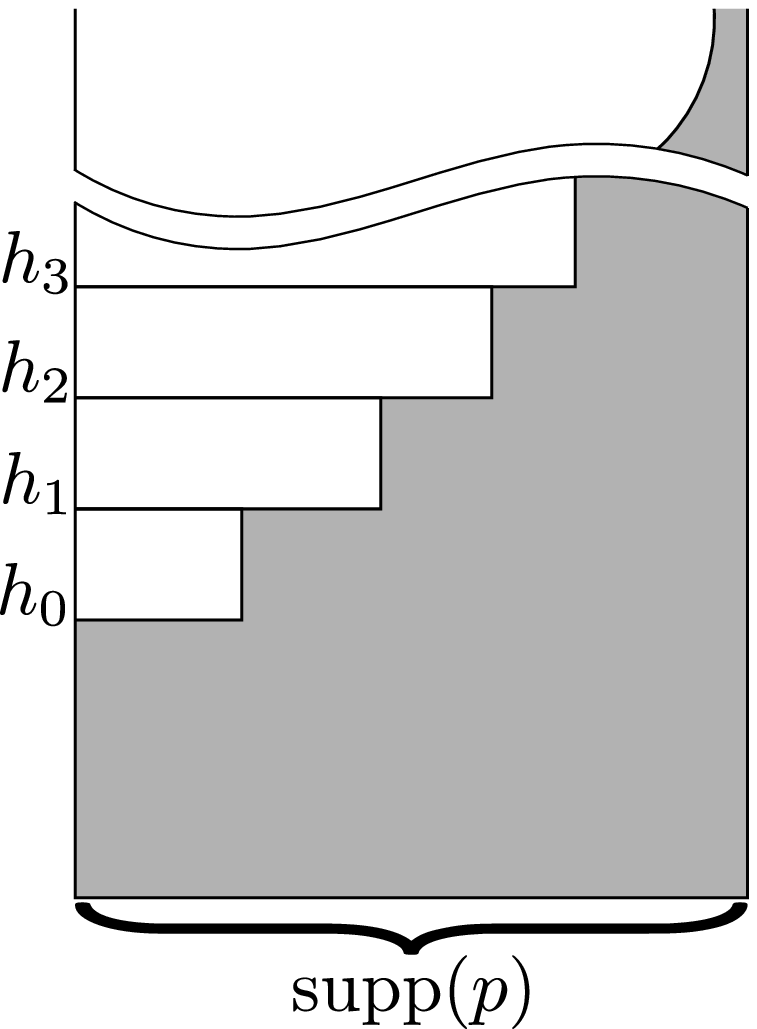}}
}
\quad
\subfigure[]{%
\label{subfig:cond-exttrunk}
\scalebox{0.4}{\includegraphics{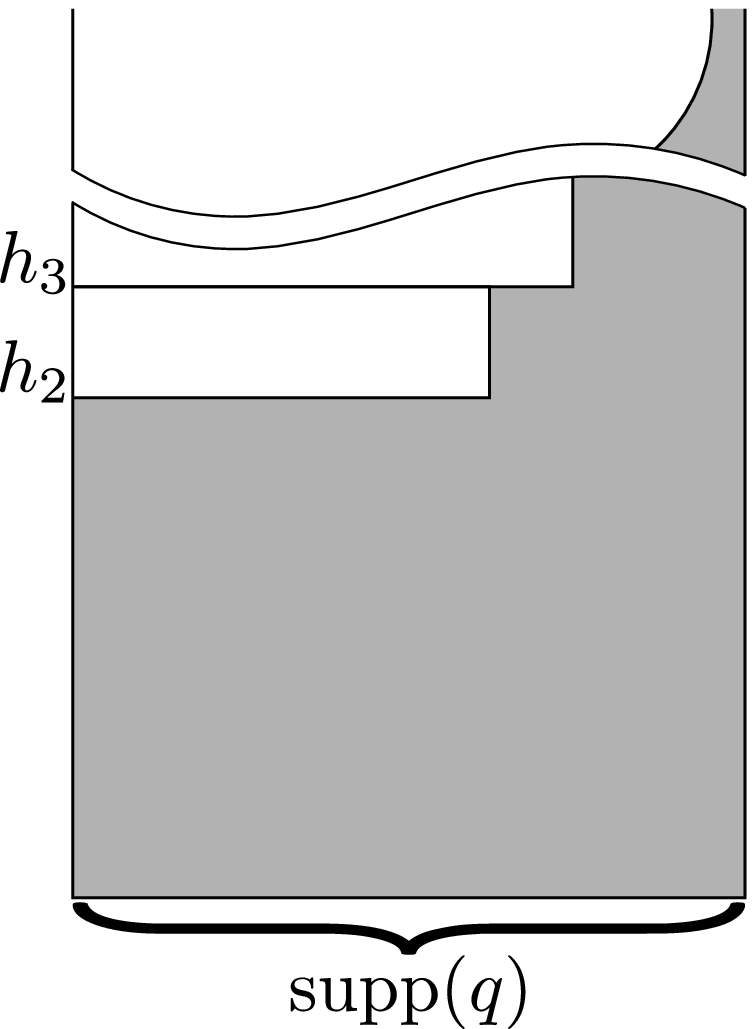}}%
}
\\
\subfigure[]{%
\label{subfig:cond-gluing}
\scalebox{0.4}{\includegraphics{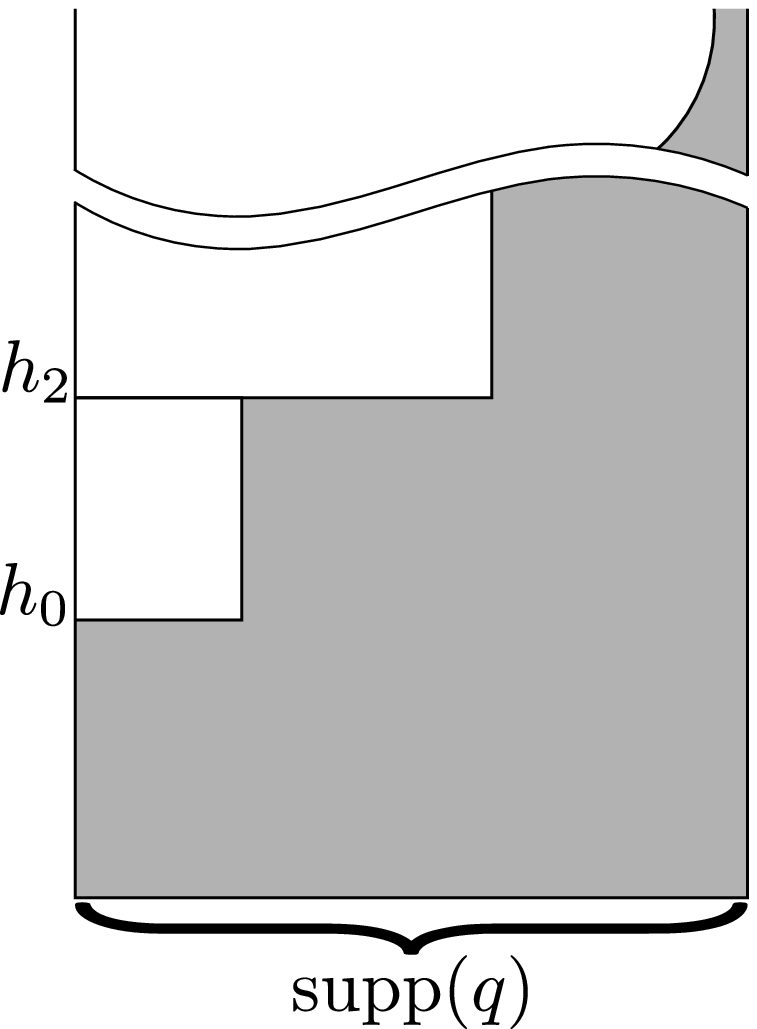}}%
}
\quad
\subfigure[]{%
\label{subfig:cond-disunion}
\scalebox{0.4}{\includegraphics{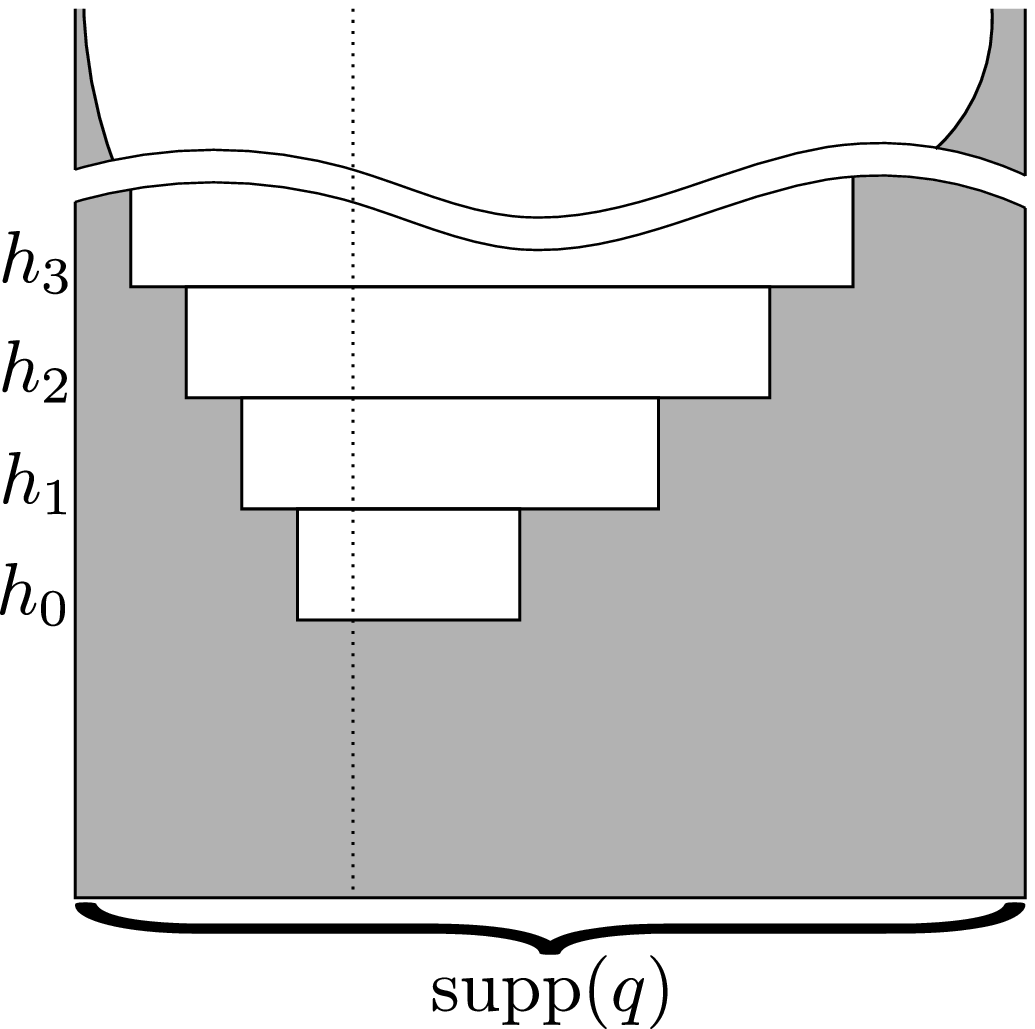}}%
}
\caption{\subref{subfig:cond} A schematic diagram of a condition $p$ of the forcing.  The $h_i$ indicate an increasing enumeration of $\ww^p$, while the shaded region represents the domain of the trunk function $t^p$.
\subref{subfig:cond-exttrunk} A condition $q = p \wedge \eta$, where $\eta \in \poss (p,{<}h_2)$. In particular, all of the compound creatures above level $h_2$ have been left unchanged, and the below level $h_2$ the condition $q$ consists entirely of trunk, with values determined by $\eta$.
\subref{subfig:cond-gluing} A condition $q$ obtained from $p$ by gluing the pairs of compound creatures $p(h_0), p(h_1)$ and $p(h_2), p(h_3)$.  Note that $\trklgth{q}{\eta} = h_2$ for any $\eta \in \supp(p)$ with $\trklgth{p}{\eta} = h_1$ (and similarly if $\trklgth{p}{\eta} = h_3$).
\subref{subfig:cond-disunion} A condition $q$ obtained as the ``disjoint union'' of $p$ and another condition (represented to the left of the dotted line) with the same $\ww$ and the same halving parameters at each level as $p$.
}
\label{fig:alltheconds}
\end{figure}

We have now seen five specific ways to strengthen a condition.
Actually, every $q\le p$ can be obtained from $p$ by a combination of these 
methods. (We will not use the following fact, nor the subsequent remark,
in the rest of the paper.) 

\begin{fact}
For $p,q\in \QQQ$, $q\leq p$ iff there are $p_1$, $p_2$, $p_3$ and $p_4$ such that:
\begin{enumerate}
  \item $p_1$ results from increasing the trunk in $p$, i.e.,
    $p_1=p\wedge\eta$ for some $\eta\in\poss(p,{<}\min(\ww^q))$
    (in fact, for the unique $\eta$ which is extended by $t^q$);
  \item $p_2\leq p_1$ results from gluing $p_1$ along $\ww^q$, as above.
  \item $p_3$ is purely stronger than $p_2$; 
  \item $p_4\leq p_3$ results from increasing halving parameters; and 
  \item $q$ is the naturally defined ``disjoint union'' of $p_4$ and some condition 
    $p'$ which has the same $\ww$ and halving parameters as $p_4$, $\supp(p')$ is 
    disjoint from $\supp(p_4)$, and which jointly satisfies ``modesty'' with $p_4$.
\end{enumerate}
\end{fact} 

\begin{remark}
\begin{itemize}
\item
Every $q$ obtained by the above construction is stronger than $p$,
provided it is a condition. Note that constructions (1), (2) and (5) always result in conditions (for (5), this is the same argument as in~\ref{lem:unioncreature}),
whereas constructions (3) and (4) will generally decrease the norms of the compound creatures in an uncontrolled fashion. 
So to get a condition, we have to make sure 
that the norms of the new compound creatures still converge
to infinity. Also, to be able to find a suitable $p'$ in
(5), we should make enough room for modesty in (3).
\item
The order is not entirely irrelevant, as gluing (2) has to be done before pure
strengthening (3), since glued Sacks columns always have the form of products, 
whereas generally the Sacks columns in $q$ will not be of this form.
\end{itemize}
\end{remark}

We will later use the following specific gluing construction.

\begin{lem}\label{lem:blab}
  Assume that
  $\cc_0, \dots, \cc_n$ is a properly stacked sequence of compound creatures, 
  $n>0$, and $\nor(\cc_i)\ge M$ for all $i\le n$.  
  Pick for each  $i<n$  some compound creatures $\cd_i$, purely stronger
  than $\cc_i$, such that 
  $\cd_i$ and $\cc_i$ agree on the lim-inf part
  (but $\cd_i$ could consist of singletons on the lim-sup
   and the Sacks part). 
   Set $\cd_n=\cc_n$.
  Then $\glue(\cd_0,\dots,\cd_n)$ has norm $\ge M$ as well.
\end{lem}

\begin{proof}
  The lim-sup norm and the Sacks-norms will be large because $\nor(\cd_n)=\nor(\cc_n)\ge M$.
  The lim-inf norm will be large because we did not change anything on the lim-inf part.
\end{proof}

\subsection{Projections and complete subforcings}

\begin{lem}\label{lem:quotient}
Assume that $\Xili\subseteq\Xi'\subseteq \Xi$.\footnote{If we do not assume 
$\Xi'\supseteq \Xili$ we get problems with the lim-inf norm when we combine the 
increased halving parameters of $q'$ with the lim-inf creatures in $p_1$.}
Let $\QQQ_{\Xi'}\subseteq \QQQ$ consist of all $p\in\QQQ$ with $\supp(p)\subseteq \Xi'$.
Then $\QQQ_{\Xi'}$ is a complete subforcing of $\QQQ$, and the restriction map is a projection
on an open dense subset.
\\
Of course, $\QQQ_{\Xi'}$ will satisfy all the properties that we will prove generally for
$\QQQ$ (as $\QQQ_{\Xi'}$ is defined just like $\QQQ$, only with a different underlying index set).
\end{lem}

\begin{proof}
The dense set $D$ is the set of all conditions $p$ with 
$\supp(p)\cap\Xili\neq\emptyset$.
Fix $p\in D$, set $p':=p\restriction\Xi'$, and assume that $q'\leq p'$ is in $\QQQ'$.
It is enough to show that $q'$ is compatible with $p$.
To do this we will construct $q\leq p$ such that $q'=q\restriction\Xi'$.

Set $p_1:=p\restriction (\Xi\setminus\Xi')$. Increase the trunk of $p_1$ to $\min(\ww^{q'})$, glue along $\ww^{q'}$, and increase the halving parameters to match those of $q'$ to get a condition $q_1\leq p_1$ with $\ww^{q_1}=\ww^{q'}$.\footnote{As $\supp(p_1) \cap \Xili = \emptyset$ it follows that increasing the halving parameters does not affect the norms of the compound creatures, and therefore $q_1$ is a condition of $\QQQ$.}
Letting $q$ be the disjoint union of $q_1$ and $q'$, it follows that $q$ is a condition of $\QQQ$, and clearly $q \restriction \Xi' = q'$.
\end{proof}

\section{An inductive construction of \texorpdfstring{$\QQQ$}{Q}}\label{sec:complete.construction}

We will now review the ``framework'' from Definition~\ref{def:frame}, finally giving all the
assumptions (including the previous Assumption~\ref{asm:bigfattreesandcreatures}) that are required to make the
forcing proper.

In the following construction, we have the freedom to choose the following
(as long as the assumptions are satisfied).

\begin{itemize}
  \item $\Xi = \Xils \cup \Xili \cup \Xisk$, as in Definition~\ref{def:frame}.
  \item Natural numbers $H({<}\slu)$ (for each sublevel $\slu$)
    such that $H$ is increasing.
   \\ \textit{Remark.} The function $H$ gives us the possibility to impose additional 
    demands on the bigness $B$ (as given in \eqref{eq:cc4b}, below).  It is not needed to get properness and $\omega^\omega$-bounding, but will be used
later\footnote{Here is a very informal description of how $H$ will be used. 
The basic requirement is that at each sublevel $\slu$ we have bigness (namely $B(\slu)$)  which is large with respect to everything that happened below.  However, the notion of ``large with respect to'' will slightly  depend on 
the actual construction that increases the relevant cardinal characteristic.  The parameter $H$ will allow us to accommodate these different interpretations.  The function $H$ will be used as a parameter when defining ``rapid reading'' in Definition~\ref{def:continuousreading}.}
in our specific constructions.
  \item For each $\xi \in \Xinonsk$ and each subatomic sublevel $\slu$ the subatomic family $\cK_{\xi,\slu}$ living on some finite set 
     $\POSS_{\xi,\slu}$.
\end{itemize}

The other parameters are determined by the construction.
\begin{itemize}
  \item  Natural numbers $\maxposs({<}\slu)$ for each sublevel $\slu$.
\\
This will turn out to be an upper bound to the cardinality 
 of $\poss(p,{<}\slu)$ for any pruned condition $p$.
 \item For each sublevel $\slu$, we set
\begin{equation}\label{eq:cc4b}
  B(\slu):=2^{H({<}\slu)\cdot\maxposs({<}\slu)}
\end{equation}
(and we set $B((0,-1)):=2$).
$B(\slu)$ is the bigness required for the subatoms (or Sacks columns) at $\slu$.
  \item The Sacks intervals $I_{\typesk,\ell}$ and subatomic index sets $J_\ell$, for each $\ell\in\omega$, as in Definition~\ref{def:frame}.
\end{itemize}
Note that, as usual,
   for a Sacks  sublevel $\slu = (\ell,-1)$ we may write $B(\ell)$ for 
   $B(\slu)$.   We similarly use $\maxposs({<}\ell)$ and $H({<}\ell)$.

By induction of $\ell$ we now make the following definitions and requirements.
(We also set the ``initial values'' $\maxposs({<}(0,-1)):=1$ and $I_{\typesk,-1}=\{-1\}$.)
\begin{construction}\ 
  \begin{enumerate}[$(*1)$]
    \item
          We require that $H({<}\ell)>\maxposs({<}\ell)+\ell+2$.
    \item
          \textit{The Sacks sublevel.}
          We let $I_{\typesk,\ell}$ be the interval starting at $\max(I_{\typesk,\ell-1})+1$
          and of minimal size such that 
          $\sacksnor^{B(\ell),\ell}(2^{I_{\typesk,\ell}})\ge \ell$.
          \\
          This gives us ``bigness'' in the form of Lemma~\ref{lem:martinsgloriouscorollary}(\ref{item:big}) for $B:=B(\ell)$.
   \item
          We set $\maxposs({<}(\ell,0)):=\maxposs({<}\ell)\cdot 2^{|I_{\typesk,\ell}|\cdot \maxwidth{\ell}}$.
    \item We set $J_\ell:=3^{\maxwidth{(\ell+1)}\cdot 2^{\ell\cdot\maxposs({<}\ell)}}$.  So $\mu^{\maxwidth{\ell}}(J_\ell)=2^{\ell\cdot \maxposs({<}\ell)}$.\footnote{$\mu^\ell$ is defined in~\ref{def:mudef}.}
    \item\label{item:sa} \textit{The subatomic sublevels.} By induction on $j\in J_\ell$ we now deal with the sublevel $\slu=(\ell,j)$.
      \begin{enumerate}
        \item  For each $\xi\in\Xinonsk$, we require that $\cK_{\xi,\slu}$ is a subatomic family living on some finite set $\POSS_{\xi,\slu}$.
        \item\label{si:sanorm}  For each $\xi\in\Xinonsk$, we require that there is a subatom $x\in \cK_{\xi,\slu}$ with norm at least $2^{\ell\cdot \maxposs({<}\ell)}$. 
        \item\label{si:sabig}  For each $\xi\in\Xinonsk$, we require that $\cK_{\xi,\slu}$ is $B(\slu)$-big.
        \item\label{si:sabound}  We require that there is a uniform bound $M(\slu)=\max(\{|\POSS_{\xi,\slu}|:\, \xi\in\Xinonsk\}) $. Then we set, for 
        $\slv$ the successor sublevel of $\slu$,  
\[ 
\maxposs({<}\slv):=\maxposs({<}\slu)\cdot {M(\slu)^{\maxwidth{\ell+1}}}.
\]
               (In particular this defines $\maxposs({<}(\ell+1,-1))$ if $\slu=(\ell,J_\ell-1)$.)
       \end{enumerate}
  \end{enumerate}         
\end{construction}

The assumptions guarantee that the previous 
Assumption~\ref{asm:bigfattreesandcreatures} is satisfied (in particular
that there are compound creatures with norm $\mdn$,
and that $\mathbb Q\neq \emptyset$).

By induction, we immediately get the following (which is the reason for the name
``maxposs'').

\begin{fact} \label{fact:p.pruned} Let $p$ be pruned.
  Then $|\poss(p,{<}\slu)|\le \maxposs({<}\slu)$ for $\slu\in \sublevels(p)$.
  In particular, $|\poss(p,{<}h)|\le \maxposs({<}h)$ for $h\in\ww^p$.
\end{fact}

The following shows that each $p(\slu)$ is $B(\slu)$-big.
\begin{fact}
\label{fact:bigisramsey} 
    Let $p$ be a pruned condition, and let $\slu$
    be an active sublevel of $p$ (which can be Sacks or subatomic).

   Then whenever $F:   \poss'(p,{=}\slu)\to B(\slu)$ is a coloring, 
   there is a strengthening $q(\slu)$ of $p(\slu)$
   (i.e., either $q(\slu)$ is a subatom stronger than $p(\slu)$, or $q(\slu)$
   is a sequence of Sacks columns such that each one is stronger than the according
   column in $p(\slu)$)
   such that the subatomic norm (or, each Sacks norm) 
   decreases by at most $1$ and such that $F\restriction \poss'(q(\slu))$\footnote{Here $\poss'(q(\slu))$ is either $\poss(q(\slu))$ if $\slu$ is a subatomic sublevel, or the product of the Sacks columns from $q(\slu)$ if $\slu$ is a Sacks sublevel.} is constant.
\end{fact}

As $B(\slu)$ is much larger than $\maxposs({<}\slu)$, we also get a version of
``compound bigness'' (we will not directly use the following version, but we
will use similar constructions). First note that a function $G: \poss'(p,{\le}\slu)\to H({<}\slu)$ may be interpreted as
$F: \poss'(p,{=}\slu)\to H({<}\slu)^Y$ for $Y:=\poss'(p,{<}\slu)$ (cf.~\ref{lem:trivialdd}).
As $|\poss'(p,{<}\slu)|\le \maxposs({<}\slu)$, and $B(\slu)$ is big with respect to
$\maxposs({<}\slu)$ and $H({<}\slu)$, we can use the previous item and strengthen
$p(\slu)$ to make $G$ independent of the possibilities at $\slu$.

Iterating this downwards we get the following.
\begin{fact}
Let $p$ be pruned, and let $\slv < \slu$ be active sublevels of $p$.
\begin{itemize}
\item
   If $G: \poss'(p,{<}\slu)\to H({<}\slv)$ is a coloring, then we can strengthen the $p(\slu')$
   to $q(\slu')$ for $\slv\leq \slu'<\slu$, decreasing all subatomic/Sacks norms (and therefore also all compound norms)
   by at most $1$, such that $G$ restricted to $\poss'(q,{<}\slu)$ 
   only depends on $\poss'(q,{<}\slv)$.
\item
   In particular, if $G: \poss'(p,{<}\slu)\to 2$, then we can strengthen 
   $p$ to $q$ as above such that $G\restriction  \poss'(q,{<}\slu)$ is constant.
\end{itemize}
\end{fact}

\section{Properness, \texorpdfstring{$\omega^\omega$}{omega-omega}-bounding and rapid reading}\label{sec:proper}
\subsection{Bigness, rapid reading from continuous reading}\label{sec:bigness}
(Remark: This section is the straight-forward modification of
\cite[Lemma 1.13]{MR2864397}.)

\begin{definition}\label{def:continuousreading}
\begin{itemize}
   \item
   Let  $\n \tau$ be the name of an ordinal. We say that 
   $\n \tau$ is decided below the sublevel $\slu$
   (with respect to the condition $p$), if 
   $p\curlywedge \eta$ decides the value of $\n\tau$ for 
   all $\eta\in\poss(p,{<}\slu)$; in other words, there is a 
   function $R:\poss(p,{<}\slu)\to  \plaintext{Ord}$
   such that $p\curlywedge \eta \Vdash \n \tau = R(\eta)$ for all
   $\eta\in\poss(p,{<}\slu)$. 
  \item We also write ``$\n \tau$ is decided $<\slu$''; and we write
    ``$\n \tau$ is decided $\le\slu$'' for the obvious concept
    (i.e., ``$\n \tau$ is decided $<\slv$'', where $\slv$ is the 
    successor sublevel of $\slu$).
  \item 
  $p$ \emph{essentially decides} $\n\tau$, if there is some 
  sublevel $\slu$
  such that $\tau$ is decided below $\slu$.
  \item
   Let $\n r$ be the name of an $\omega$-sequence of ordinals.
   We say that a condition $p$ continuously reads $\n r$, 
   if all $\n r(m)$ are essentially decided by $p$.    
   \item
  $p$ rapidly reads $\n r\in 2^\omega$, if, for each sublevel $\slu$, 
  $\n r\restriction H({<}\slu)$ is decided below 
  $\slu$.
  \item
    Let $\Xi_0 \subseteq \Xi$.  We say that $p$ ``reads $\n r$ continuously
    only using indices in~$\Xi_0$'' if $p$ reads $\n r$
    continuously and moreover (using the relevant functions $R$ mentioned above)
    the value of $R(\eta)$ depends only on $\eta \restriction 
    \Xi_0$.  

   In other words:  For every $n$ there exists a sublevel $\slu$ 
   such that $p\curlywedge \eta$ decides the value of $\n r(n)$ for 
   all $\eta\in\poss(p,{<}\slu)$, and whenever $\eta\restriction \Xi_0 = \eta'\restriction \Xi_0$, then 
    $p\curlywedge \eta$  and 
    $p\curlywedge \eta'$  agree on the value of $\n r(n)$. 
  \item  We define the notion ``reads $\n r$ rapidly only using indices in $\Xi_0$'' similarly. 
  \item Instead of  ``only using indices in $\Xi\setminus\Xi_1$'' 
we also write ``not using indices in $\Xi_1$''.
\end{itemize}
\end{definition}

Note that for $X\supseteq \Xili$, a real $\n r$ is read continuously from $X$
iff it exists in the  $\QQQ_X$-extension (cf.~\ref{lem:quotient}).

\begin{remark}
   For a fixed condition $p$, the possibilities (at all sublevels) 
  form an infinite tree in the obvious way.  The set of branches $T_p$ of this 
   tree carries a natural topology.   
   $p$ continuously   reads $\tau$ iff there is a continuous function $F$
   on $T_p$ 
   in the ground model
   such that $p$ forces $\n \tau  =  \tilde F(\bar \yy)$, 
   where $\tilde F$ is the canonical extension of $F$.

  In our case, the tree is finitely splitting, so $T_p$ is compact, and continuous is the same as uniformly continuous.  (Note that the definition above really 
uses a uniform notion of continuity.) 

   Rapid reading corresponds to a form of Lipschitz continuity.

\end{remark}

\begin{lem}\label{lem:strongerstillcont}
\begin{enumerate}
   \item  
If $p$ continuously (or: rapidly) reads $\n r$ and $q\leq p$ 
with $\supp(q)\supseteq \supp(p)$,
then 
$q$ continuously (or: rapidly) reads $\n r$. The same holds
if we add ``only using $\Xi_0$'' or: ``not using  $\Xi_1$''.
\item  If $q\le^* p$, and $\n \tau $ is a name of an ordinal
essentially decided by $p$, then also $q$ essentially decides
$\tau$. 
\end{enumerate}
\end{lem}

\begin{proof}
(1) Intuitively, this is clear: If $q\le p$ and $\eta\in\poss(q,{<}\slu)$
then $\eta$ morally is an element of $\poss(p,{<}\slu)$,
and $q\curlywedge\eta\leq p\curlywedge \eta$. 

The formal proof uses Lemma~\ref{lem:superlemma}.

(2) $p$ forces that $\n \tau $ is decided by a finite case
distinction; so $q$ forces the same. 
\end{proof}

\begin{lem}\label{lem:fewrealsreadcontinuously}
  In $V$,  let $\kappa$ be $\max(\aleph_0,|\Xi_0|)^{\aleph_0}$. Then in the
  extension, there are at most $\kappa$ many 
  reals which are continuously read only using\footnote{More
   formally: reals $r$ such that there is a $p\in G$ 
  and a name $\n r$ such that $p$ continuously reads $\n r$ only using
  $\Xi_0$ and  such that $G$ evaluates $\n r $ to $r$.}
   indices
  in $\Xi_0$.
\end{lem}

\begin{proof} This is the usual ``nice names'' argument:
  Given $p$ continuously reading $\n r$. 
  We can define the obvious name $\n r'$ continuously read by
  $p'=p\restriction \Xi_0$, such that $p$ forces $\n r=\n r'$.
  There are at most $\kappa$ many countable subsets of $\Xi_0$,
  and therefore only $\kappa$ many conditions $p'$ with 
  $\supp(p')\subseteq \Xi_0$. Given such a condition $p'$,
  there are only $2^{\aleph_0}$ many ways to continuously read
  a real (with respect to $p'$).
\end{proof}

We will first show that we can ``densely'' get from continuous reading to rapid reading.
Later we will show that ``densely'' we can continuously read reals.  Both proofs are the obvious modifications of the corresponding proofs in \cite{MR2864397}.

\begin{lem}\label{lem:rapidreading}
  Assume that $p$ continuously reads $\n r\in 2^\omega$,
  then there is a $q\leq p$ rapidly reading~$\n r$.   

  The same is true if we add  ``only using $\Xi_0$''. 	
\end{lem}

\begin{proof}
Without loss of generality we can assume that $p$ is pruned
(use Lemmas~\ref{lem:variousprunedcrap} and~\ref{lem:strongerstillcont}).

For a sublevel $\slu$, we set
\begin{equation}\label{eq:lkwjetoiw}
  \slvdec(\slu)\text{  is the maximal sublevel
such that }\n r\restriction H({<}\slvdec(\slu))\text{ is decided below }\slu,
\end{equation}
The function $\slvdec$ is nondecreasing; and
continuous reading implies that $\slvdec$ is an unbounded function
on the sublevels; but $\slvdec$ can generally grow very slowly.
($p$ ``rapidly reads $\n r$'' would mean that $\slvdec(\slu)\ge\slu$
for all $\slu$.)

For all sublevels $\slv\le \slu$ we set
\begin{equation}
\n x^{\slu}_{\slv}:= \n r\restriction(H({<}\min(\slv,\slvdec(\slu))))
\text{ (which is by definition decided below $\slu$).}
\end{equation}
There are at most 
\begin{equation}\label{eq:Cbound}
   2^{H({<}\slv)}
\end{equation} many possibilities for $\n x^{\slu}_{\slv}$, as 
$H(({<}\min(\slv,\slvdec(\slu))))\le H({<}\slv)$.

\begin{dscription}
\item[1]
For now, fix a Sacks  sublevel $\slu=(\ell,-1)$ with $\ell\in\ww^p$.

We will define (or rather: pick) by 
downwards induction on $\slu'\in\sublevels(p)$, $\slu'\le \slu$,
objects
$\cd^{\slu}_{\slu'}$, 
which are 
either a sequence of Sacks columns (if $\slu'$ is Sacks) or
a subatom; 
and functions $\psi^{\slu}_{\slu'}$. 
\item[1a]
For $\slu'= \slu$, we set $\cd^\slu_\slu :=p(\slu)$,
i.e., the sequence of Sacks columns of level $\ell$. 
We let $\psi^\slu_\slu$ be the function
with domain $\poss(p,{<}\slu)$ which 
assigns to each $\eta\in \poss(p,{<}\slu)$
the corresponding value of $\n x^\slu_\slu$.

In other words:
$p\curlywedge \eta$ forces that $\n x^{\slu}_{\slu}=\psi^{\slu}_{\slu}(\eta)$ for each $\eta\in \poss(p,{<}\slu)$.
\item[1b] We continue the induction on $\slu'$.  For now, we write 
   $\cd':= \cd^\slu_{\slu'}  $,
   $\psi':= \psi^\slu_{\slu'}  $, and 
   $x':= x^\slu_{\slu'}  $.
\begin{itemize}
  \item If $\slu'$ is subatomic,
    then we choose for $\cd'$ a subatom stronger than the active subatom $p(\slu')$,
    with $\nor(\cd')\ge \nor(p(\slu'))-1$.
  \item Otherwise, i.e., if $\slu'=(\ell',-1)$ is Sacks with $\ell'\in \ww^p$, 
    set $S:=\supp(p,\ell')\cap\Xisk\neq \emptyset$. 
    Then $\cd'$ is a sequence $(\cs'_\xi)_{\xi\in S}$
    of Sacks columns 
    such that $\cs'_\xi\subseteq p(\xi,\ell')$ and 
    $\sacksnor(\cs'_\xi)\ge \sacksnor(p(\xi,\ell))-1$
    for each $\xi\in S$.
  \item 
    $\psi'$
    is a function with domain
    $\poss(p,{<}\slu')$ such that
\proofclaim{eq:bla2342}{modulo
    $(\slv:\, \slu'\le \slv <  \slu)$, each $\eta\in\poss(p,{<}\slu')$
    decides $\n x'$ to be $\psi'(\eta)$,} by which we mean: 
\proofclaimnl{
    $p\curlywedge\eta$   
    forces the following: If the generic $\bar\yy$ is compatible with
    $\cd^{\slu}_{\slv}$ for each sublevel $\slv\in\sublevels(p)$ with $\slu'\le \slv<\slu$, then
    $\n x' =\psi'(\eta)$.
} 
  \end{itemize}
  How can we find such $\cd'$, $\psi'$?

  Let $\slu''$ be the smallest element of $\sublevels(p)$ 
    above $\slu'$.
    By induction we already know that
    $\psi'':=\psi^{\slu}_{\slu''}$ 
    is a function with domain
    $\poss(p,{<}\slu'')$ such that 
    modulo
    $(\slv:\, \slu''\le \slv< \slu)$
    each $\eta\in\poss(p,{<}\slu'')$
    decides $\n x'':=\n x^{\slu}_{\slu''}$ to be $\psi''(\eta)$. 
    
    Let $\psi''_0(\eta)$ be the restriction of $\psi''(\eta)$ to 
    $H({<}\min(\slu',\slvdec(\slu)))$, i.e., 
    $\psi_0''$  maps each $\eta\in  \poss(p,{<}\slu'')$ to a 
restriction of $\n x''$, which   
    is a potential value for $\n x'$.
   
    We can write\footnote{cf.~\ref{lem:trivialdd}}
    $\psi_0''$ as a function $A\times B\to C$,
    for $A:=\poss(p,{<}\slu')$, $B=\poss(p,{=}\slu')$ and
    $C$ is the set of possible values of $\n x'$, which has, according
to~\eqref{eq:Cbound}, size $\le2^{H({<}\slu')}$.
    This defines a function from $B$ to 
$C^A$, a set of cardinality $ \le 2^{\maxposs({<}\slu')\cdot H({<}\slu')}$;
  so according to~\eqref{eq:cc4b} and Fact~\ref{fact:bigisramsey} 
  we can use bigness at sublevel $\slu'$ to find 
  $\cd'$ such
  that $\psi_0''$ does not depend on sublevel $\slu'$.
  This naturally defines $\psi'$.

\item[2] We perform this downwards induction from
each Sacks sublevel $\slu$ of~$p$.
So this defines for each $\slv<\slu$ in $\sublevels(p)$ the objects
$\cd^{\slu}_{\slv}$ and $\psi^{\slu}_{\slv}$, satisfying (which is just
~\ref{eq:bla2342}):
\proofclaim{eq:bla2342b}{modulo
    $(\slv':\, \slv\le \slv'\le \slu)$, each $\eta\in\poss(p,{<}\slv)$
    decides $\n x^{\slu}_{\slv}$ to be $\psi^{\slu}_{\slv}(\eta)$.}  
   Also, the norms of each Sacks column and subatom drop by at most 1.
\item[3]
Note that for a given $\slv$, there are only finitely many possibilities
for 
$\cd^{\slu}_{\slv}$ and $\psi^{\slu}_{\slv}$.
So by König's Lemma 
there is a sequence $(\cd^{*}_{\slv},\psi^{*}_{\slv})_{\slv\in\sublevels(p)}$
such that 
\proofclaim{eq:bla2342c}{for each sublevel $\slv'$ there is an $\slu>\slv'$
such that $\cd^{\slu}_{\slv''}=\cd^{*}_{\slv''}$ and $\psi^{\slu}_{\slv''}=\psi^{*}_{\slv''}$ for all $\slv''\le \slv'$.}
\item[4]
We now construct $q$ by replacing the subatoms and Sacks columns in $p$ at
sublevel~$\slv$ with $\cd^{*}_{\slv}$ (for each $\slv\in\sublevels(p)$).
So $q$ has the same $\ww$ as $p$, the same supports, the same halving 
parameters and the same trunk; and all norms decrease by at most~$1$.
We claim that $q$ rapidly reads $\n r$, i.e., we claim that
each $\eta\in\poss(q,{<}\slv)$ decides $\n r\restriction H({<}\slv)$.
\item[5]
Pick a $\slv'>\slv$ such that $\slvdec(\slv')\ge \slv$.
According to the definition~\eqref{eq:lkwjetoiw}, this means that
$\n r\restriction H({<}\slv)$ is decided below $\slv'$.
Then pick $\slu>\slv'$ as in~\eqref{eq:bla2342c}.
Recall (from~\eqref{eq:bla2342b}) that $\n x^{\slu}_{\slv}$ is decided below
$\slv$ by $\psi^{\slu}_{\slv}$ modulo the sequence $(\cd^{\slu}_{\slv''}:\, 
\slv\le \slv''< \slu)$.
Recall that $\slvdec(\slv')\ge\slv$ and $\slu\ge\slv'$.
So $\min(\slvdec(\slu),\slv)=\slv$, therefore
$\n x^{\slu}_{\slv}=\n r\restriction H({<}\slv)$.
And, since $\slvdec(\slv')\ge\slv$, $\n x^{\slu}_{\slv}$ is decided already (by the original condition $p$)
below $\slv'$. So we can omit the assumption that the generic
is compatible with $\cd^{\slu}_{\slu''}$
for any $\slv'\le \slu''< \slu$ and still correctly compute 
$\n x^{\slu}_{\slv}$ with $\psi^{\slu}_{\slv}$ modulo
$(\cd^{\slu}_{\slu''}:\, \slv\le \slu''<\slv')$.

In particular,
$\psi^{\slu}_{\slv}=\psi^{*}_{\slv}$ correctly
computes $\n x^{\slu}_{\slv}=\n r\restriction H({<}\slv)$ 
modulo $q$ (since $q$ contains $\cd^{\slu}_{\slu''}=\cd^{*}_{\slu''}$
for each $\slu''<\slv'$.)\qedhere
\end{dscription}
\end{proof}

\subsection{Halving and unhalving}\label{ss:halving}

We will now, for the first and only time in this paper, make use of the halving
parameter. We will show how to ``halve'' a condition $q$ to $\half(q)$, and
then ``unhalve'' any $r\leq \half(q)$ with ``positive norms'' to some $s\leq^*
q$ with ``large norms''.  This fact will only be used in the next section, to
show pure decision.

We repeat the definition of the lim-inf norm from~\ref{def:norm}:
\[
  \linor^{\maxposs({{<}\mdn})}(\cc,h)=\frac{\log_2(N^\cc_h-d(\cc,h))}{\maxposs({<}\mdn)}
  \quad\text{for }
  N^\cc_h:=\min \{\nor(\cc(\xi,h)):\, \xi\in \suppord \cap\Xi_\typeli\}.
\]
If we increase $d:=d(\cc,h)$ to 
\begin{equation}\label{eq:dprime}
d':=d+\frac{N^\cc_h-d}2=\frac{N^\cc_h+d}2,
\end{equation}
then
the resulting lim-inf norm (hence also the compound norm) 
decreases by at most $\nicefrac{1}{\maxposs({<}\mdn)}$.

\begin{definition}
  Given a compound creature $\cc$, we set $\half(\cc)$ to
  be the same compound creature as $\cc$, except that we replace
  each halving parameter $d(h)$ by the $d'(h)$ described above.
\\
  So $\nor(\half(\cc))\ge \nor(\cc)-\nicefrac{1}{\maxposs({<}\mdn)}$.
\\
  Similarly, given a condition $p$ and a level $h\in\ww^p$,
  we set $\half(p,{\ge} h)$ to be the same as $p$,
  except that all compound creatures $p(\ell)$ for $\ell\ge h$
  are halved (and nothing changes below $h$).
\end{definition}

The point of halving is the following:
Assume that the norms in $q$ are ``large'' and that $r\leq \half(q)$ has norms
that are just $>0$.  Then there is an ``unhalved version'' of $r$, an $s\leq
q$, such that the norms in $s$ are ``large'' and still $s\leq^* r$.

In more detail:
\begin{lem}[Unhalving]\label{lem:unhalving}
  Fix 
  \begin{itemize}
    \item $M\in \mathbb R$,
    \item a condition $q$,
    \item $h\in\ww^q$ such that
      $\nor(q(\ell))\ge M$
      for all $\ell\ge h$ in $\ww^q$,
    \item a condition $r\leq \half(q,{\ge} h)$ such that
      $\min(\ww^r)=h$ and $\nor(r(\ell))>0$
      for all $\ell$ in $\ww^r$.
  \end{itemize}
  Then there is an $s$ such that
  \begin{enumerate}
    \item $s\le q$.
    \item $h=\min(\ww^s)$. 
    \item Writing $h_1$ for the  successor of $h$ in $\ww^s$, we have $\nor(s,\ell)\ge M$ for all $\ell\ge h_1$ in $\ww^s$. 
    \item\label{item:wlog} $\supp(s,h)=\supp(q,h)$.
    \item  Above $h_1$, $s$ is the same as $r$, i.e.:
       \begin{itemize}
             \item 
             For $\ell\ge h_1 $:
             $\ell\in\ww^s$ iff $\ell\in\ww^r$, and for such $\ell$ we have
	          $s(\ell)=r(\ell)$.  
	  \item  The trunks agree above $h_1$. 
             \item So in particular, $\supp(s)=\supp(r)$, and the norms do not 
             change above $h_1$ (hence are $\ge M$). 
       \end{itemize}
    \item $\nor(s,h)\ge M-\nicefrac{1}{\maxposs({<}h)}$.
    \item $\poss(s,{<}h_1)\subseteq \poss(r,{<}h_1)$.
  \end{enumerate}
\end{lem}
Note that (5) and (7) implies $s\le^* r$ (by~\ref{lem:martinwantsthis}). So
(by~\ref{lem:strongerstillcont}),
if $r$ essentially decides a name $\n\tau$, then so does $s$.

\begin{proof}
    First fix $h_0\in\ww^r$ bigger than $h$ such that
    $\nor(r(\ell))>M$ for all $\ell\ge h_0$.
    Let $h_1$ be the $\ww^r$-successor
    of $h_0$.

    We set $\ww^s:=\{h\}\cup \ww^r\setminus h_1$. 
    The trunk $t^s$ will extend $t^r$ (and will contain some additional 
    in the ``area'' $[h,h_1)\times (\supp(r,h_0)\setminus\supp(q,h))$).

    For $\ell\ge h_1$ in $\ww^s$, we set $s(\ell):=r(\ell)$. 

    We set $\cd_0:=\glue(r(h),\dots,r(h_0))$,
    and choose arbitrary $r$-compatible elements for the new parts of the trunk $t^s$.
    We then let $\cd_1$ be the restriction of $\cd_0$ to $\supp(q,h)$
    (again, choosing $r$-compatible elements for the new parts of the trunk $t^s$).

    Now we construct $\cd$ from $\cd_1$ by replacing each halving parameter
    $d^{\cd_1}(k)$ by $d^q(k)$ (for all $h\le k<h_1$).
    We set $s(h)=\cd$. This completes the construction of the condition $s$.

    It is straightforward to check that the requirements are satisfied. We 
    will show
    $\nor(s(h))=\nor(\cd)\ge M-\nicefrac1{\maxposs({<}h)}$:

    The norm of $\cd$ is the minimum of several subnorms:
    \begin{itemize} 
      \item The width norm, which is $\ge M$, 
        as $\supp(\cd)=\supp(q,h)$ and  $\nor(q(h))\ge M$.
      \item The Sacks norms of the Sacks columns $\cd(\xi)=r(\xi,h)\otimes\cdots\otimes r(\xi,h_0)$ for $\xi\in\supp(\cd)\cap\Xisk$:
\begin{align*}
    \sacksnor(\cd(\xi)) &= \sacksnor^{B(h),h}(\cd(\xi))\ge
    \sacksnor^{B(h),h}(r(\xi,h_0)) \ge
     \\
     &\ge \sacksnor^{B(h_0),h_0}(r(\xi,h_0))
     = \sacksnor(r(\xi,h_0))
     \ge M,
\end{align*}
      by~\ref{lem:martinsgloriouscorollary}.
      \item The lim-sup norms:
$\lsnor(\cd,\xi)\ge\lsnor(r(h_0),\xi)\ge M$.
      \item So it remains to deal with the lim-inf norm.
    \end{itemize}
    So we have to
    show that for $h\le \ell<h_1$,
\begin{equation}\label{eq:jheiur}
  \linor^{\maxposs({{<}h})}(\cd,\ell)=\frac{\log_2(N^\cd_\ell-d(\cc,\ell))}{\maxposs({<}h)}\ge M-\frac{1}{\maxposs({<}h)},
\end{equation}
where
  $N^\cd_\ell:=\min \{\nor(\cd(\xi,\ell)):\, \xi\in \supp(\cd) \cap\Xi_\typeli\}$.

Recall $d'(\ell)$ as defined in~\eqref{eq:dprime}.
These are the halving parameters used in $\half(q)$, and since 
$r\leq \half(q)$ we know
that $d^r(\ell)\geq d'(\ell)$ (where $d^r$ are the halving parameters
used in $r$).

Let $m\in \ww^r$ correspond to $\ell$
(i.e., $m\le\ell$ and $\ell$ less than the $\ww^r$-successor of $m$).
As $\nor(r(m))>0$,
we know that
\[
  0<\linor^{\maxposs({<}m)}(r(m),\ell)\le
\linor^{\maxposs({<}h)}(r(m),\ell)\le
\frac{\log_2(N^\cd_\ell-d^r(\ell))}{\maxposs({<}h)}
\]
for $N^\cd_\ell$ as above.\footnote{The last $\le$ holds
since $r(m)$ contains the same subatoms as $\cd$ 
(on the common support; however the support of $r(m)$ may be larger,
therefore the last inequality is not necessarily an equality).}

Fix any $\xi\in\supp(q,h)\cap \Xili$. 
Let $k\in\ww^q$ correspond to $\ell$ (as above), and set 
$\cc=q(k)$. The inequality above gives
$0< \log_2(\nor(\cd(\ell,\xi))-d^r(\ell))$, which implies
\[
  \nor(\cd(\xi,\ell))>d^r(\ell)\ge d'(\ell)=d^q(\ell)+\frac{N^\cc_\ell-d^q(\ell)}{2}.
\]
So $\nor(\cd(\ell,\xi))-d^q(\ell)>\frac{N^\cc_\ell-d^q(\ell)}{2}$ for all $\xi$,
and so
\begin{align*}
\linor^{\maxposs({<}h)}(\cd,\ell)&\ge
\linor^{\maxposs({<}h)}(\cc,\ell)-\frac1{\maxposs({<}h)}\\
&\ge 
\linor^{\maxposs({<}k)}(\cc,\ell)-\frac1{\maxposs({<}h)}\\
&\ge 
M-\frac1{\maxposs({<}h)}.\qedhere
\end{align*}
\end{proof}

\subsection{Halving and pure decision}\label{ss:puredecision}
(Remark: This section is a straightforward modification of
\cite[Lemma 1.17]{MR2864397}.)

\begin{lem}\label{lem:blablbagh}
Suppose that $\n \tau$ is a name for an element of $V$, that $p_0
\in \mathbb Q$, that $M_0 \in \ww^{p_0}$ and $n_0 \geq 1$ are such
that $\nor ( p_0 ( h ) ) \geq n_0 + 2$ for all $h \in \ww^{p_0} \setminus M_0$.
Then there is a condition~$q$ such that:
\begin{itemize}
\item $q \leq p_0$.
\item $q$ essentially decides $\n \tau$.
\item Below $M_0$, $q$ and $p_0$ are identical,\footnote{$\supp(q)$ can be larger than $\supp(p)$, so below $M_0$ there will be new parts of the trunk $t^q$.}
   i.e.: $\ww^q \cap M_0 = \ww^{p_0} \cap M_0$ and
$q(h)=p_0 (h)$ for all $h \in \ww^q \cap M_0$.
\item $\nor ( q(h) ) \geq n_0$ for all $h \in \ww^q \setminus M_0$.
\end{itemize}
\end{lem}

\begin{proof}
We may assume that $p_0$ is pruned. 
Our proof will consist of several steps:

\noindent \emph{1. Using halving; the mini-steps.}

Suppose that we are given $p \in \mathbb Q$, $M \in \ww^p$, and $n \geq 1$ such that $\nor ( p(h) ) > n$ for all $h \in \ww^p \setminus M$.
We show how to construct an extension of $p$, denoted $r(p,M,n)$.

First enumerate $\poss ( p , {<}M )$ as $(\eta^1 , \ldots , \eta^m)$.
Note that $m \leq \maxposs ({<} M )$.
Setting $p^0 = p$, we inductively construct conditions $p^1 , \ldots , p^m$ 
and the auxiliary conditions $\tilde p^1,\dots,\tilde p^m$ so that for each $k < m$
the following holds:
\begin{enumerate}
\item $\tilde p^{k+1}$ is $p^k$ where we replace everything below
  $M$ (and in $\supp(p)$) with $\eta^{k+1}$. 
  \\
  Remarks:
  \begin{itemize}
  \item
  By (3) below, we will get $\min( \ww^{\tilde p^{k+1}})  = M$.
  \item
  If $k=0$, then $\tilde p^1$ is just $p\wedge \eta^1$.
  But for $k>0$, $\eta^{k+1}$ will not be in $\poss(p^k,{<}M)$,
  so we cannot use the notation $\tilde p^{k+1}=p^k\wedge \eta^{k+1}$.
  \item
  Note that generally $\supp(p^k)$ will be larger than $\supp(p)$,
  so we do not replace the whole trunk below $M$ by $\eta^{k+1}$,
  but just the part in $\supp(p)$.
  \end{itemize}
\item $p^{k+1} \leq \tilde p^{k+1}$.  Note that we do not have $p_{k+1}\le
  p_k$, for trivial reasons: their trunks are incompatible. 
\item $\min( \ww^{p^{k+1}})  = M$.
\\ Remarks:
\begin{itemize}
  \item So by strengthening $\tilde p^{k+1}$ to $p^{k+1}$, we do not increase the overall trunk-length $\min(\ww)$.
  \item Note that we do not assume that $\ww^{p^{k+1}}=\ww^{p^{k}}\setminus M$, i.e., generally the $\ww$-sets will become
    thinner due to gluing.
\end{itemize}
\item $\supp ( p^{k+1} , M ) = \supp ( p , M )$.
\begin{itemize}
  \item Remark: This only holds at level $M$: Generally, $\supp ( p^{k+1})$ will be larger than $\supp ( p^{k})$.
\end{itemize}
\item $\nor ( p^{k+1} , h ) > n - \frac{k+1}{\maxposs({<}M)}$ for all $h \in \ww^{p^{k+1}} \setminus M$.
\item One of the following two cases holds:
\begin{itemize}
\item \textbf{(decide)} $p^{k+1}$ essentially decides $\n \tau$.
\item \textbf{(halve)} $p^{k+1} = \half(\tilde p^{k+1},{\ge}M)$.
\end{itemize}
More explicitly: If the deciding case is possible, then we use it.
Only if it is not possible, we halve.
\end{enumerate}

We then define $r = r( p , M , n )$ as follows:
Below $M$, $r$ is identical to $p$; and
above (including) $M$, $r$ is identical to $p^m$ (the last one of the $p^k$ constructed above).
In more detail:
\begin{itemize}
\item $\ww^r = ( \ww^p \cap M ) \cup ( \ww^{p^{m}} \setminus M )$;
i.e., below $M$ the levels of $r$ are the ones of $p$;
and above (including) $M$ the levels of $r$ are the ones of $p^m$.
\item $r ( h ) = p ( h )$ for all $h \in \ww^r \cap M$;
\item $r ( h ) = p^m ( h )$ for all $h \in \ww^r \setminus M$;
\item  This determines the domain of  $t^r$;
   and we set  $t^r $ to be $t^{p^m} $ restricted to this domain. 
\end{itemize}

$r=r(p,M,n)$ has the following properties:
\proofclaim{eq:blubblubb}{
\begin{itemize}
  \item $r\in \QQQ$, $r\le p$.
  \item $\nor(r(\ell))> n-1$ for all $\ell\ge M$ in $\ww^r$.
  \item If $\eta\in\poss(r,{<}M)$ and if there is a $s\leq r\wedge\eta$
    such that $s$ essentially decides $\n\tau$,
    $\min(\ww^s)=M$ and $\nor(s(\ell))>0$
    for all $\ell\ge M$ in $\ww^s$, then $r\wedge\eta$
    essentially decides $\n\tau$.
\end{itemize}}

\begin{proof}[Proof of \eqref{eq:blubblubb}]
 $\eta$ extends some $\eta^{k+1}\in\poss(p,{<}M)$;
so $s\le r\wedge\eta \le p^{k+1}\le \tilde p^{k+1}$. 
All we have to show is that $p^{k+1}$ was constructed
using the ``decide'' case. Assume towards a contradiction
that the ``halve'' case was used. 
Then $s$ is stronger than $\half(\tilde p^{k+1},{\ge}M)$,
so we can unhalve it (using Lemma~\ref{lem:unhalving})
to get some $s'\leq \tilde p^{k+1}$ with large norm
such that $s'\leq^* s$, showing that we could have used the
``decide'' case after all.
This ends the proof of~\eqref{eq:blubblubb}.
\end{proof} 

\noindent \emph{2. Iterations of the mini-steps; the condition $q$.}

Given $p_0 , M_0 , n_0$ as in the statement of the Lemma, we inductively construct conditions $p_k$ and natural numbers $M_k$ for $k \geq 1$.
Given $p_k$ and $M_k$, our construction of $p_{k+1}$ and $M_{k+1}$ is as follows:
Choose $M_{k+1} \in \ww^{p_k}$ bigger than $ M_k$
such that
\[
\nor ( p_k (h) ) > k + n_0 + 3\text{ for all }h \in \ww^{p_k} \setminus M_{k+1}.
\]
Then set $p'_{k+1} = r( p_k , M_{k+1} , k + n_0 + 3 )$,
and construct $p_{k+1}$ by gluing together everything between (including) $M_k$ and (excluding) $M_{k+1}$.

The sequence of conditions $( p_k )_{k\in\omega}$ converges to a condition of $\QQQ$, which we will denote by $q$.
Note that $r\leq q$ implies that $\ww^r$ is a subset of $(\ww^{p_0}\cap M_0)\cup\{M_0,M_1,M_2,\dots\}$ (as we have glued everything between each $M_i$ and $M_{i+1}$).

It is clear that $q \leq p_0$, and that $\nor ( q , h ) > n_0 + 1$ for all $h \in \ww^{q} \setminus M_0$.

We will later show that $q$ essentially decides $\n\tau$ (thus proving the
lemma).

The following property will be central:
\proofclaim{eq:blabkf}{Assume that $\eta\in\poss(q,{<}M_\ell)$ 
for some $\ell\in\omega$, and $r\leq q\wedge \eta$ essentially decides $\n\tau$
and $\min(\ww^r)=M_\ell$ and each $r(m)$ has norm $>1$ for each $m\in\ww^r$.
\\
Then $q\wedge\eta$ essentially decides $\n\tau$.
}

Proof of~\eqref{eq:blabkf}: $\eta$ (or rather: a restriction of $\eta$ to
$\supp(p)$) was considered as a possible trunk $\eta^{k+1}$ in the ``mini-step''
when constructing $r(p_{\ell-1},M_{\ell},\ell+n_0+2)$. So we can
use~\eqref{eq:blubblubb}.  This ends the proof of~\eqref{eq:blabkf}.

\noindent \emph{3. Using bigness to thin out $q$ to prove essentially deciding.}

We now repeat the construction of the proof of Lemma~\ref{lem:rapidreading},
but this time we do not homogenize on the potential values of some $\n x$,
but rather on whether $q\curlywedge \eta$ essentially decides $\n \tau$
or not.

For now, fix a sublevel $\slu=(\ell,-1)$ above $(M_0,-1)$ with $\ell\in\ww^q$.
\begin{itemize}
\item We set $\cd^\slu_\slu$ to be the 
collection of Sacks columns $q(\slu)$.
We set $B^\slu_\slu$ to be the set of $\eta\in\poss(q,{<}\slu)$
such that $q\curlywedge \eta$ essentially decides $\n\tau$.
\item By downwards induction on $\slu'\in\sublevels(q)$, $(M_0,-1)\le \slu'<\slu$, we 
construct $\cd^\slu_{\slu'}$ and $B^\slu_{\slu'}$ such that the following is
satisfied:
\begin{itemize}
  \item $\cd^\slu_{\slu'}$ is a strengthening of the  subatom (or: collection of Sacks columns)  $q(\slu')$, the norm decreases by at most 1.
  \item (Homogeneity) $B^\slu_{\slu'}$ is a subset of $\poss(q,{<}{\slu'})$,
   such that
   for  each $\eta\in B^\slu_{\slu'}$ and each $\nu\in\poss(\cd^\slu_{\slu'})$
   $\eta^\frown \nu\in B^\slu_{\slu'+1}$; and 
   analogously for each $\eta\in \poss(q,{<}{\slu'})\setminus B^\slu_{\slu'}$ and each $\nu\in\poss(\cd^\slu_{\slu'})$,
    $\eta^\frown \nu\notin B^\slu_{\slu'+1}$.
\end{itemize}
(Just as in the case of rapid reading, we can find these objects using bigness:
Assume that $\slu''$ is the $\sublevels(q)$-successor of $\slu'$; by induction
there is a function $F$ which maps each $\eta\in \poss(q,{<}{\slu''}$ to $\{\in
B,\notin B\}$; we thin out $q(\slu')$ to $\cd^\slu_{\slu'}$ such that for each
$\nu\in\poss(q,{<}\slu')$ each extension of $\nu$ compatible with
$\cd^\slu_{\slu'}$ has the same $F$-value $F^*(\nu)$; this in turn defines
$B^\slu_{\slu'}$.)
\item 
Assume that $\slv<\slu$ as above, that $\eta\in\poss(q,{<}\slv)$, that
$q\curlywedge \eta$ essentially decides $\n \tau$ and 
that $\eta'\in \poss(q,{<}\slu)$ extends $\eta$.
Then trivially $q\curlywedge \eta'$ also essentially decides $\n\tau$.
So we get:
\proofclaim{eq:hrqjwq}{If $q\curlywedge \eta$ essentially decides $\n \tau$
for $\eta\in\poss(q,{<}\slv)$, then $\eta\in B^\slu_\slv$ 
for any $\slu>\slv$.}
\item We now show the converse:
\proofclaim{eq:fefwe}{Whenever $\eta\in B^\slu_{\slu'}$ for some sublevel 
$\slu'$ of the form $(M_{\ell'},-1)\leq \slu$ for some $\ell'$, then
$q\wedge \eta$ essentially decides $\n\tau$.}
(Equivalently: $q\curlywedge \eta$ essentially decides $\n\tau$, as
$q\curlywedge \eta=^*q\wedge \eta$.)
Proof: We can modify $q$ to a stronger condition $r$ 
using $\eta$ as trunk and using 
$\cd^\slu_{\slu''}$ for all $\slu'\le \slu''\le\slu$.
Any $\eta'\in \poss(r,{<}\slu)$ is in $B^\slu_{\slu}$,
so $q\curlywedge \eta'=^*r \curlywedge\eta'$ essentially decides 
$\n\tau$. So $r$ essentially decides $\n\tau$. Also,
each compound creature in $r$ has norm $>1$, so we can use~\eqref{eq:blabkf}.
This ends the proof of~\eqref{eq:fefwe}.
\item
So to show that $q$ essentially decides $\n\tau$, it is enough
to show that for all $\eta\in\poss(q,{<}(M_0,-1))$
there is a $\slu$ such that $\eta\in B^\slu_{(M_0,-1)}$.
\item
As in the rapid reading case, we choose an ``infinite branch''
$(\cd^*_{\slv},B^*_{\slv})$. I.e.: for each $\slv'$ there is a $\slu>\slv'$
such that $(\cd^\slu_{\slv},B^\slu_{\slv})=( \cd^*_{\slv},B^*_{\slv})$ for each
$\slv\le \slv'$.
This defines a condition $q_1\leq q$. 
\item
To show that $q $ essentially decides $\n\tau$,
it is enough to show $\eta\in B^*_{(M_0,-1)}$ for all
$\eta\in\poss(q,{<}M_0)=\poss(q_1,{<}(M_0,-1))$.

So fix such an $\eta$. Find any $r\leq q_1\wedge \eta$ deciding $\n\tau$.
Without loss of generality, $\min(\ww^r)=M_{\ell}$ for some $\ell$,
and each compound creature in $r$ has norm at least $1$.
Let $\eta'>\eta$ be the trunk of $r$ (restricted to $\supp(q)$
and $M_{\ell}$).
According to~\eqref{eq:blabkf}, $q\wedge \eta'$ essentially decides $\n\tau$.

Pick some $\slu>(M_\ell,-1)$ such that $(\cd^\slu_{\slv},B^\slu_{\slv})=
(\cd^*_{\slv},B^*_{\slv})$ for each $\slv\le (M_\ell,-1)$.
According to~\eqref{eq:hrqjwq}, $\eta'\in B^*_{\slv}$.
By homogeneity, $\eta\in B^*_{(M_0,-1)}$.
So according to~\eqref{eq:fefwe}, $q\wedge\eta$ essentially decides $\n\tau$.
\qedhere
\end{itemize}
\end{proof}

\subsection{Properness, \texorpdfstring{$\omega^\omega$}{omega-omega}-bounding, rapid reading, no randoms}

A standard argument now gives the following:
\begin{thm}
$\mathbb Q$ satisfies (the finite/$\omega^\omega$-bounding version of) Baumgartner's Axiom~A, in particular it is proper and $\omega^\omega$-bounding and (assuming CH in the ground model) preserves all cofinalities.
Also, $\mathbb Q$ rapidly reads every $\n r\in 2^\omega$.
\end{thm}

\begin{proof}
We already know that we can rapidly read each real if we can
continuously read it.

We define $q\leq_n p$ as: $q\leq p$ and there is an $h\in\ww^q$, $h\ge n$,
such that $q$ and $p$ are identical below $h$ and $\nor(q(\ell))>n$
for all $\ell\ge h$.

It is clear that any  sequence $p_0\ge_0 p_1\ge_1 p_2\ge_2\dots$ has a limit;
and Lemma~\ref{lem:blablbagh} shows that for any name $\n\tau$ of an ordinal, $n\in\omega$
and $p\in\QQQ$, there is a $q\le_n p$ such that modulo $q$ there
are only finitely many possibilities for $\n\tau$.
\end{proof}

Rapid reading gives us:
\begin{lem}\label{lem:norandoms}
  Every new real is contained in a ground model null set, i.e., no random reals are added.
So assuming CH in the ground model, we will have $\cov(\NULL)=\al_1$ in the extension.
\end{lem}
\begin{proof}
  Let $\n r$ be the name of an element of $2^\omega$ and $p$
  a condition. Let $q\leq p$ rapidly read $\n r$.
  So for all
  $\ell\in\ww^q$, $\n r\restriction H({<}\ell)$ is determined by each $\eta\in\poss(q,{<}\ell)$.
  Hence, the set $A^q_\ell$ of possibilities for $\n r\restriction H({<}\ell)$ has size at most $\maxposs({<}\ell)<H({<}\ell)<\nicefrac{2^{H({<}\ell)}}{\ell}$. So $A^q_\ell$ has ``relative 
size'' $<\nicefrac1{\ell}$, and the sequence
$(A^q_\ell)_{\ell\in\omega}$ defines (in the ground model)
the null set 
\[ N=\{ s\in 2^\omega:\, 
    (\forall \ell\in\ww^q)\   s\restriction H({<}\ell)\in  A^q_\ell\}.
\]
  And $q$ forces that $\n r\in N$.
\end{proof}

\section{The specific forcing and the main theorem}\label{sec:specificQ}

\subsection{The forcing}\label{ss:forcing88}

Recall that $\Xi$ is partitioned into $\Xisk$, $\Xili$ and $\Xils$.
We now further partition $\Xils$ into $\Xinn$ and $\Xicn$.
So every $\xi\in \Xi$ has one of the following four types:
\begin{itemize}
  \item type $\typesk$ (Sacks) for $\xi\in\Xisk$,
  \item type $\typecn$ (cofinality null) for $\xi\in\Xicn$,
  \item type $\typenn$ (non null) for $\xi\in\Xinn$, and
  \item type $\typenm$ (non meager) for $\xi\in\Xili$. So $\typenm$ is  the only lim-inf type.
\end{itemize}
Let $\kappa_t$ be the size of $\Xi_t$.

In the inductive construction of $\QQQ$ in Section~\ref{sec:complete.construction}, several assumptions are made in the subatom stages $\slu$.
We will satisfy those assumptions in the
following way:

For each type $t\in\{\typecn,\typenn,\typenm\}$ we assume that we have a family
of subatomic families $\cK'_{t,b}$ indexed by a parameter $b$, such that for
each $b\in\omega$, $\cK'_{t,b}$  is a subatomic family living on some
$\POSS'_{t,b}$ satisfying $b$-bigness. Actually, we will require a stronger
variant of $b$-bigness such that we can find an homogeneous successor subatom
while decreasing the norm not by $1$ but by at most $1/b$. I.e., we require:
\proofclaim{eq:parambig}{For $x\in\cK'_{t,b}$ and $F:\poss(x)\to b$ there is
a $y\le x$ such that $\nor(y)\ge \nor(x)-\nicefrac1b$ and $F\restriction\poss(y)$ is constant.}

Additionally we require that 
\proofclaim{eq:paramlarge}{there is at least one subatom in $\cK'_{t,b}$ with norm $\ge b$}.

Then we set for each subatomic sublevel $\slu=(\ell,j)$
\begin{equation}\label{eq:defbA}
  b(\slu):=B(\slu)\cdot (b(\slv)+1)+1,
\end{equation}
where $\slv$ is the largest\footnote{If $\slu$ is $(0,0)$, the smallest of all subatomic sublevels,
we just set $b(\slu):=B(\slu)$.  By the way, it would be enough to set $b(\slu):=B(\slu)$,
as this sequence would be increasing sufficiently fast, but this would require two extra lines of
calculations.}
subatomic sublevel smaller than $\slu$.
So the sequence $b(\slu)$ is strictly (actually: very quickly) increasing.
According to the definition~\ref{eq:cc4b} of $B(\slu)$, we also get:
\begin{lem}
\label{lem:defb}
  $b(\slu)\ge 2\cdot \maxposs({<}\slu)$, 
and even $b(\slu)\ge 
 2^{
         (\text{number of sublevels below $\slu$})\cdot\maxposs({<}\slu)}$.
\end{lem}
Then we set (for all $\xi\in\Xi_t$)
\[
  \cK_{\xi,\slu}:=\cK'_{t,b(\slu)}.
\]

This way we automatically satisfy requirements~(b) and~(c) 
of item~$(*\ref{item:sa})$ on page~\pageref{item:sa}.  And since there are only four, i.e., finitely many,
types, there is automatically a bound $M$ on $|\POSS_{\xi,\slu}|$ as
required in~(d).

Strong bigness gives us the following property:
\begin{lem}\label{lem:schoeneslemma}
   Let $I$ be a finite set of subatomic sublevels (and thus $I$ is naturally ordered).
   Let $\slv$ be the minimum of $I$.
   For each $\slu\in I$  let $\xi_\slu\in\typenonsk$ and  $x_\slu$ a subatom in $\cK_{\xi_\slu,\slu}$.
   Let $F:\prod_{\slu\in I}\poss(x_\slu)\to b(\slv)$.
   Then there are $y_\slu<x_\slu$ with $\nor(y_\slu)\ge \nor(x_\slu)-\nicefrac1{b(\slu)}$
   and such that $F\restriction \prod_{\slu\in I}\poss(y_\slu)$ is constant.
\end{lem}

\begin{proof}
  We construct $y_\slu$ by downwards induction on $\slu\in I$: 
  Let $\slu'$ be the maximum of $I$, then
  $F$ can be written as function from $\poss(x_{\slu'})$ to
  $b(\slv)^P$, where $P=\prod_{\slu\in I\setminus\{\slu'\}}\poss(x_\slu)$.
  As $|P|$ is less than the number of sublevels below $\slu'$ times
  $\maxposs({<}\slu')$, we get $|P|<b(\slu')$, and thus 
  can use strong bigness to get $y_{\slu'}<x_{\slu'}$.

  Now continue by induction.
\end{proof}

The families $\cK'_{t,r}$ that we will actually use are described in
Section~\ref{ss:cn} for $t=\typecn$, Section~\ref{ss:nn} for $t=\typenn$,
and Section~\ref{ss:nm} for $t=\typenm$.

In addition, we will define there for each $\cK'_{t,b}$ a number $H'(t,{=}b)$, and
in the inductive construction, we define $H$ as follows: 

\begin{defin}
   \label{eq:hfromhprime}
$H({<}(0,-1)):= 3$. 
If $\slu = (\ell,j)$ is a sublevel with immediate predecessor $\slu'$,
we define $H({<}\slu)=H({\le}\slu')$ in cases by: 
\begin{itemize}
\item     
  For a Sacks sublevel $\slu$  (i.e., $j=-1$), 
   $ H({<} \ell )=H({<}\slu) := 2+\ell+\maxposs({<}\ell) +  H({<}\slu') + \max( \{H'(t,{=}b(\slu')):\, t\in\{\typenm,\typenn,\typecn\}\})$. 
\item  For $j=0$: 
  $H({<}\slu):=1+H({<}\slu') + \max(I_{\typesk,\ell}) $. 
\item  For $j > 0$, 
  $H({<}\slu):=1+ H({<}\slu') +  \max \{H'(t,{=}b(\slu')):\, t\in\{\typenm,\typenn,\typecn\}\})$. 
\end{itemize}
\end{defin}

So in particular, if $p$ rapidly reads $\n r$, then for all
$t\in\{\typenm,\typenn,\typecn\}$ and all subatomic sublevels $\slu$
\begin{equation}\label{eq:Hprimerapid}
  \n r\restriction H'(t,{=}\slu)\text{ is decided }\le b(\slu) .
\end{equation}

Note that once we fix the parametrized subatomic families $\cK'_{t,b}$ and
$H'(t,{=}b)$ (and the cardinalities $\kappa_t$), we have specified everything required to construct $\QQQ$, and
$\QQQ$ will satisfy Baumgartner's Axiom~A, will be $\omega^\omega$-bounding, and, assuming CH,
will have the $\al_2$-cc.  We also get rapid reading.

\subsection{The main theorem}
We will show:
\begin{thm}\label{thm:main}
   Assume (in $V$) CH, 
   $\kappa_\typenm \le \kappa_\typenn\le\kappa_\typecn\le\kappa_\typesk$
   and $\kappa_t^{\aleph_0}=\kappa_t$ for $t\in\{\typenm,\typenn,\typecn,\typesk\}$.
   Then there is a forcing  $\mathbb  Q$ which   forces
   \begin{enumerate}
     \item $\cov(\NULL)=\mathfrak d=\aleph_1$,
     \item $\non(\MEAGER)=\cof(\MEAGER)=\kappa_\typenm$,
     \item $\non(\NULL)=\kappa_\typenn$,
     \item $\cof(\NULL)=\kappa_\typecn$,
     \item $2^{\aleph_0}=\kappa_\typesk$.
   \end{enumerate}
   Moreover, $\mathbb Q$ preserves all cardinals and all cofinalities. 
\end{thm}

As mentioned above, we fix disjoint index sets $\Xi_t$ ($t\in \{\typesk,\typecn,\typenn,\typenm\}$)
of respective sizes $\kappa_t$, and we construct $\QQQ$ as described 
above.   
Then the following points are obvious or have already been shown:
\begin{enumerate}
\item[(1)]
     $\mathfrak d=\aleph_1$, since $\QQQ$ is $\omega^\omega$-bounding.
     And it was already shown in Lemma~\ref{lem:norandoms} that no
     random reals are added, so $\cov(\NULL)=\aleph_1$.
\item[(5)]
     If $\alpha\neq \beta\in\Xisk$, then the generic reals at 
     $\alpha$ and $\beta$ are forced to be different, so we have
     at least $\kappa_\typesk$ many reals.
     Every real in the extension is read continuously, so
     by Lemma~\ref{lem:fewrealsreadcontinuously} there are 
     at most $\kappa_\typesk^{\aleph_0}=\kappa_\typesk$ many reals.
\item[($\bullet$)]
     The ``moreover'' part is clear because $\mathbb Q$ satisfies Baumgartner's
       Axiom~A and has the $\aleph_2$-cc.
\end{enumerate}

In the rest of the paper, we will describe the families $K'_{t,b}$
and $H'(t,{=}b)$ and  prove the remaining parts
of the main theorem:
\begin{itemize}
  \item[(2)]
    In ZFC, 
    $\max(\mathfrak d,\non(\MEAGER))=\cof(\MEAGER)$. And
    $\non(\MEAGER)\le \kappa_\typenm$
    is shown in~\ref{lem:nmleknm}, and $\ge$ in~\ref{cor:nmgeknm}.
  \item[(3)]
    $\non(\NULL)\le \kappa_\typenn$
    is shown in~\ref{lem:nnleknn}; and $\ge$ in~\ref{cor:nngeknn}.
  \item[(4)] $\cof(\NULL)\le \kappa_\typecn$
    is shown in~\ref{cor:cnlekcn}; and $\ge$ in~\ref{cor:cngekcn}.
\end{itemize}

\subsection{The Sacks part: \texorpdfstring{$\cof(\NULL)\le \kappa_\typecn$}{cof(N) <= kappa-cn}}\label{sec:SacksB}

We will show that every null set added by $\QQQ$ is contained
in a null set which is already added by the
non-Sacks part.

We will first show that the quotient $\QQQ/\QQQ_{\Xinonsk}$
(in other words: the extension from the universe obtained not using the 
sacks coordinates to the full generic extension)
has the Sacks property.

Recall that the Sacks property states (or, depending on the definition, 
is equivalent to): Every function in $\omega^\omega$ in the extension
is caught by an $(n+2) $-slalom from the ground model. (I.e., there is a function  $S:\omega\to [\omega]^{<\omega}$ in the
ground model  with 
$|S(n)|\le n+2$, and $f(n)\in S(n)$ for all~$n$.)

The Laver property is similar, but only applies to functions $f$ in the extension which are bounded by a ground model function. 

We get
\begin{lem}\label{lem:laversacks}
\begin{enumerate}
\item Laver property is equivalent to:
\begin{quote}
       Whenever $\n r\in 2^\omega$ is in the extension
       and $G:\omega\to\omega$ in the ground model, 
       \\
       then there is in the ground model a tree $T$ 
       (without terminal nodes) such that $\n r\in [T]$ and 
        $|T\restriction 2^{G(n)}|<n+2$ for all $n$.
\end{quote}
\item The Sacks property is equivalent to
  the conjunction of Laver property and $\omega^\omega$-bounding.
\item If an extension has the Sacks property, then any new null set
is contained in an old null set.
\end{enumerate}
\end{lem}

\begin{proof}
For the well known (2) and (3) see, e.g.,~\cite[Theorem 2.3.12]{MR1350295}.
For (1), we only show how to get the Laver property (which is enough for this paper,
and the other direction is similarly easy).

Suppose that $g : \omega \to \omega$ is given.
Enumerate $\{ (n,m) :\, m \leq g(n) \}$ in lexicographic order as $(n_i,m_i)$.
Define a function $G : \omega \to \omega$ by 
\[
\textstyle
G(n) = \min \{ i : n_i > n \} = n+1 + \sum_{k \le n} g(k).
\]
(For convenience we will think of $G(-1) = 0$.)
Note that according to the enumeration given above, every function $r : \omega \to 2$ determines a subset of $\prod_{n < \omega} ( g(n) + 1 )$ by $\{ (n_i,m_i) : r(i) = 1 \}$.
Accordingly, certain functions $r$ induce a function bounded by $g$: those functions $r$ such that given any $n$ there is a unique $m \leq g(n)$ such that $(n,m)$ is in the subset determined by $r$ as described above.
(Equivalently, for each $n$ there is a unique $G(n-1) \leq i < G(n)$ such that $r(i) = 1$.)
Given such an $r$, by $\val(r,n)$ we denote $m_i$ where $G(n-1) \leq i < G(n)$ is such that $r(i) = 1$.

Note that given any function $f$ bounded by $g$ there is a unique function $r_f : \omega \to 2$ (which determines a function bounded by $g$ as described above) such that $\val(r_f,n) = f(n)$ for all $n$.

Suppose that $\n f$ is a name for a function bounded by the ground model function $g$.
Let $\n r_f$ be a name for the function $\omega \to 2$ as described above, and let $T$ be the tree guaranteed
to exist by the assumption (using the function $G$ defined from $g$ above). We may assume that all branches $x$ of 
$T$ determine a function bounded by $g$ as described above.
Now define a slalom $S$ by $S(n)= \{\val(x,n): x\in [T]\}$.   It is clear that $S$ catches $\n f$.
\end{proof}

We now prove our version of the Laver property for the 	quotient. 
As the whole forcing is $\omega^\omega$-bounding, this implies the Sacks
property. 

\begin{lem}\label{lem:blubb}
\begin{enumerate}
\item 
	Assume that $p$ is a condition,
        $\n r\in2^\omega$ a name 
        and $G:\omega\to\omega$ is in~$V$.
	Then there is a $q\leq p$ and a name
        $\n T\subseteq 2^{<\omega}$ (of a tree without terminal nodes) such that:
        $q$ continuously reads
        $\n T$
	not using any Sacks indices; $q$
	forces $r\in[\n T]$; and
        $|\n T\restriction 2^{G(n)}|<n+2$
        for all $n$.
\item   Therefore the quotient $\QQQ/\QQQ_{\Xinonsk}$ has the Laver property (and thus the Sacks property). 
\end{enumerate}
\end{lem}

\begin{proof}
If $G_1(n)\le G_2(n)$ for all $n$, and $\n T$ witnesses the conclusion of the lemma for
$G_2$, then $\n T$ also witnesses the lemma for $G_1$.   So we may without 
loss of generality increase the function $G$ whenever this is convenient. 

We can assume that $p$ rapidly reads~$\n r$, i.e., $\poss(p,{<}n)$ determines
$\n r\restriction H({<}n)$ for all $n\in\ww^p$.

We can then assume that there is a strictly increasing
 function $G'$ such that $G'(n)\in\ww^p$ and
$G(n)=H({<}G'(n))$ for all $n$ (as we can increase $G$).

Also, to simplify notation, we can assume that $\ww^p=\{G'(0),G'(1),\dots\}$.
(Otherwise, just glue.)

So each $\eta\in\poss(p,{<}G'(n))$ determines a value for $\n r\restriction
G(n)$, which we call $R^n(\eta)$.
We view $\eta$ as a pair $(\eta_\typesk,\eta_\typenonsk)$ for
$\eta_t:=\eta\restriction\Xi_t$ for $t\in\{\typenonsk,\typesk\}$.
Accordingly we write $R^n(\eta_\typesk,\eta_\typenonsk)$.
If we fix $\eta_\typesk$, then $R^n(-,\eta_\typesk)$ can be viewed as a name
(for an element of $2^{G(n)}$)
which does not depend on the Sacks part, in the following way:
If there is a $\eta_\typenonsk$ compatible with the generic filter such that
$(\eta_\typenonsk,\eta_\typesk)=\eta\in\poss(p,{<}G'(n))$, then the value is
$R^n(\eta)$ (and otherwise $\emptyset$, say).

Below we will construct $q\le p$ by gluing and by strengthening Sacks columns
(and we will leave the support, the subatoms and the halving parameters
unchanged).

Assume we have such a $q$, and assume that $G'(m_0)<G'(m_1)$ are
consecutive elements of $\ww^q$. Note that
$G'(m_0)<G'(m_0+1)<\cdots<G'(m_1-1)<G'(m_1)$ 
are consecutive elements of $\ww^p$.
Fix $\eta\in\poss(q,{<}G'(m_1))$ and $m_0\le \ell \le m_1$.
Then $\eta$ extends a unique element of $\poss(q,{<}G'(\ell))$,
which we call $\eta^\ell$.
We can then restrict $\eta^\ell$ to the Sacks part:
$\eta^\ell_\typesk:=\eta^\ell\restriction\Xisk$.

Note:
\begin{itemize}
  \item $\eta^\ell_\typesk $ is 
    $\eta$  restricted to the Sacks part and to ``height $G'(\ell)$'',
    i.e.,
    \[
      \eta^\ell_\typesk:=\eta\restriction \Xisk\times (1+\max(I_{\typesk,G'(\ell)})).
    \]
  \item
$q\wedge \eta$ forces that the name  
$R^\ell(-,\eta^\ell_\typesk)$ (which does not depend on
the Sacks part) is evaluated to $\n r\restriction G(\ell)$.
\item
So $q$ forces that  $\n r\restriction G(\ell)$ is an element of
\[
  \n T^\ell := \{ R^\ell(-,\eta^\ell_\typesk):\, \eta\in\poss(q,{<}G'(m_1)) \},
\]
a name not depending on the Sacks part.
\end{itemize}

So it is enough to show that there are few $\eta^\ell_\typesk$, i.e.,
\begin{equation}\tag{$\star_\ell$}\label{eq:starstar}
  |S_\ell| < \ell+2 \text{ for }
S_\ell:=\{ \eta^\ell_\typesk:\, \eta\in\poss(q,{<}G'(m_1))\}.
\end{equation}

We will now by induction on $n$: \
\begin{enumerate}
  \item construct
    $h_n$, where $\ww^q$ will be the set $\{G'(h_0),G'(h_1),\dots\}$;
  \item construct $q$ below $G'(h_n)$,
  \item and show that 
\eqref{eq:starstar} holds for all $\ell\le h_n$.
\end{enumerate}

We set $h_0=0$; so $G'(h_0)=\min(\ww^p)$ and 
$q$ below $G'(h_0)$ has to be 
identical to $p$. And $(\star_0)$ holds as $S_0$ is a singleton.

Assume we have already constructed $h_n$ and $q$ below $G'(h_n)$,
satisfying \eqref{eq:starstar} for $\ell\le h_n$.
\begin{enumerate} 
  \item For any $I$ and $\mathfrak s\subseteq 2^I$, we write
$\sacksnor^{*}(\cs)$ for $\sacksnor^{B(G'(h_n)),G'(h_n)}(\cs)$, see~\ref{eq:sacksnor}.
(I.e., the Sacks norm that would be assigned to a Sacks column
starting at $G'(h_n)$ which has the same $\splitnor$ as $\cs$.)
Let $\Sigma:=\supp(p,G'(h_{n}))\cap\Xisk$, the set of Sacks indices
active at the current level.
Let $s$ be minimal such that $\sacksnor^{*}(2^s)\ge n$, and
define $h'$ by 
\begin{equation}\label{eq:what}
  h':=  (h_n+1)\cdot 2^{s\cdot |\Sigma|}.
\end{equation}
Finally, let $h_{n+1}$ be minimal such that
for all $\xi\in\Sigma$ there is an
$\ell(\xi)$ with $h'\le \ell(\xi)<h_{n+1}$ and
$\sacksnor^{*}(p(\xi,G'(\ell(\xi))))\ge n$.
(We can find such $\ell(\xi)$, as even
$\sacksnor  
    (p(\xi,G'(\ell)))$ diverges to infinity.)
  \item  
    $G'(h_{n})<G'(h_{n}+1)<\cdots<G'(h_{n+1}-1)<G'(h_{n+1})$ are
    consecutive elements of $\ww^p$. 
    We glue $p$ between $G'(h_{n})$ and $G'(h_{n+1}-1))$, so
    $G'(h_{n})$ and $G'(h_{n+1})$ will be consecutive elements of $\ww^q$.

    We now define the compound creature $q(G'(h_{n}))$, a pure strengthening of the 
    compound creature $\glue(p(G'(h_{n}),\dots,p(G'(h_{n+1}-1))))$:
    The subatoms are unchanged. So we just
    have to specify for each $\xi\in\supp(p,h_{n})\cap\Xisk$
    the new Sacks column
    $q(\xi,h_{n})\leq p(\xi,G'(h_{n}))\otimes\cdots\otimes p(\xi,G'(h_{n+1}-1))$
    as follows:
    Recall that there is one $\ell(\xi)$ such that
    $h'\le \ell(\xi) < h_{n+1}$ and $\sacksnor^{*}(p(\xi,G'(\ell(\xi))))\ge n$.
    Choose a singleton subset of $p(\xi,G'(m))$ for
    all $m\neq \ell(\xi)$, and at $m=\ell(\xi)$
    pick a subtree of $p(\xi,G'(m))$ which is 
    isomorphic to $2^s$ (in the sense that each branch has $s$ splitting points).

    By the definition of $s$, we have 
    $\sacksnor(q(\xi,h_{n}))\ge n$, and therefore
    $\nor(q(h_n))\ge \min(n,\nor(p(h_n),\dots,\nor(p(h_{n+1}-1))))$.
    So in particular the $q$ we get after the induction will be an
    element of $\QQQ$.
  \item 
    As we  choose singletons below $G'(h')$, 
    $|S_{h_n}|=|S_{h_n+1}|=\cdots=|S_{h'-1}|$.
    By induction, $|S_{h_n}|<h_n+2$; so 
    ($\ast_\ell$) holds for $\ell\le h'$.
    For each $h'\le  \ell \le h_{n+1}$,
    we added at each $\xi\in \Sigma$ at most once
    at most $2^s$ many possibilities.
    So $|S_{\ell}|\le (h_n+1)\cdot 2^{s\cdot |\Sigma|}< \ell+2$,
    by~\eqref{eq:what}.\qedhere
\end{enumerate}
\end{proof}

By Lemma~\ref{lem:laversacks}(3), we conclude: 
  
\begin{cor}\label{cor:cnlekcn}
\begin{enumerate}
\item
  If $\n N$ is the name of a null set and $p$ a condition, then there
  is a $q\leq p$ and some name of a null set  $\n N'$ not depending on any Sacks
  indices such that $q$ forces $\n N\subseteq \n N'$.
\item
  $\QQQ$ forces $\cof(\NULL)\le \kappa_\typecn$.
\end{enumerate}
\end{cor}

\subsection{Lim-inf and lim-sup: \texorpdfstring{$\non(\MEAGER)\le\kappa_\typenm$}{non(M) <= kappa-nm}}

The following does not require any knowledge about the particular subatoms
used in the forcing construction, the only relevant fact is that 
the  $\typenm$ indices are the only ones that use a
lim-inf construction.

\begin{lem}\label{lem:nmleknm}
  $\mathbb Q$ forces $\non(\MEAGER)\le \kappa_\typenm$.
\end{lem}
\begin{proof}
  We claim that the 
  set of all reals that can be read continuously from $\typenm$-indices
  is not meager. This set has size $\le \kappa_\typenm$ by Lemma~\ref{lem:fewrealsreadcontinuously}.
  
  Let  $\n M $ be a name for a meager set.
  We can find names
   $\n T_n\subseteq 2^{<\omega}$ for nowhere dense trees
   such that $\n M=\bigcup_{n\in\omega}[\n T_n]$ is forced. 
  We want to show that we can continuously read a real $\n r\notin \n M$
  using only the $\typenm$-indices.

  As $\mathbb Q$ is $\omega^\omega$-bounding and $\n T_n$ is nowhere dense, there is in $V$ a 
  function $f_n:\omega\to\omega$ such that for each $\nu\in 2^k$ there is a 
  $\nu'\in 2^ {f_n(k)}$ extending $\nu$ and not in $\n T_n$.

  We fix some $p\in \mathbb Q$ forcing the above, and assume that $p$ is pruned
  and continuously reads $\n T_n$ for each $n$. 
  We will construct (in $V$) a $q\leq p$ and an $\n r$ continuously read by $q$ only using
$\typenm$ indices,
  such that $q$ forces $\n r\notin \n M$.

  Assume we have already constructed $q$ below some $k_n\in\ww^q$, and 
  that we already have some
  $h_n\in\omega$ and a name  $\n \ell_n$ for an element of $2^{h_n}$
  that is decided by $\poss(q,{<}k_n)\restriction \Xinm$.
  (The real $\n r$ will be the union of the $\n \ell_n$.)
  We also
  assume that is already guaranteed that $\n \ell_n$ is not in 
  $\n T_0\cup\cdots\cup\n T_{n-1}$).

  Enumerate $\poss(q,{<}k_n)$ as $\eta_0,\dots,\eta_{K-1}$.
 
  Set $k^0:=k_n$, $h^0:=h_n$, $\n\ell^0:=\n\ell_n$, and we define  $q'$ below $k^0$ to be $q$.
  By induction on $r\in K$ we now deal with $\eta_r$:
  Assume we are given
  a name $\n \ell^r$ for an element of $2^{h^r}$
  that is decided by $\poss(q',{<}k^r)\restriction \Xinm$,
  and that we have constructed $q'$ below $k^r\in\ww^p$, in a way that between $k^0$
  and $k^r$ on the non-$\typenm$ indices, all subatoms and Sacks columns in $q'$
  are singletons.

  Set $h^{r+1}:=f_n(h^r)$. Choose $k^{r+1}\in\ww^p$ bigger than $k^r$ and large enough  to determine $\n X:=\n T_n\restriction h^{r+1}$. I.e., there is a function $F$
from $\poss(p,{<}k^{r+1})$ to potential values of $\n X$.
We now define $q'$ between $k^r$ and $k^{r+1}$:
  The $\typenm$-subatoms are unchanged (i.e., the ones of $p$), for the other subatoms
  and Sacks columns, we choose arbitrary singletons.
  A $\nu\in\poss(p,{<}k^{r+1})$ consists of: 
  the part below $k^r$ called $A$, then non-$\typenm$-part above $k^r$ called $B$,
  and the $\typenm$-part above $k^r$ called $C$.
  So we can write $\n X=F(A,B,C)$.
  If we assume that the generic chooses $\eta_r$ (i.e., $A=\eta_r$)
  and then follows
  the singleton values of $q$ on the non-$\typenm$-part
  (which determines $B$ to be some $B_q$), then
  $\n X$ can be written as $\typenm$-name. More formally:
  We can define $\n X'$ as $F(\eta_r,B_q,-)$, which is a $\typenm$-name
  and forced by $q$ to be $\n X$.
  
  Also, we know 
  that $p$ forces that there is an element $\ell'\in 2^{h^{r+1}}$ which 
  extends $\n \ell^r$ (which by induction 
  is already determined by the $\typenm$-part of $\eta_r$) and which  is not in $\n X$.
  So (in $V$) we can pick
  for all choices of $C$ an $\ell'(C)\in 2^{h^{r+1}}\setminus F(\eta_r,B_q,C)$
  extending $\n\ell^r$.
  Then $\n \ell^{r+1}=\ell(-)$ is a $\typenm$-name determined below $k^{r+1}$, and 
  $q$ forces that $\n\ell^{n+1}$ extends $\n\ell^n$, and $q\wedge \eta_r$ forces that
  $\n\ell^{n+1}\notin \n T_n$.
  
  We repeat the construction for all $r\in K$,
  and set $\ell_{n+1}:=\ell^K$, $h_{n+1}:=h^K$ and
  set  $k_{n+1}$ to be the $\ww^p$-successor of $k^K$,
  where we use the Sacks columns and subatoms of $p$ between $k^K$ and $k_{n+1}$.
  We now glue the condition between $k_n$ and $k_{n+1}$.
  This results in a condition that still has ``large'' norm,
  as described in Lemma~\ref{lem:blab}.
\end{proof}

\section{The \texorpdfstring{$\typenm$}{nm} part}\label{ss:nm}

\subsection{The subatomic creatures for type \texorpdfstring{$\typenm$}{nm}}\label{ss:nmdef}

We now describe the subatomic family $\cK'_{\typenm,b}$ used at $\typenm$-indices (depending on the parameter $b$).

\begin{definition}\label{def:Knm}
\begin{enumerate}
  \item
Fix a finite index set $I\subseteq \omega$ which is large enough so 
that item~(\ref{item:nmlarge}) below is satisfied.  For
notational simplicity, we assume that $I$ is disjoint to all intervals already
used.\footnote{This is a bit fuzzy, but it does not matter how we interpret it.
More specifically, we could use any of the following:
``disjoint to all $I$ that are associated to smaller parameter values
 $b'<b$'', or: ``disjoint to all $I$ that have actually been used 
in type $\typenm$ for some $\cK_{\xi,\slv}$'';
and since $H'(\typenm,{=}b)$ is larger than $\max(I)$,
it would also follow from:
``the minimum of $I$ is bigger than $H({<}\slu')$, where
$\slu'$ is the predecessor of the current sublevel''.}
\item $\POSS_{\typenm,b}:=2^I$.
\item 
A subatomic creature $x$ is just a nonempty subset of $2^I$, 
where we set $\poss(x):=x$ and 
\[
  \nor(x):=\frac{1}{b}\log_b(|\poss(x)|). 
\]
\item\label{item:nmlarge}
We require $\nor(\POSS)>b$ (thus satisfying~\eqref{eq:paramlarge}).
\item\label{item:nmH}
We set $H'(\typenm,b):=2^{\max(I)+1}$.
\end{enumerate}
\end{definition}
Clearly, the norm satisfies strong $b$-bigness (i.e., satisfies the requirement~\eqref{eq:parambig}).
\begin{nb}\label{nb:othernorm}
We just used the simplest possible norm here. 
It turns out that the details of the definition of this  norm are not relevant,
as long as the norm has bigness.   Later in section~\ref{sec:switching}
we will use a different norm to get a different constellation of cardinal 
characteristics.
\end{nb}

\subsection{The generic object}\label{ss:nmgeneric}

Recall that (according to Section~\ref{ss:forcing88})
when constructing the forcing at subatomic sublevels $\slu$, we use for all
$\xi\in\Xi_\typenm$ the subatomic family
$\cK_{\xi,\slu}=\cK'_{\typenm,b(\slu)}$ living on some
interval $I$, which we will call $I_{\typenm,\slu}$.

Fix $\alpha$ of type $\typenm$.
Recall that the generic object $\yy_\alpha$ 
assigns to each subatomic sublevel $\slu$
the element of $\POSS_{\alpha,\slu}$ chosen by the generic filter.

We define the name $\n M_\alpha$ of a meager set as follows: 
\proofclaim{eq:defmalpha}{A real
$r\in2^\omega$  is in $\n M_\alpha$ iff for all but finitely many levels $\ell$
there is a subatomic sublevel $\slu=(\ell,j)$ such that
$r\restriction I_{\typenm,\slu}\neq \yy_\alpha(\slu)$.}
If $p$ rapidly reads $\n r$, then
according to~\eqref{eq:Hprimerapid} and \ref{def:Knm}(\ref{item:nmH}),
\proofclaim{eq:bla7788}{
$\n r\restriction I_{\typenm,\slu}$
is decided $\le\slu$.}

Also, since $b(\slu)>\maxposs({<}\slu)$, we get:
\proofclaim{eq:blubb66}{
If the norm of a $\typenm$-subatom $x$ at sublevel~$\slu$ is at least $1$, 
then $|\poss(x)|>\maxposs({<}\slu)$.
}
(Recall Note~\ref{nb:othernorm}: This is true whenever the norm has bigness.)

\subsection{\texorpdfstring{$\non(\MEAGER)\ge \kappa_\typenm$}{non(M) >= kappa-nm}}\label{ss:nmkappa}

\begin{lem}\label{lem:nmgeknm}
   Let $\n r$ be a name of a real, 
   $p$ a condition that rapidly reads $\n r$ not using\footnote{cf.~\ref{def:continuousreading}}
   $\alpha\in\Xi_\typenm$.
   Then $p$ forces that $\n r\in \n M_\alpha$.
\end{lem}
\begin{proof}
It is enough to prove that some $q\leq p$ forces that 
$\n r\in \n M_\alpha$:
Assume that $p$ does not force $\n r\in \n M_\alpha$,
then some $p'\leq p$ forces the
negation; $p'$ still rapidly reads $\n r$ not using $\alpha$,
so if we know that there is a $q\leq p'$ as claimed, we get a contradiction.

We can assume that $p$ is pruned and that $\alpha\in\supp(p)$.
We will construct a $q$ purely stronger than $p$
(in particular with the same $\ww$, halving parameters, and trunk).
Actually, we will only strengthen one subatom at index $\alpha$
for each level $h\ge \min(\ww^p)$.

For all $h\ge \min(\ww^p)$ (not necessarily in $\ww^p$), there are several
$j\in J_h$ such that $\nor(x)>1$ for the subatom $x=p(\alpha,(h,j))$.
For each such $h$ we pick exactly one subatomic sublevel $\slu(h)=(h,j)$, with
$x(h)$ the according subatom.

According to~\eqref{eq:bla7788}, $\n r\restriction I_{\typenm,\slu}$ is decided $\le\slu$
and therefore even below $\slu$ (since $\alpha$
is the active index at sublevel $\slu$; according 
to modesty no other index can be active; and $\n r$ does not depend on $\alpha$). Therefore there are at most $\maxposs({<}\slu)$ many possibilities
for $\n r\restriction I_{\typenm,\slu}$. According to~\eqref{eq:blubb66}
there has to be at least one element $s$ of $\poss(x(h))$ which differs from
all of these possibilities. So we can in $q$ replace the 
subatom $x(h)$ with the singleton $\{s\}$. Then the norms in $q$ will
still be large. (If $A\subseteq J_h$ witnesses the large $\linor$ of
$p$, then $A\setminus \{j\}$ for $\slu(h)=(h,j)$
witnesses that the $\linor$ of $q$ decreases only slightly.)

So $q$ is constructed by strengthening each $x(h)$ in this way. Clearly $q\leq p$ is still a valid condition,
and forces $\n r\in \n M_\alpha$, as $\n r\restriction  I_{\typenm,(h,\slu(h))}$ disagrees with $\yy_\alpha $
for all $h\ge\min(\ww^p)$.
\end{proof}

\begin{cor}\label{cor:nmgeknm}
  $\mathbb Q$ forces $\non(\MEAGER)\ge \kappa_\typenm$.
\end{cor}

\begin{proof}
  Assume that $\kappa_\typenm>\al_1$ (otherwise there is nothing to show).
  Fix a condition $p$ and $\kappa<\kappa_\typenm$ and names  $(\n r_i)_{i\in\kappa}$ of reals. It is enough to show that there is an $\alpha\in\Xinm$
  such that $p$ forces that
  $\{\n r_i:\, i\in\kappa\}$  is a subset of the meager set $\n M_\alpha$.

  For each $i$ fix a maximal antichain $A_i$ below $p$ such that each $a\in A_i$ 
  rapidly reads $\n r_i$.
  Due to $\al_2$-cc, and since $\kappa_\typenm>\al_1$ and $\kappa_\typenm>\kappa$,
  we can find an index $\alpha\in\Xi_\typenm$ not appearing in the support
  of any condition in any $A_{i}$.
  According to the previous lemma, every element $a\in A_i$ (and hence also $p$ itself) 
  forces that $\n r_i\in \n M_\alpha$.
\end{proof}

\section{The \texorpdfstring{$\typenn$}{nn} part}\label{ss:nn}
\subsection{The subatomic creatures for type \texorpdfstring{$\typenn$}{nn}}\label{ss:nndef}

We describe the subatomic families $\cK'_{\typenn,b}$ , depending on a
parameter $b$.

\begin{definition}\label{def:Knn}
\begin{enumerate}
  \item Fix an interval $I$ large enough such that (\ref{item:nnlarge}) is satisfied
(and in particular $|I|>b$).
As in the $\typenm$ subatoms, we assume that this interval $I$ is disjoint to
all intervals previously chosen. 
\item The basic set of all possibilities, $\POSS$, consists of all
subsets $X$ of $2^I$ with relative size $1-1/2^b$:
  \[
\POSS:=\{X\subseteq 2^I:\, |X|=(1-1/2^b)|2^I|\}.
\]
  \item A subatom $C=\poss(C)$ is a subset of $\POSS$, where we set
  \begin{gather*}
  \nor(C):=\frac{1}{b}\log_b(\nor_0(C)),
  \text{ where }
  \\
  \nor_0(C):=\min\{|Y|:\, Y\subseteq 2^I,\ (\forall X\in \poss(C))\, X\cap Y\neq \emptyset\}.
  \end{gather*}
  \item\label{item:nnlarge}
    We require $\nor(\POSS)>b$ (thus satisfying~\eqref{eq:paramlarge}).
  \item We set $H'(\typenn,{=}b):=\max(I)+1$.
\end{enumerate}
\end{definition}
Note that $\nor_0$ of the subatom with full possibility set
is approximately $2^{|I|}/2^b$. In particular, for large $I$ the norm
gets large, i.e., we can satisfy~(\ref{item:nnlarge}).
\begin{lem}\label{lem:bla}
\begin{enumerate}
  \item The subatomic family has strong $b$-bigness (i.e., satisfies the requirement~\eqref{eq:parambig}).
  \item Given $E\subseteq 2^I$ and a subatom $C$,
    then the
    subatom $C'$ with possibilities $\{H\in\poss(C):\, H\cap E=\emptyset\}$
    satisfies $\nor_0(C')\ge \nor_0(C)-|E|$.
    \item From
     the above it follows that:  If $|E| \le \nicefrac{b^{\nor(C)}}{2}$, then
    $\nor(C')\ge \nor(C)-\log_b(2)$.
\end{enumerate}
\end{lem}
\begin{proof}
(1): Fix $F:\poss(C)\to b$. Let $C_i$ be the subatom with
$F\restriction \poss(C_i)=i$ for all $i\in b$.
Assume that all $C_i$ have $\nor_0$ at most $r$, witnesses by $X_i\subseteq 2^I$.
Then $\bigcup X_i$ witnesses that $\nor_0(C)\le b\cdot r$.
So $\nor(C)\le \nicefrac{\log_b(b\cdot r)}{b}\le \nicefrac{1}{b}+\max(\nor(C_i))$.
So there is at least one $i$ with $\nor(C_i)\ge \nor(C)-\nicefrac{1}{b}$, as required.

(2): Assume $Y$ witnesses $\nor_0(C')$, then $Y\cup E$ witnesses $\nor_0(C)$.

(3): 
\begin{align*}
   \frac{b^{\nor(C)}}{2} &=  
   \frac{\nor_0(C)^{\nicefrac{1}{b}}}{2}  =
 \biggl( \frac{\nor_0(C)}{2^b}\biggr)^{\nicefrac1b} \le \\ 
    & \le \biggl[ (1-\frac1{2^b})^{\nicefrac1b} \cdot \nor_0(C)^b\biggr]^{\nicefrac1b} = 
                  (1-\frac1{2^b}) \cdot \nor_0(C) \qedhere
\end{align*}
\end{proof}

\subsection{The generic object}\label{ss:nngeneric}

The following paragraph is just as in the $\typenm$ case~\ref{ss:nmgeneric}:

According to Section~\ref{ss:forcing88},
when constructing the forcing at subatomic sublevels $\slu$, we use for all
$\xi\in\Xi_\typenn$ the subatomic family
$\cK_{\xi,\slu}=\cK'_{\typenn,b(\slu)}$ living on some
interval $I$, which we temporarily call $I_{\typenn,\slu}$. Also,
if $p$ rapidly reads $\n r$, then 
$\n r\restriction I_{\typenn,\slu}$
is decided below $\le \slu$.

Fix $\alpha$ of type $\typenn$.  Recall that the generic object $\yy_\alpha$
assigns to each subatomic sublevel $\slu$ the element $\n R_{\alpha,\slu}$ of
$\POSS_{\alpha,\slu}$ chosen by the generic filter.  So $\n R_{\alpha,\slu}$
is a subset of $2^{I_{\typenn,\slu}}$ of  relative size
$(1-1/2^{b(\slu)})$.

Note that $b(\slu)$ is strictly monotone (cf.~\eqref{eq:defbA}),
and hence
$\prod_{\slu\text{ subatomic sublevel}}(1-1/{2^{b(\slu)}})>0$.
Therefore
\[
  \{x\in 2^\omega:\, \forall \slu:\, x\restriction I_{\typenn,\slu}\in \n R_{\alpha,\slu}\}
\]
is positive, and
\[
  \{x\in 2^\omega:\, \forall^\infty \slu:\, x\restriction I_{\typenn,\slu}\in \n R_{\alpha,\slu}\}
\]
has measure one.
Therefore 
\begin{equation}\label{eq:N.alpha}
  \n N_\alpha:=  \{x\in 2^\omega:\, \exists^\infty \slu:\, x\restriction I_{\typenn,\slu}\notin \n R_{\alpha,\slu}\}
\end{equation}
is a null set. (Here, $\slu$ ranges over all subatomic sublevels.)

\subsection{\texorpdfstring{$\non(\NULL)\ge\kappa_\typenn$}{non(N) >= kappa-nn}}\label{ss:blabla}

\begin{lem}\label{lem:nngeknn}
  Let $p\in \QQQ$ rapidly read $\n r\in 2^\omega$ not using $\alpha\in\Xinn$.
  Then $p$ forces $r\in N_\alpha$.
\end{lem}

\begin{proof}
As in~\ref{lem:nmgeknm}, it is enough to find a $q\le p$ forcing  $r\in N_\alpha$; and we
assume that $p$ is pruned and that $\alpha\in\supp(p)$.

We construct $q$ purely stronger than $p$ by induction, only modifying 
subatoms at index $\alpha$ (and decreasing their subatom norms by at
most $1$):

Pick a subatomic 
sublevel $\slu$ (higher than any sublevel previously considered) where $\alpha$ is active with the subatom $C$ ``living'' on $I:=I_{\typenn,\slu}$. 

$\n r\restriction I$ is decided $\le\slu$ and therefore even below $\slu$
(as $\n r$ is read from $p$ not using $\alpha$; and due to modesty $\alpha$
is the only index active at sublevel $\slu$).
So the set $E$ of possibilities
for $\n r\restriction I$ has size at most $\maxposs({<}\slu)$,
and we can remove them all from the subatom at $C$ while decreasing
the norm by at most 1, according to 
Lemma~\ref{lem:bla}(2) and~\eqref{lem:defb}.

Repeat this for infinitely many sublevels $\slu$.
\end{proof}

Just as in~\ref{cor:nmgeknm}, this implies:
\begin{cor}\label{cor:nngeknn}
  $\mathbb Q$ forces $\non(\NULL)\ge \kappa_\typenn$.
\end{cor}

\section{Some simple facts about counting}\label{sec:counting}

We now list some simple combinatorial properties that will be used for the definitions and proofs
in the $\typecn$-part.

\subsection{Large families of positive sets have positive intersection, \texorpdfstring{$\norint$}{nor(cap)}}

\begin{lem}\label{lem:Mepsilon}
	For $\delta\in(0,1)$ and $\ell\in\omega $
	there are $M(\delta,\ell)\in\omega $
        and $\epsint(\delta,\ell)>0$ 
        such that:
	Whenever we have a probability space $\Omega$ and a family $(A_i:\, i<M)$ 
	of sets of measure $\ge\delta$, 
	we can find a subfamily of $\ell$
	many sets  
        whose intersection has measure at least $\epsint(\delta,\ell)$.
\end{lem}
\begin{proof}
By straightforward counting.\footnote{Originally we used a stronger statement
for which we only had a more complicated proof. We are grateful to William
B.~Johnson for pointing out in \url{http://mathoverflow.net/q/108380} that the
statement in the current form has the obvious straightforward proof.}

We write $\chi_B$ for the characteristic function of $B$.
Assume we have $M$ many sets $A_i$,
and set $X\subset \Omega$ to contain all points
that lie in at least $\ell$ many of the $A_i$. Then
\[
  \delta\cdot M \le \int \sum_{i\in M}\chi_{A_i} \le 
    \mu(X)\cdot M + \mu(\Omega\setminus X) \cdot (\ell-1) \le 
    \mu(X)\cdot M + \ell,
\]
and $\mu(X)\ge \delta-\nicefrac{\ell}M$.
So if we set 
\[
  M>2\frac{\ell}{\delta},
\]
 then there are at least
$\nicefrac{\delta}{2}$ ``many'' points in $X$.
We can assign to each point $x\in X$ a subset $M_x$ of $M$ (of size
at least $\ell$) by
\[
  i\in M_x \text{ iff } x\in A_i.
\]
This partitions  $X$ into at most $2^M$ many sets;
and at least one of the pieces has to have size at least
\[
  \epsint(\delta,\ell):=\frac{\delta}{2\cdot 2^M}.\qedhere
\]
\end{proof}

Let us set $F_b^0:=1$ and  $F_b^{n+1}= M(\nicefrac1b,F_b^n)$.  We can use this
notion to define a norm on natural numbers:
\begin{definition}\label{def:nor1}
  For $m> 0$:  
  $\norint_b(m)\ge n$ iff $m\ge F_b^n$.
\end{definition}

So we get the following:
\proofclaim{eq:hjhjhjh7}{
  Fix a measure space $\Omega$ and a 
  sequence $(T_i)_{i\in A}$ of sets of measure $\ge \nicefrac1b$.
  Then there is a subset $B\subseteq A$ such that $\norint_b(|B|)\ge \norint_b(|A|)-1$
  and $\bigcap_{i\in B} T_i$ has measure $\geq \epsint(\nicefrac1b,|A|)$.}

Note that without loss of generality the function $\epsint$
satisfies: $\epsint(\delta,\ell_1)\ge \epsint(\delta,\ell_2)$
whenever $\ell_2>\ell_1>0$. 
We write down the following trivial consequence of~\eqref{eq:hjhjhjh7}
for later reference:
\proofclaim{eq:jjkwr}{
  Assume that $A$ is a subset of some finite set $\POSS$.
  Fix a measure space $\Omega$ and a 
  sequence $(T_i)_{i\in A}$ of sets of measure $\ge \nicefrac1b$.
  Then there is a subset $B\subseteq A$ such that $\norint_b(|B|)\ge \norint_b(|A|)-1$
  and $\bigcap_{i\in B} T_i$ has measure $\geq \epsint(\nicefrac1b,|\POSS|)$.}

\subsection{Most large subsets do not cover a half-sized set}

Let $\Omega$ be the set of subsets of some finite set $A\in\omega$
of relative size $1-\epsilon$ (for $0<\epsilon <\nicefrac14$).
(Since $A\in\omega$, we can write $A$ for the cardinality $|A|$.)
I.e.:
$x\in \Omega$ implies $x\subseteq A$ and $|x| = A\cdot (1-\epsilon)$.
We can assume $A\gg \nicefrac1{\epsilon}$ and that $A\cdot\epsilon$ is an integer.

Let $T\subseteq A$ be of relative size $\ge 1/2$, i.e.,
$|T|\ge \nicefrac{A}2$.
Let $\Omega_T$ be the elements of $\Omega$ that
cover $T$, i.e., $x\in\Omega^T$ iff $x\in\Omega$ and $T\subseteq x$.

We will use the following easy fact from combinatorics: 
\begin{fact}\label{fact:combi}
For any natural number $k\ge 2$, the quotient
\[
  \frac{\binom{2Nk }N}{\binom{Nk}{N}}
\]
tends to infinity with  $N\to \infty$.
\end{fact}
\begin{proof}
   This can be checked with Stirling's approximation formula, or with the following
   elementary estimate: From 
\[ 
   \forall a,b: \frac{(a-b)^b}{b!}\le \binom ab \le \frac{a^b}{b!}
\]
   we get 
 \[ 
  N!\cdot {\binom{2Nk }N} \ge  (2Nk-N)^N \text{\quad and\quad} 
  N!\cdot {\binom{Nk }N} \le  (Nk)^N,  
\]
and hence 
 \[ 
  \frac{\binom{2Nk }N}{\binom{Nk}{N}} \ge \frac{ (2Nk-N)^N }{(Nk)^N  }\ge (2-\frac1k)^N\to \infty.
\]
\end{proof}

\begin{lem}\label{def:nor2}
  Fix $b>2$ and a finite set $I$ with $|I|>b$.
  Let $\POSS$ 
  be the family of subsets of $2^I$ of relative size $1-\nicefrac{1}{2^b}$.
  For $m\in\omega$ we 
  define $\norhalf_{I,b}(m):=\lfloor \nicefrac{m}{\binom{2^{|I|-1}}{2^{|I|-b}}} \rfloor$.

  Then:
\begin{enumerate}
\item
For any $T\subseteq 2^I$ of at least relative size $\nicefrac1{2}$
  and for any $C\subseteq \POSS$ there is a subset $D\subseteq C$
  with $\norhalf_{I,b}(|D|)\ge \norhalf_{I,b}(|C|)-1$ and $T\not\subseteq x$ for all $x\in D$.
\item   
  If $I$ is chosen sufficiently large (with respect to $b$), 
  then $\norhalf_{I,b}(\POSS)$ is large.
\end{enumerate}
\end{lem}
\begin{proof}
\begin{enumerate}
\item
 It is enough to show this in case $T$ has exactly size $2^{|I|-1}$. 
  If $x\in C\setminus D$, then the set $2^I\setminus x$ has size  $2^{|I|-b}$
  and is 
  a subset of $2^I\setminus T$.  So there are at most $\binom{2^{|I|-1}}{2^{|I|-b}}$
  possibilities for $2^I\setminus x$, hence (by definition of $\norhalf_{I,b}$)
  we get $\norhalf(C\setminus D)\le 1$.  From the implication 
  \[   x\le y \text{ and } \lfloor x - y \rfloor  \le 1 \ \Rightarrow \      
                          \lfloor x \rfloor - 
                          \lfloor y \rfloor \le 1 
 \] 
 we get 
  $\norhalf_{I,b}(C)-\norhalf_{I,b}(D)\le 1$. 
\item
 Note that the cardinality of $\POSS$ is equal to 
$\binom{2^{|I|}}{2^{|I|-b}}$.   Using Fact~\ref{fact:combi} with 
$N:=2^{|I|-b}$ and $k:=2^{b-1}$  we get that 
$\nicefrac{\binom{2^{|I|}}{2^{|I|-b}}}{\binom{2^{|I|-1}}{2^{|I|-b}}}$ is large for large~$I$.\qedhere
\end{enumerate}
\end{proof}

\subsection{Providing bigness}

In this section, we write $\log$ to denote $\log_2$.

Apart from unimportant rounding effects,
$\log$ of $\norhalf$ satisfies $2$-bigness (and the same for
$\norint$).   Instead of thinking about
such effects, we just define for any norm a $2$-big version. Actually, we
define a $2$-big version
of the combinations of two norms (of course, any finite number of norms 
can be combined in this way):

\begin{definition}\label{def:golnor}
  Assume that $\nor_1,\nor_2:\omega\to\omega$ are
  weakly increasing and converge to infinity.

  Then we define $\gol=\gol(\nor_1,\nor_2):\omega\to\omega$ as follows:
  By induction on $m$, we define $\gol(x)\ge m$ by the conjunction of the following clauses:
  \begin{itemize}
    \item $\nor_1(x)\ge m$ and $\nor_2(x)\ge m$.
    \item $\gol(\lfloor \frac{x}{2} \rfloor )\ge m-1$.
    \item If $y\in\omega$ and $i\in\{1,2\}$ satisfies
      $\nor_i(y)\ge \nor_i(x)-1$, then 
      $\gol(y)\ge m-1 $.
  \end{itemize}
  We set $\gol(x):=\gol(\norint,\norhalf)$.
\end{definition}
\begin{lem}\label{lem:golnor}
  Let $\gol=\gol(\nor_1,\nor_2)$.
  \begin{itemize}
    \item $\gol(x)$ is a well-defined natural number for all $x$, i.e.,
      there is a maximal $m$ such that $\gol(x)\ge m$ holds.
    \item $\gol$ is weakly increasing and diverges to infinity.
    \item $\gol$ has $2$-bigness:
If $F:m\to 2$ is a coloring function
and $\gol(m)=n$,
then there is some $c\in 2$ such that
$\gol(F^{-1}(c))\ge n-1$.
    \item So if we define $\nor_b(x)$ as $\frac{\gol(x)}{\lceil \log(b)\rceil}$,
      then $\nor_b$ will be $b$-big.
    \item If $\nor_i(y)\ge \nor_i(x)-1$ for some $i\in\{1,2\}$,
      then $\gol(y)\ge \gol(x)-1$.
  \end{itemize}
\end{lem}

\begin{proof}
  ``Well-defined'' follows from $\gol(x)\leq \nor_i(x)$.
 
  Monotonicity follows from the monotonicity of $\nor_1$ and $\nor_2$.
 
  We now prove that by induction on $m$ that
  there are only finitely many $x$ with $\gol(x)<m$.
  For $m=0$ this is obvious, as all 
  $x$ satisfy $\gol(x)\ge 0$.
  For $m>0$:
  $\gol(x) <m$  iff either  $\nor_1(x)<m$ or  $\nor_2(x)<m$ or
  $\gol(\lfloor \frac{x}{2} \rfloor )<m-1$
  or there is some $y$ and some $i\in\{1,2\}$ with $\nor_i(y)\ge \nor_i(x)-1$
  and $\gol(y)< m-1 $; for each case there are only finitely many possibilities.

  $2$-bigness and the last item follow directly from the definition.
  $b$-bigness is Lemma~\ref{lem:bbigfrom2big}.
\end{proof}

\section{The \texorpdfstring{$\typecn$}{cn} part}\label{ss:cn}
\subsection{The subatomic creatures for type  \texorpdfstring{$\typecn$}{cn}}\label{ss:cndef}
We now describe the subatomic families $\cK'_{\typecn,b}$ used for the 
$\typecn$-indices. 

\begin{definition}\label{def:Kcn}
\begin{enumerate}
  \item Fix an interval $I$ which is large enough to satisfy~(\ref{item:cnlarge}).
In particular, $|I|>b$.
Again, we assume that this interval is disjoint to all intervals previously
chosen.
  \item
The basic set of all possibilities and the set of subatoms is the same
as in the $\typenn$-case~\ref{def:Knn} (but the norm will
be different). So $\POSS$ consists of all
subsets $X$ of $2^I$ with relative size $1-1/2^b$:
\[
  \POSS=\{X\subseteq 2^I:\, |X|= (1-1/2^b)|2^I|\}.
\]
\item 
A subatom $C$ is a subset of $\POSS$, with $\poss(C):=C$, and 
\[
  \nor(C):=\frac{\gol(\norint_b,\norhalf_{I,b})(|C|)}{2^{\min(I)}\cdot b^2}.
\]
  \item\label{item:cnlarge}
    We require $\nor(\POSS)>b$ (thus satisfying~\eqref{eq:paramlarge}).
  \item\label{item:Hcn}
We set 
$H'(\typecn,{=}b):=\max(H'_0,H'_1)$
for 
$H'_0:=  
   2^{\binom{2^{|I|}}{2^{|I|-b}}}  $ and
$H'_1:=\nicefrac1{\epsint(\nicefrac 1b,|\POSS|)}$, where
$\epsint$ is defined in~\ref{lem:Mepsilon}. 
\end{enumerate}
\end{definition}

Note that $H'(\typecn,{=}b)>|\cK'_{\typecn,b}|$ (this is what we need $H'_0$ for).

Recall that $\gol$ satisfies $2$-bigness, so after dividing by $b$
(actually, $\lceil \log_2(b)\rceil \cdot b$ would be sufficient) we get strong
$b$-bigness (i.e., the norm satisfies the requirement~\eqref{eq:parambig}).

Note that (in contrast to the $\typenn$ case) this norm
is a counting norm, i.e., $\nor(C)$ only depends on $|C|$,
not on the ``structure'' of $C$.

\subsection{The generic object}

Just as in the $\typenn$-case, we set $I_{\typenn,\slu}$ to be the $I$ used for
$\cK'_{\typenn,b(\slu)}$;  and we define $\n N_\alpha$ analogously to the $\typenn$-case.%
\footnote{Of course, generally
$I_{\typecn,\slu}\neq I_{\typenn,\slu}$, so $\n N_\alpha$
for $\alpha\in\Xinn$
lives on a different domain than $\n N_\beta$ for $\beta\in\Xicn$.}

As before, $\n N_\alpha$ is a name for a null set, and a real $r$ is in $\n
N_\alpha$ iff there are infinitely many sublevels $\slu$ such that
$r\restriction I_{\typecn,\slu}$  is not in the possibility $X$ of
$\cK'_{\typecn,\slu}=\cK_{\alpha,\slu}$ that is chosen by the generic filter.

This time, the purpose  of $\n N_\alpha$ is not to cover all reals not
depending on $\alpha$, but rather to avoid being covered by any null set not
depending on $\alpha$.

\begin{lem}\label{lem:lemmaA}
  Fix a subatomic sublevel $\slu$, an index $\alpha\in \Xicn$
  and a subatom $C\in \cK'_{\typecn,\slu}=\cK_{\alpha,\slu}$.
  \begin{enumerate}
  \item
  Given $T\subseteq 2^{I_{\typecn,\slu}}$ 
  of relative size $\ge \nicefrac12$ 
  we can strengthen $C$ to $D$, decreasing the norm by at most
  $\nicefrac1{2^{\min(I)}\cdot b(\slu)}$  
  such that $T\not\subseteq X$ for all $X\in \POSS(D)$.
  \item
	Fix a probability space $\Omega$ and a function 
	$F$ that maps every $X\in\poss(C)$ 
	to $F(X)\subseteq \Omega$ 
	of measure $\ge \nicefrac1{b(\slu)}$.
	Then we can strengthen $C$ to $D$, decreasing the norm
	by at most $\nicefrac1{2^{\min I}\cdot b(\slu)}$
	such that $\bigcap_{X\in\poss(D)} F(X)$ has measure at least 
        $\nicefrac1{b(\slu+1)}$. Here, $\slu+1$ denotes the smallest
        subatomic sublevel above $\slu$.
\end{enumerate}
\end{lem}
\begin{proof}
This is an immediate consequences of~\eqref{eq:jjkwr},
\ref{def:nor2} and~\ref{lem:golnor},
just note that
\[
  b(\slu+1) > H'(\typecn,{=}b(\slu)) \ge
  \nicefrac1{\epsint(\nicefrac 1{b(\slu)},|\POSS|)}.  \qedhere
\]
\end{proof}

Again, let $\slu+1$ denote the smallest
        subatomic sublevel above $\slu$. Then
\[
b(\slu+1) > H'(\typecn,{=}b(\slu)) >  |\cK_{\typecn,b(\slu)}|.
\]
In other words,
\proofclaim{eq:blghgh}{
The cardinality of $\cK_{\typecn,b(\slu)}$ is less than $b(\slu+1)$.
}

\subsection{Names for null sets}\label{sec:namesnull}

Let $T\subseteq 2^{{<}\omega}$ be a tree (without terminal nodes) of measure $\nicefrac12$.
(Such trees correspond bijectively to closed sets of measure $\nicefrac12$.)
Then the set 
\begin{equation}
  N_T:=2^\omega\setminus \bigcup \{r+[T]:\, r\in\mathbb Q\}.
\end{equation}
is a null set (closed under rational translations). Conversely, 
for every null set $N$ there is such a $T$ with $N\subseteq N_T$.

The relative measure of $s$ in $T$ (for $s\in 2^{n}$, $n\in\omega$)
is defined as $\mu([T]\cap [s])\cdot 2^n$.
For completeness, we say that the relative measure of $s$ is
$0$ if $s\notin T$.
(Analogously, we can define the relative measure of a node $s$ 
in a finite tree $T\subseteq 2^{\le m}$ with no terminal nodes
of height $<m$.)
Note the following easy consequence of the Lebesgue density theorem:
\begin{fact}\label{fact:lebdens}
  If $T$ is a tree without terminal nodes, 
  $s\in T$ has positive relative measure, and $\delta<1$,
  then there is a $t>s$ with relative measure $>\delta$.
  (And for all levels above the level of $t$, there is an 
  extension $t'>t$ which also has relative measure $>\delta$.)
\end{fact}

By removing nodes with relative measure $0$, the measure of $T$ does not
change.  We give such trees a name:
\begin{definition}
$T$ is a \prunedtree tree, if $T\subseteq 2^{<\omega}$ has measure $\nicefrac12$
and has no
nodes of relative measure zero (and in particular no terminal nodes).
\end{definition}
Note that each null set is contained in $N_T$ for some  \prunedtree $T$.
So instead of investigating arbitrary names for null sets,
we will consider names $\n T$ for \prunedtree trees.

Note that there are fewer than $2^{2^h}$ many possibilities for 
the level $h$ of $\n T$. So we can 
``code'' $\n T$
by a real 
$\n r\in 2^\omega$ 
such that $\n T\restriction h$
is determined by $\n r\restriction 2^{2^{(h+1)}}$.

Assume that $p$ rapidly reads this $\n r$.
Then
$\n T\restriction(\max(I_{\typecn,\slu })+1)$ is determined $\le \slu$
(according to~\eqref{eq:Hprimerapid} and~\ref{def:Kcn}(\ref{item:Hcn})).

We will describe this situation by ``$p$ rapidly reads $\n T$''.

\subsection{\texorpdfstring{$\cof(\NULL)\ge \kappa_\typecn$}{cof(N) >= kappa-cn}}\label{ss:cnA}

\begin{lem}\label{lem:basiccn}
  Let $p\in Q$ rapidly read the \prunedtree tree $\n T$ not using the index
  $\alpha\in\Xicn$. 
  Then $p$ forces that $\n N_\alpha $ is not a subset of $N_{\n T}$, i.e.,\footnote{as $\n N_\alpha$
  is closed under rational translates} there is some
  $s\in \n N_\alpha\cap [\n T]$.
\end{lem}

\begin{proof}
  We can assume that $p$ is pruned and that $\alpha\in \supp(p)$. 
  It is enough to find  a name $\n r\in 2^\omega$ and a $q\leq p$ 
  forcing $\n r\in \n N_\alpha \cap[T]$.
  For this, we will inductively modify $p$ at infinitely many sublevels
  $\slu$ (resulting in the $1$-purely stronger $q$):

  Let $\slu$ be a subatomic sublevel (above all the sublevels that 
  we have already modified), 
  where $\alpha$ is the active index
  with subatom $C$ of norm at least $10$,
  living on the interval $I:=I_{\typecn,\slu}$.

  The finite tree
  $\n T':=\n T\restriction \max(I)+1$ is determined $\le\slu$,
  and even $<\slu$,
  as $\n T$ does not depend on $\alpha$
  (as usual, note that due to modesty $\alpha$ is the only active
  index at sublevel $\slu$).
  In particular the set $Y$ of potential values of $\n T'$ has size
  $\le\maxposs({<}\slu)$.

  We now enumerate all $T^*\in Y$ and 
  $t\in T^*\cap 2^{\min(I)}$ with relative measure 
  (in $T^*$) at least $\nicefrac12$.
  There are at most 
  $\maxposs({<}\slu)\times 2^{\min(I)}$ many such pairs $(T^*,t)$.
  
  Starting with $C^0:=C$, we iteratively use Lemma~\ref{lem:lemmaA}(1)
  to strengthen the subatom
  $C^n$ to some $C^{n+1}$ such that for the current 
  $(T^*,t)$
  and all $X\in \poss(C^{n+1})$
  there is some $t'\in 2^I\setminus X$ such that 
  $t^\frown t'\in T^*$.

  So in the end we get a subatom $D\leq C$ of norm $\ge\nor(C)-1$ such that for
  all $(T^*,t)$ and $X\in \poss(D)$ there is some $t' \in 2^I\setminus X$ with
  $t^\frown t'\in T^*$.

  In this way, we modify infinitely many sublevels $\slu$, resulting in
  a condition $q\leq p$. 
 
  Now work in the forcing extension, 
  where $q$ is in the generic filter.
  We can now construct by induction an element $r$ 
  of $\n N_\alpha\cap [\n T]$ (i.e., 
  $r\restriction I_{\typecn,\slu}$ is not in the
  generically chosen $X$ at index $\alpha$ and sublevel
  $\slu$, for infinitely many sublevels $\slu$.)

  Assume we already have $r\restriction n\in \n T$ for some $n$.
  Since $\n T$ has no nodes of relative norm $0$, there is a $h'>n$ and an
  $t'\in T\cap 2^{h'}$ extending $r\restriction n$ with relative measure $\ge\nicefrac12$ (see~\ref{fact:lebdens}).
  Pick a sublevel $\slu$ such that: $\min(I)=:h>h'$ for $I:=I_{\typecn,\slu}$,
  and $\slu$ was considered in our construction of $q$.
  There is still some 
  $t\in 2^{h'}$ extending $\n r\restriction n$ 
  of relative measure $\nicefrac12$.
  Set $T^*:=\n T\restriction \max(I)+1$.
  Note that in our construction of $q$, when considering $\slu$,
  we dealt with the pair
  $(T^*,t)$, and thus made sure for all $X\in\poss(q(\alpha,\slu))$
  (so in particular for the one
  actually chosen by the generic filter)
  there is some $t'\in 2^I$ such that
  $t^\frown t'\in T^*$ and $t'\notin X$.
  So we can just set $r\restriction \max I:=t^\frown t'$.
\end{proof}

\begin{cor}\label{cor:cngekcn}
  $\QQQ$ forces that $\cof(\NULL)\ge \kappa_\typecn$.
\end{cor}
\begin{proof}
  This is very similar to the proof of~\ref{cor:nmgeknm}:
  Assume that there is a $\aleph_1\le\kappa<\kappa_\typecn$
  and a $p$ forcing that $(\n N^*_i)_{i\in\kappa}$
  is a basis of null sets.
  As described above, we can assume that each $\n N^*_i=N_{\n T_i}$
  for some \prunedtree tree $\n T_i$ of measure $\nicefrac12$.
  For each $i$, fix a maximal antichain $A_i$ below $p$
  of conditions rapidly reading~$\n T_i$. 
  $X:=\bigcup_{i\in \kappa, q \in A_i} \supp(q)$ has size $\kappa$, 
  so there
  is an $\alpha\in\Xicn\setminus X$.
  Each $a\in A_i$ rapidly reads $\n T_i$
  not using $\alpha$.
  So by the preceding lemma, 
  $\n N_\alpha\not\subseteq N_{\n T_i}$ is forced by $a$
  (and therefore by $p$, as $A_i$ is predense below $p$).
\end{proof}

\subsection{\texorpdfstring{$\non(\NULL)\le \kappa_\typenn$}{non(N) <= kappa-nn}}\label{ss:nnB}

We want to show that the set $X$ of reals reals that are added by (or more precisely:
rapidly read from) the $\typenm$ and $\typenn$ parts 
(i.e., not depending on the $\typecn$ and Sacks parts)
is not null.

Let $\QQQ_{\Xinonsk}$ be the set of 
conditions $p$ with $\supp(p)\cap \Xisk=\emptyset$.
Recall that according to Lemma~\ref{lem:quotient},
$\QQQ_{\Xinonsk}$ is a complete subforcing of $\QQQ$
(and satisfies $\omega^\omega$-bounding, rapid reading, etc).
We have seen in~\ref{sec:SacksB} that the quotient of $\QQQ$ and 
$\QQQ_{\Xinonsk}$ satisfies the Sacks property, and in particular that
every null set $N$ in the $Q$-extension is contained in a 
null-set $N'\supseteq N$ in the intermediate $\QQQ_{\Xinonsk}$-extension.

So it is enough to show that $X$ is still non-null in the $\QQQ_{\Xinonsk}$-extension;
in other words, we can in the rest of the paper ignore 
the Sacks indices altogether (i.e., work in $\QQQ_{\Xinonsk}$, or in other
words assume that $\Xisk=\emptyset$).

We have seen that the sets of the form $N_T$ for \prunedtree trees $T$ form a
basis of null sets; so we just have to show the following:

\begin{lem}\label{lem:khwtetew}
  Let $\n T^*$ be a \prunedtree tree 
  rapidly read by $p$. 
  Then there is a $q\leq p$ continuously reading some $\n r\in 2^\omega$
  not
  using the $\typecn$ part, such that 
  $q$ forces $\n r\in [\n T^*]$.
  (As described above, the Sacks part is not used at all.)
\end{lem}

As $\n r\in [\n T^*]$ implies $\n r\notin N_{\n T^*}$, and
$\n r$ only depends on the $\typenm$ and $\typenn$
parts, we get:

\begin{cor}\label{lem:nnleknn}
  $\mathbb Q$ forces $\non(\NULL)\le \kappa_\typenn$.
\end{cor}

To prove Lemma~\ref{lem:khwtetew} we will use:
\begin{lem}\label{lem:treejkjkljl}
  Let $T$ be a tree of positive measure and fix $\epsilon>0$.
  Then for all sufficiently large $m\in\omega$ 
  there are many fat nodes in $T\cap 2^m$, by which we mean:
  \[
    \mu([T^{[s]}])\ge 2^{-m}(1-\epsilon)\text{ for at least }
    |[T]\cap  2^m|\cdot (1-\epsilon)\text{ many }s\in T\cap 2^m.
  \]
\end{lem}

\begin{proof}
  Write $\mu$ for the measure of $[T]$.
  Note that $|T\cap 2^m|\cdot 2^{-m}$ decreases and converges to $\mu$.
  Hence from some $m$ on, we have 
  \begin{equation}\label{eq:mar1}
    |T\cap 2^m|\cdot 2^{-m}-\mu\epsilon^2\le \mu.
  \end{equation}
  Let $l$ be the number of fat nodes at level $m$, and $s=|T\cap 2^m|-l$
  the number of non-fat nodes. We want to show $l\ge 2^m\mu\cdot (1-\epsilon)$.

  Clearly,
  \begin{equation}\label{eq:mar2}
\mu <  l\cdot  2^{-m} + s \cdot  2^{-m}(1-\epsilon)=|T\cap 2^m|\cdot 2^{-m}-2^{-m}s\epsilon .
  \end{equation}
  
  Combining~\eqref{eq:mar1} and~\eqref{eq:mar2}, we get 
  $|T\cap 2^m|\cdot 2^{-m}-\mu\epsilon^2\le |T\cap 2^m|\cdot
  2^{-m}-2^{-m}s\epsilon $, and hence $s\le 2^m \mu\cdot\epsilon $.
  As $l+s=|T\cap 2^m|\ge 2^m \mu$, we get $l\ge 2^m\mu\cdot (1-\epsilon) $,
  as required.
\end{proof}

\begin{proof}[Proof of Lemma~\ref{lem:khwtetew}]
We can assume that $p$ is pruned. By induction on $n\in\omega$, we construct:
\begin{enumerate}[(a)]
  \item $k_n\in \omega$.
  \item\label{item:nor1a} A condition $q_n\le p$ with $k_n\in\ww^{q_n}$ such that 
    $\nor(q_n,k' )\ge n+6$ for all $k'\ge k_n$ in $\ww^{q_n}$.
  \item\label{item:nor1b} We will additionally require:
    $q_{n+1}\le q_n$; $q_{n+1}$ is identical to $q_n$ below $k_n$,
    and has norms $\ge n$ between $k_n$ and $k_{n+1}$.

    (Therefore there is a limit condition $q_\omega$ stronger than
    each $q_n$.)
  \item $i_n\in \omega$ and 
    a name $\n s_n$ for an element of $\n T^*\cap 2^{i_n}$ 
    such that $q_n$ decides $\n s_n$ below $k_n$ not using any 
    $\typecn$-indices. 
  \item\label{item:notoolarge}
    We additionally require that $i_n$ is ``not too large'' with respect to $k_n$,
    more particularly:
    \[
      2^{i_n+2}<b((k_n,0)).
    \]
    ($(k_n,0)$ is the  the smallest subatomic sublevel above $k_n$.)
    (As $b$ is strictly monotone, it suffices to have $k_n>2^{i_n+2}$.)
  \item
    We additionally require: $i_{n+1}>i_n$, and
    $\n s_{n+1}$ is forced (by $q_{n+1}$) to extend $\n s_n$.

    So $q_\infty$ will force that the 
    union of the $\n s_n$ will be the required branch
    through $\n T^*$, proving the Lemma.
  \item We will also construct a name $\n T_n$,
   which is (forced by $q_n$ to be) a subtree of $\n T^*$
   with stem $\n s_n$ 
   and relative measure $>\nicefrac12$ (i.e., $\mu ([\n T_n])> \nicefrac1{2} \cdot 2^{-i_n}$),
   which is read continuously by $q_n$ not using any $\typecn$-indices below
   $k_n$.\footnote{I.e.: For all $\ell$ there is a $k$ and a function defined on
   $\poss(q_n,{<}k)$ giving the value of $\n T_n\cap 2^\ell$ such that the value is the 
   same for $\eta,\eta'\in\poss(q_n,{<}k)$ that differ only on the 
   $\typecn$-part below
   $k_n$.} 
\end{enumerate}

We set $i_0:=0$, $\n s_0:=\langle\rangle$ and $\n T_0=\n T^*$.  We choose $k_0$
such that the norms of the compound creatures in $p$ are $\ge 6$
above $k_0$ and set $q_0$ to be $p$ where
we increase the trunk to $k_0$. So $\n T_0$ does not depend on any
$\typecn$-indices below $k_0$ (as below $k_0$ there is only trunk and
thus a unique possibility).

So assume we already have the objects mentioned above for some $n$
(i.e., $k_n$, $q_n$, $i_n$, $\n s_n$ and $\n T_n$).
For notational simplicity we refer to them without the subscript $n$, 
i.e., we set $k:=k_n$ etc.   We will
now construct the objects for $n+1$.

\begin{enumerate}
   \item  We choose $k^*$ so large that for each $\xi\in \supp(q(k))\cap \Xils$
     there is an atom $q_n(\xi,\ell)$ 
     of norm $>n+2$ for some $\ell$ between $k$ and $k^*$.
   \item It is forced that Lemma~\ref{lem:treejkjkljl} holds
     for $\n T$ and for
 $\epsilon:=\nicefrac1{\maxposs({<}k^*)\cdot\maxposs({<}k)}$. So we get
     a name $\n m$ for a level where there are many 
     fat nodes.
     Using Lemma~\ref{lem:blablbagh}, we strengthen $q$ to $q^1$,
     not changing anything below $k^*$ and keeping all norms $\ge n+4$,
     such that we can find (in $V$) some $m>i$ which is forced by
     $q^1$ to be $\ge\n m$. Note that Lemma~\ref{lem:treejkjkljl}
     is forced to hold
     for this $m\ge \n m$ as well, i.e., there is a name of 
     a ``large'' set $\n L \subseteq 2^m$ of ``fat'' nodes.

     This $m$ will be our $i_{n+1}$. So $i_{n+1}>i_n$ is satisfied.
  \item So can further strengthen $q^1$ to $q^2$ not changing anything below $k^*$
     and keeping all norms $\ge n+2$
     such that $\n L\subseteq 2^m$ is essentially decided,
     i.e., decided below some level $k^{**}> k^*$. 
     Since we already assumed that $\n T$ 
     is read continuously, we can assume that 
     $q^2$ also decides $\n T\cap 2^m$ below $k^{**}$. 
     Also, we can assume that 
     all norms of compound creatures 
     in $q^2$ above (including)  $k^{**}$ are $>n+7$, and that 
     $k^{**}>2^{m+2}$.

     This $k^{**}$ will be $k_{n+1}$. Note that this ensures 
     item~(\ref{item:notoolarge}) for $n+1$.
  \item
     $\n L$ is forced to be a subset of $\n T\cap 2^m$ of 
     relative size $\ge (1-\epsilon)$, and both 
     $\n L$  and $\n T\cap 2^m$ are decided below $k^{**}$.
     Also, $\n T\cap 2^m$ does not depend on the $\typecn$-part below $k$.
     Therefore, we can construct a  name $\n L'\subseteq \n L$
     that also does not depend on such coordinates, and such that 
     $\n L'\subseteq \n T\cap 2^m$ has relative size 
     $\ge (1- \epsilon \cdot\maxposs({<}k)) \ge \nicefrac12$.
\begin{quote}
     Proof: Each $\eta\in\poss(q_2,k^{**})$ determines objects
     $L_\eta\subseteq S_\eta$ (where $q^2\wedge\eta$
     forces ``$L_\eta=\n L$ and $S_\eta=\n T\cap 2^m$'').
     We call $\eta_1,\eta_2$ equivalent if they differ 
     only on the $\typecn$-part below $k$ (which implies $S_{\eta_1}=
     S_{\eta_2}$).
     Clearly, each equivalence class has size at most $\maxposs({<}k)$.
     For an equivalence class $[\eta]$, we set $L'_{[\eta]}:=\bigcap_{\eta'\in[\eta]}L_{\eta'}$.
     So the map assigning $\eta$ to $L'_{[\eta]}$ defines a name
     (not depending on the $\typecn$-part below $k$)
     of a subset of $S_\eta$ of relative size
     $\ge \nicefrac12$. 
\end{quote}

     Recall  that  $\n T$ is forced to have 
     stem $s\in 2^i$ and measure $> \nicefrac1{2}\cdot 2^{-i}$, so
     the cardinality of $\n T\cap 2^m$ is forced to be
     $> 2^{m-i-1}$, and thus the cardinality of
     $\n L'$ is 
     forced to be $>2^{m-i-1}(\nicefrac12)= 2^{m-i-2}>\nicefrac{2^m}{b((k,0))}$,
     according to item~(\ref{item:notoolarge}).
     
     To summarize:
     \begin{itemize}
     \item
     $\n T\cap 2^m$ and its subset $\n L'$
     are decided by $q^2$ below $k^{**}$, not using the $\typecn$-part below
     $k$.
     \item We set $\Omega=2^m$.
     (As a finite set, it carries the uniform probability measure.)
     $\n L'$ as subset of $\Omega$ is forced to have measure $>\nicefrac1{b((k,0))}$.
     \item
     $q^2$ forces that each $s\in \n L'$ satisfies $\mu([\n T^{[s]}])\ge 2^{-m}(1-\epsilon)$.
   \end{itemize}
  \item Now we glue $q^2$ between $k$ and $k^{**}$, 
     and replace  
     all 
     lim-sup subatoms between $k^*$ and $k^{**}$ with singletons
     (not changing the lim-inf subatoms, nor anything between $k$ and $k^*$),
     resulting in $q^*$ and the 
     compound creature $\cd^*=q^*(k)$ (with
     $\mdn(\cd^*)=k$, $\mup(\cd^*)=k^{**}$ and $\supp(\cd^*)=\supp(q,k)$).
     So above $k^{**}$, $q^*$ is identical to $q^2$, and 
     below $k^*$ it is identical to $q$. 

     Note that $\nor(\cd^*)\ge n+2$:
     Gluing results in a norm
     at least the minimum of the norms of the glued creatures; and 
     replacing lim-sup subatoms above $k^*$ with singletons does not drop the
     norm below $n+2$ as we made sure that there are large
     subatoms between $k$ and $k^*$.

     We will in the following find a strengthening $\cd^{**}$ of $\cd^*$ with
     $\nor(\cd^{**})\ge\nor(\cd^*)-2\ge n$
     and we will set $q_{n+1}$ to
     be $q^*$ where we replace $\cd^*$ with $\cd^{**}$. 
     Then items~(\ref{item:nor1a}) and~(\ref{item:nor1b})
     will be satisfied for $n+1$.
  \item Recall that  $q^*$ decides both $\n L'$ and $\n T\cap 2^m$ below $k^{**}$,
     not using the $\typecn$-part below $k$.
     Note that $\poss(q^*,{<}k^{**})$
     is isomorphic to $X\times Y\times Z$, for 
     \begin{itemize}
       \item $X:=\poss(q^*,{<}k)=\poss(q,{<}k)$,
       \item $Y$ are the possibilities of $\cd^*$ between $k$ and $k^*$, and
       \item $Z$ are the possibilities of $\cd^*$ between $k^*$ and $k^{**}$
     (which we can restrict to the lim-inf part, as there are only singletons in the lim-sup-part).
     \end{itemize}
  \item\label{item:induction} Fix a $\nu\in Z$.
    We will now perform an induction on the (subatomic) sublevels $\slu$ between $k$ and $k^*$,
    starting with the lowest one, $(k,0)$.
    We assume that we have arrived in this construction
    at sublevel $\slu$ with the
    active subatom $C$, and that we
    already have constructed the following:
    \begin{itemize}
      \item The (final) subatoms for all sublevels $\slv$ below $\slu$ (and above $k$), 
        with subatom-norm at most $2$
        smaller than the norms of the original subatoms (i.e., those in $\cd^*$).
      \item (Preliminary) subatoms for all sublevels $\slu'$ above (including) $\slu$
        (and below $k^*$),
        where the norm of the subatom at $\slu'$ 
        has been reduced from the original one by at most 
        $\nicefrac{K}{b(\slu')}$,
        where $K$ is the number of steps already performed in the 
        current induction (i.e., $K$ is the number of subatomic sublevels 
        between $k$ and $\slu$). So our current $C$ is one of these ``preliminary subatoms''.
      \item A function $F^{\slu}$ that maps each possibility $\eta\in X\times Y$ to
        a subsets $F^{\slu}(\eta)$ of $2^m$; such that for all $\eta$
        \begin{itemize}
          \item 
            $F^{\slu}(\eta)$ is forced to be a subset of $\n L'$ by the condition
            $q^*$ modulo the fixed $\nu\in Z$, modulo $\eta$
            and modulo the already 
            constructed subatoms (the final ones as well as the preliminary ones).\footnote{See~\eqref{eq:bla2342} for a definition of ``modulo''. 
If $\eta$ is not a compatible with the currently constructed
(final and preliminary) subatoms, then $F^{\slu}(\eta)$ is irrelevant.}
          \item $F^{\slu}(\eta)\subseteq 2^m$ is of relative size $\ge \nicefrac{1}{b(\slu)}$.
          \item $F^{\slu}(\eta)$ does not depend on any $\typecn$-indices 
           below $\slu$.
        \end{itemize}
    \end{itemize}
    The first sublevel, $(k,0)$, is clear: there are no sublevels below 
    where we have to define final subatoms,
    the preliminary subatoms above are just the original ones, and $F^{(k,0)}$
    is just given by the name $\n L'$.

    Now we perform the inductive step.
    If our subatom $C$ is not of $\typecn$-type, we do nothing\footnote{slightly more formally: 
we make the current preliminary subatom final, and set $F^{\slu+1}:=F^{\slu}$}
    and go to the next step.
    So let us assume that the current (preliminary) $C$ is of $\typecn$-type.

    Let $Y^-$ be $Y$ restricted to the sublevels below $\slu$,
    and $Y^+$ to the ones above.
    Every\footnote{We are concerned only about the $\eta$ still are
    compatible with the currently constructed preliminary/final subatoms.}
    $\eta\in X\times Y$ can be written as $(\eta^-,\eta^\slu,\eta^+)$
    for $\eta^-\in X\times Y^-$, $\eta^\slu\in\poss(C)$ and $\eta^+\in Y^+$.
    
    When we fix some $\eta^-\in X\times Y^-$ and $\eta^+\in Y^+$, 
    the function $F^{\slu}$ reduces to a function $F^{\eta^-,\eta^+}$ that maps 
    $\poss(C)$ to subsets of $2^m$ of relative size $\ge \nicefrac{1}{b(\slu)}$.
    So we can use Lemma~\ref{lem:lemmaA}(2) and strengthen $C$ to $D(\eta^-,\eta^+)$ 
    decreasing the norm by at most $\nicefrac1{b(\slu)}$
    such that
    \[
    F'(\eta^-,\eta^+):=\bigcap_{\mu\in\poss(D(\eta^-,\eta^+))} F^{\eta^-,\eta^+}(\mu)\]
is a set of measure  $\ge\nicefrac{1}{b(\slu+1)}$.

    For fixed $\eta^+\in Y^+$, we can iterate this strengthening
    for all
    $\eta^-\in X\times Y^-$: From $D$ to some $\tilde D:=D(\eta^-,\eta^+)$, then 
    from $\tilde D$ to $D(\eta^{\prime -},\eta^+)$ for the next $\eta^{\prime -}$, etc.,
    resulting in a $D(\eta^+)$ with norm reduced by at most 
    $\nicefrac{\maxposs({<}\slu)}{b(\slu)}<1$.

    Note that there are less than $b(\slu+1)$
    many possibilities for $D(\eta^+)$, cf~\eqref{eq:blghgh}.
    Finally we can use bigness of the $Y^+$-part, as stated in
    Lemma~\ref{lem:schoeneslemma}, to find successor subatoms at all sublevels
    above $\slu$, resulting in a new set of possibilities 
    $\tilde Y^+\subseteq Y^+$
    such that for each $\eta^+\in\tilde Y^+$ 
    we get the same $D:=D(\eta^+)$. 
    This $D$ will be the (final) subatom at our current level $\slu$.

    We can now define 
    \[
    F^{\slu+1}(\eta):=\bigcap_{\mu\in\poss(D)} F^{\slu}(\eta^-,\mu,\eta^+).
    \]
    As above, this is a set of measure 
    $\ge\nicefrac{1}{b(\slu+1)}$, does not depend on the $\typecn$-part $\le \slu$,
    and it is forced (modulo $D$) to be 
    a subset of $\n L'$.

    We have now chosen the new final subatom $D$, the new preliminary subatoms and 
    $F^{\slu+1}$ in a way that we can perform the next step of the iteration.
  \item We perform the whole inductive construction of~(\ref{item:induction}) 
    for every $\nu\in Z$ independently
    (i.e., we start at the original $\cd^*$ for each $\nu\in Z$).

    So for every $\nu$ we get a different sequence $\bar D(\nu)$ of 
    subatoms between $k$ and $k^*$.
    Using bigness (again as in Lemma~\ref{lem:schoeneslemma}), 
    we can thin out the subatoms between $k^*$
    and $k^{**}$,
    resulting in $Z'\subseteq Z$,  such that
    for each $\nu \in Z'$ we get the same sequence $\bar D(\nu)=:\bar D$ which
    finally defines the compound creature $\cd^{**}$ stronger than $\cd^*$.

    We set $q_{n+1}$ to be $q^*$ with $\cd^*$ strengthened to $\cd^{**}$, and
    we set $i_{n+1}:=m$ and $k_{n+1}:=k^{**}$.
  \item
    Now work modulo $q_{n+1}$.
    So the final function $F$ of the induction in~(\ref{item:induction}) gives
    us a name for a subset $\n L''\subseteq \n L\subseteq 2^m$ of positive relative size 
    (in $2^m$), and the name $\n L''$ does 
    not depend on any $\typecn$ indices:
    Not on any below $k$, since we started with the name $\n L'$ which did not depend 
    on such subatoms;
    not on any
    between $k$ and $k^*$, as we removed this dependence sublevel by sublevel during the 
    induction; and not on any $\typecn$ subatoms
    between $k^*$ and $k^{**}$, as $\typecn$ indices are of lim-sup type, and we have
    only singleton subatoms for the lim-sup part between $k^*$ and $k^{**}$.

    So we can pick a non-$\typecn$-name $s_{n+1}$ for an arbitrary (the leftmost, say) element 
    of $\n L$.
  \item 
    $q_{n+1}$ forces that $s_{n+1}$ is in $\n L$, i.e., a ``fat'' node, more specifically:
    $\n T':= \n T_n^{[s_{n+1}]}$ has a measure greater than $\frac{1-\epsilon}{2^{m}}$.

    The tree $\n T'$ is read continuously by $q_n$ and therefore also by $q_{n+1}$.
    In particular, for each $\ell>m$ the finite tree
    $\n T'\cap 2^\ell$ is decided below some $\ell'$.
    For $\eta\in\poss(q_{n+1},{<}\ell')$ let $T^{\ell,\eta}$ be the
    according value of $\n T'\cap 2^\ell$ (a subset of $2^\ell$ with
    at least $2^\ell \cdot  \frac{1-\epsilon}{2^{m}}$ elements).
    We call $\eta$ and $\eta'$ equivalent if they differ only
    on the $\typecn$ part below $k^{**}$.
Each  equivalence class has size $\le\maxposs({<}k^{*})$, as there are only singleton 
values in the lim-sup part between $k^*$ and $k^{**}$. We assign to each equivalence
class $[\eta]$ the tree
 $T^{\ell,[\eta]}:=\bigcap_{\eta'\in[\eta]}T^{\ell,\eta'}$.
Then $T^{\ell,[\eta]}$ has size at least $2^\ell\cdot \frac{1-\maxposs({<}k^{*})\cdot \epsilon}{2^{m}}$ (and of course does not depend on the $\typecn$-part below $k^{**}$). 
So the family $T^{\ell,[\eta]}$ defines a continuous name for a tree $\n T_{n+1}$ not depending 
on the $\typecn$-part below $k^{**}$ with root $s_{n+1}$ and measure
$>\nicefrac{1}{2^{m+1}}$, as required.
\qedhere
\end{enumerate}
\end{proof}

\section{Switching $\typenm$ and $\typenm$}\label{sec:switching}
It turns out that the same proof can be used for the following variant of
Theorem~\ref{thm:main}, where the order of $\kappa_\typenm$ and $\kappa_\typenn$
is reversed:
\begin{thm}\label{thm:switching}
   Assume (in $V$) CH, 
   $\kappa_\typenn \le \kappa_\typenm\le\kappa_\typecn\le\kappa_\typesk$
   and $\kappa_t^{\aleph_0}=\kappa_t$ for $t\in\{\typenm,\typenn,\typecn,\typesk\}$.
   Then there is a forcing  $\mathbb  Q$ which   forces
   \begin{enumerate}
     \item $\cov(\NULL)=\mathfrak d=\aleph_1$,
     \item $\non(\NULL)=\kappa_\typenn$,
     \item $\non(\MEAGER)=\cof(\MEAGER)=\kappa_\typenm$,
     \item $\cof(\NULL)=\kappa_\typecn$,
     \item $2^{\aleph_0}=\kappa_\typesk$.
   \end{enumerate}
   Moreover, $\mathbb Q$ preserves all cardinals and all cofinalities. 
\end{thm}

\begin{proof}
We now use the $\typecn$-norm for the $\typenm$ part as well.
(Recall~\ref{nb:othernorm}: We can use any $\typenm$-norm, as long as
bigness is satisfied.)
The proofs above do not change, apart the one of
$\non(\NULL)\le \kappa_\typenn$:
In the inductive construction, we only had to do something at the 
$\typecn$-indices, and we could ignore the  $\typenm$-indices (as there were only few).
In the new version,  we have to include the $\typenm$-indices as well.
But this is no problem: We now do exactly the same at $\typenm$-indices  as at
$\typecn$-indices (which we can, as the $\typenm$-norm is the same as the $\typecn$-norm).
\end{proof}

\bibliographystyle{alpha}     
\bibliography{1044}

\end{document}